\documentclass{article}[12pt]

\usepackage[T2A]{fontenc}
\usepackage[utf8]{inputenc}
\usepackage{hyphenat}

\usepackage{amsmath}
\usepackage{amssymb}
\usepackage{amsfonts}
\usepackage{amscd,euscript,amsxtra,amsthm}

\newtheorem{theorem}{Theorem}

\newtheorem{corollary}{Corollary}
\newtheorem{proposition}{Proposition}
\newtheorem{lemma}{Lemma}[section]

\usepackage{amscd,euscript,amsxtra,amsthm}

\theoremstyle{definition}
\newtheorem{definition}{Definition}
\newtheorem{convention}{Convention}

\theoremstyle{remark}
\newtheorem{remark}{Remark}
\theoremstyle{remark}

\theoremstyle{remark}

\usepackage[all]{xy}

\numberwithin{equation}{section}

\def\A{{{\mathbb A}}}
\def\I{{{\mathbb I}}}

\def\G{{{\mathbb G }}}
\def\E{{{\mathbb E }}}
\def\D{{{\mathbb D }}}
\def\C{{{\mathbb C }}}

\def\P{{{\mathbb P }}}

\def\L{{{\mathbb L }}}

\def\U{{{\mathbb U}}}
\def\V{{{\mathbb V }}}

\def\CC{{{\mathcal C}}}

\def\EE{{{\mathcal E}}}
\def\FF{{{\mathcal F}}}

\def\HH{{{\mathcal H}}}
\def\II{{{\mathcal I}}}

\def\OO{{{\mathcal O}}}
\def\QQ{{{\mathcal Q}}}

\def\TT{{{\mathcal T}}}

\def\EExt{{{\mathcal E}xt}}
\def\TTor{{{\mathcal T}or}}
\def\FFitt{{{\mathcal F}itt}}

\def\Spec{{{\rm Spec \,}}}
\def\Sing{{{\rm Sing \,}}}

\def\Proj{{{\rm Proj \,}}}

\def\Grass{{{\rm Grass }}}
\def\Quot{{{\rm Quot \,}}}
\def\Supp{{{\rm Supp \,}}}
\def\Pic{{{\rm Pic \,}}}

\def\dim{{{\rm dim \,}}}
\def\codim{{{\rm codim \,}}}
\def\hd{{{\rm hd \,}}}
\def\coker{{{\rm coker \,}}}
\def\ker{{{\rm ker \,}}}
\def\rank{{{\rm rank \,}}}
\def\im{{{\rm im \,}}}
\def\id{{{\rm id }}}

\def\tors{{{\rm tors \,}}}

\def\red{{{\rm red}}}

\def\gr{{{\rm gr}}}
\title
{Moduli of Admissible Pairs for Arbitrary Dimension, II: Functors and Moduli}

\author{{Nadezhda V. Timofeeva}\footnote{Center of Integrable Systems, P.G. Demidov Yaroslavl State University, Russia}}
\date{}

\begin{document}
\large
\maketitle

\renewcommand{\refname}{References}
\renewcommand{\proofname}{Proof.}
\thispagestyle{empty}

\maketitle{}\vspace{1cm} 
\begin{quote}
\noindent{\sc Abstract. } Admissible pairs $((\widetilde S, \widetilde
L), \widetilde E)$ consisting of an $N$-dimensional projective scheme~$\widetilde S$ of certain class with a special ample invertible sheaf $\widetilde L$ and a locally free $\OO_{\widetilde S}$-sheaf $\widetilde E$ are considered. An admissible pair can be produced in the procedure of a transformation (which is called a resolution) of a torsion-free coherent sheaf $E$ on a nonsingular  $N$-dimensional projective algebraic variety $S$ to a locally free sheaf $\widetilde E$ on some projective scheme  $\widetilde S$. Notions of stability and semistability of admissible pairs and a moduli functor for semistable admissible pairs are introduced. Also the relation of the stability and semistability for admissible pairs to the classical  stability and semistability for coherent sheaves under the resolution is examined.  Morphisms between the moduli functor  of
admissible semistable pairs   and  the  Gieseker--Maruyama moduli functor (of semistable coherent torsion-free sheaves) with the same Hilbert polynomial on a nonsingular $N$-dimens\-ional projective algebraic variety are  constructed. Examining these morphisms
it is shown that the moduli
scheme for semistable admissible pairs $((\widetilde S, \widetilde
L), \widetilde E)$ is isomorphic to the Gieseker--Maruyama
moduli scheme for coherent sheaves. The considerations involve all the existing components of these moduli
functors and of their corresponding moduli schemes.

Bibliography: 25 items.
\medskip

\noindent{\bf Keywords:} moduli space, semistable coherent
sheaves, semi\-stable admissible pairs, vector bundles, compactification of moduli,
$N$-dimensional alge\-braic variety.
 \end{quote}

\bigskip

\section{Introduction}

The present paper is a continuation of \cite{Tim12}, where a resolution of a flat family of torsion-free coherent sheaves on a nonsingular projective variety  was developed. This resolution procedure is called in \cite{Tim12} the standard resolution.

We construct the scheme of moduli for
admissible pairs (defined below) of arbitrary dimension. In the series
of articles \cite{Tim10}, \cite{Tim0}, \cite{Tim9}---\cite{Tim3} the construction was done in the case when the initial variety and admissible schemes had dimension two. Admissible schemes were obtained
from a nonsingular projective algebraic surface
$S$ by some transform\-ation. The final result was an isomorphism between the moduli functor of
Gieseker-semistable torsion-free coherent sheaves of rank $r$ and with
Hilbert polynomial $rp(n)$ on the surface $S$ with fixed
polarization $L$ and the moduli functor of semistable admissible pairs
$((\widetilde S, \widetilde L), \widetilde E)$. Each pair consists
of an admissible scheme $\widetilde S$ with a distinguished polarization
$\widetilde L$ and of a semistable locally free sheaf $\widetilde E$ of
rank $r$ and with Hilbert polynomial $rp(n)$. The advantage of this
 moduli functor is that its moduli scheme is isomorphic to
the Gieseker--Maruyama moduli scheme for the  same rank and  Hilbert
polynomial. In particular, this can be interpreted as a method to produce
a compactification of the moduli space of stable vector bundles by vector
bundles on some special (admissible) schemes instead of the classical
compactification by attaching non-locally free coherent sheaves.

The reason for the interest to this subject is the Kobayashi--Hitchin
correspondence to be discussed below. It validates application of algebraic geometrical methods to
problems of differential geometrical or gauge theoretical setting by
replacing consideration of the moduli of connections in the vector
bundle (including vector bundles supplied with additional structures)
by consideration of the moduli for slope-stable vector bundles.

 Consider a compact complex algebraic variety $X$,  a complex vector
 bundle $E$ on $X$, Hermitian metrics $g$ on $X$ and $h$ on $E$
 respectively. A holomorphic vector bundle on a compact complex manifold
 with a Hermitian metric $g$ is {\it stable}
\cite{LuTel} if and only if $E$ admits an irreducible $g$-Hermitian--Einstein metric. If $E$ is a differentiable vector bundle and  $\bar
\delta$ is a holomorphic structure in $E$ then  $\bar \delta$ is called
 $g$-{\it stable} \cite{LuTel}  if the holomorphic bundle $(E, \bar \delta)$
 has this property. Let $M^{st}_g(E)$ be the  moduli space for isomorphy classes of $g$-stable
 holomorphic structures on $E$ and $M^{HE}_g(E,h)$ be the moduli space for
 gauge equivalence classes of irreducible $h$-unitary $g$-Hermitian--Einstein connections in~$E$.
Many mathematicians  (S.~Donaldson, S.~Kobayashi,
N.~Hitchin, F.~Bogomolov, N.~Buchdahl, M.~Itoh and H.~Nakajima, 
J.~Li, K.~Uhlenbeck and S.-T.~Yau and others) built up the theory in various aspects; its central result can be formulated as follows \cite{LuTel} (the Kobayashi--Hitchin correspondence):\\
{\it There exists a real-analytical isomorphism} $M^{HE}_g(E,h) \cong
M^{st}_g(E)$.

 If $X$ is a projective manifold and the Hodge metric  $g$ is associated
 to the ample invertible sheaf~$L$ which induces the immersion of $X$
 into the projective space, then \cite{LuTel} the definition of
$g$-stability coincides with the definition of $L$-slope-stability
of a vector bundle~$E$.

In the case of rank two vector bundles on a surface J. Li proved
\cite{Li93} that the Donaldson--Uhlenbeck compactification for the moduli
space  (of gauge equivalence classes) of anti-self-dual connections
admits  a complex structure  providing it with a reduced
projective scheme structure. The Gieseker--Maruyama
compactification for the moduli scheme  for stable vector bundles in this case 
has a morphism to the compactified scheme of anti-self-dual connections.

Usually \cite{Mar1,Mar} moduli schemes for stable vector bundles on an algebraic variety
 are nonprojective (and noncompact), and
for application of algebraic geometrical techniques it is useful to
include the scheme (or variety) of moduli for vector bundles into
some appropriate projective scheme as its open subscheme. This
problem is called a problem of compactification of
a moduli space. There arises a natural question of looking for
compactificat\-ions of moduli for vector bundles and of moduli for
connections which admit expanding of the Kobayashi--Hitchin
correspondence to both compactifications in whole.

The classical Gieseker--Maruyama moduli space for coherent sheaves
is not suitable for this purpose, since it includes classes of co\-herent sheaves  which are not locally free. Besides it the following compactifications
(related to the Yang--Mills field theory) for the moduli scheme of vector bundles and for the moduli space of connections are known: the Donaldson--Uhlenbeck compactification \cite{Donaldson89} 
(involving so-called ideal connections), the algebraic geometrical functorial approach to its
construction given by V.~Baranovsky \cite{Bar}, the algebraic geometrical
version for framed sheaves done by U.~Bruzzo, D.~Markushevich and
A.~Tikhomirov in \cite{BruMarTikh}, and the Taubes--Uhlenbeck--Feehan compactification \cite{Feehan}. Also there is an algebraic geometrical bubble-tree compactification announced for
vector bundles on  a surface and constructed for rank 2 vector
bundles by D.~Markushevich, A.~Tikhomirov and G.~Trautmann in 2012 \cite{MTT}. A different approach is developed by the author of the 
present paper for the 2-dimensional case (see \cite{Tim10} and references
therein). Here we describe a construction of its
multidimensional analog.

Compactifications which do not involve non-bundles can be of use for
extend\-ing the Kobayashi--Hitchin cor\-respond\-ence to a compact case.
Since the Kobayashi--Hitchin correspondence holds not only for
dimension two, we extend the ``locally free'' compactification to the case of
arbitrary dimension. The Kobayashi--Hitchin correspondence operates
with the coefficient field $\C$. We  work with an arbitr\-ary
algebraically closed field~$k$ of zero characteristic, since the
isomorphism of functors is not subject to the special properties of
$\C$.

Gieseker--Maruyama moduli schemes for sheaves on an arbitrary nonsingular algebraic variety  are projective. They can consist of several connected components, which can have non-equal dimensions and can carry nonreduced scheme structures. Also some components can contain no locally free sheaves. These phenomena (reducibility, nonreducedness, non-equidimensionality and presence of non-bundle components) occur for some combinations of numerical invariants even in the  case of dimension two.

Let $S$ be a nonsingular projective algebraic variety over an algebraic\-ally closed field $k=\overline k$ of zero characteristic, $\OO_S$ its structure sheaf, $E$ a coherent torsion-free $\OO_S$-module, $E^{\vee}:={{\mathcal H}om}_{{\mathcal O}_S}(E, {\mathcal O}_S)$ its dual ${\mathcal O}_S$-module.  A locally free sheaf and its corresponding vector bundle are canonically identified, and both terms are used as synonyms. Let  $L$ be a very ample invertible sheaf on $S$; it is fixed and used as a polarization. The symbol $\chi(\cdot)$ denotes the Euler--Poincar\'{e} characteristic of a coherent sheaf, $c_i(\cdot)$ is its $i$-th Chern class.

Recall the definition of a sheaf of 0-th Fitting ideals known from
commutat\-ive algebra. Let $X$ be a scheme and $F$ be an $\OO_X$-module with a
finite presentation $$\widehat F_1 \stackrel{\varphi}{\longrightarrow} \widehat F_0
\to F\to 0.$$ Without loss of generality we assume that  $\rank \widehat F_1
\ge \rank \widehat F_0$.
\begin{definition} {\it The sheaf of 0-th Fitting ideals} of the
$\OO_X$-module $F$ is defined as  $$\FFitt^0 F =\im
(\bigwedge^{\rank \widehat F_0} \widehat F_1 \otimes \bigwedge^{ \rank  \widehat F_0}
\widehat F_0^{\vee} \stackrel{\varphi'}{\longrightarrow}\OO_X),$$ where
$\varphi'$ is the morphism of $\OO_X$-modules induced by~$\varphi$.
\end{definition}

\begin{remark} We work with invertible sheaves $L$ and
$\widetilde L$, where the expression for $\widetilde L$ involves
an appropriate tensor power $L^m$ of $L$. If necessary, in further considerations we
replace $L$ by its tensor power $L^m$. Here $m$ is big enough for
$\widetilde L$ to be very ample. This power can be chosen uniform
and fixed, as shown in \cite{Tim4} for the case when the homological dimension of the sheaf $E$ equals one ($\hd E=1$) and in \cite{Tim12} for arbitrary homological dimension.  All Hilbert polynomials are
computed with respect to new $L$ and $\widetilde L$ respectively.
\end{remark}
%%%%%%%%%%%%%%%%%%%%%%%%%%%%%%%%%%%%%%%%%%%%%%%%%%%%%%%%%%%%%%%%%%%%%%%%%
%%%%%%%%%%%%%%%%%%%%%%%%%%%%%%%%%%%%%%%%%%%%%%%%%%%%%%%%%%%%%%%%%%%%%%%%%

First we recall the definition of an $S$-stable ($S$-semistable) pair for
the case of dimension two. The analogous notions and the whole construction for an arbitrary dimension are natural generalizations of this definition and are described below. An essential assumption is that $S$ is nonsingular.

\begin{definition}[\cite{Tim4}] \label{sst} An $S$-{\it stable}
(respectively, {\it $S$-semistable}) {\it pair} $((\widetilde
S,\widetilde L), \widetilde E)$ in the case when $\dim S=2$ is the following data:
\begin{itemize}
\item{$\widetilde S=\bigcup_{i\ge 0} \widetilde S_i$ is an
admissible scheme, $\sigma \colon \widetilde S \to S$ is a morphism which is
called {\it canonical}, $\sigma_i \colon \widetilde S_i \to S$, $i\ge 0$,  
are the restrictions of the morphism $\sigma$ to the components $\widetilde S_i$, $i\ge 0;$}
\item{$\widetilde E$ is a vector bundle on the scheme
$\widetilde S$;}
\item{$\widetilde L \in \Pic \widetilde S$ is a distinguished polarization;}
\end{itemize}
such that
\begin{itemize}
\item{$\chi (\widetilde E \otimes \widetilde
L^n)=rp(n),$ where the polynomial $p(n)$ and the rank $r$ of the sheaf
$\widetilde E$ are fixed;}
\item{the sheaf $\widetilde E$ on the scheme $\widetilde S$ is {\it
stable} (respectively, {\it semistable}) {\it in the sense of Gieseker},}

\item{on each of the additional components $\widetilde S_i,$ $i>0,$
the sheaf $\widetilde E_i:=\widetilde E|_{\widetilde S_i}$ is {\it
quasi-ideal,} i.e. it admits a description of the form
\begin{equation}\label{quasiideal}\widetilde E_i=\sigma_i^{\ast}
\ker q_0/\tors_i\end{equation} for some $q_0\in \bigsqcup_{l\le
c_2} {\rm Quot\,}^l \bigoplus^r {\mathcal O}_S$. }\end{itemize}
\end{definition}
The definition of the subsheaf $\tors_i$ for the case $\hd E=1$ is given below. A generalization of this definition for the case of arbitrary homological dimension is given in~\cite{Tim12}. 
The coefficients of the Hilbert polynomial $rp(n)$ of a sheaf $E$ depend on its Chern classes. In particular, $c_2(\cdot)$ is the $2^{\mathrm{nd}}$ Chern class of a sheaf with Hilbert polynomial equal to~$rp(n).$

Pairs $(( \widetilde S, \widetilde L), \widetilde E)$ such that
 $(\widetilde S, \widetilde
L)\cong (S,L)$ are called {\it $S$-pairs}.

In the series of articles of the author \cite{Tim0,Tim1,Tim2,Tim4,Tim3} a projective algebraic scheme~$\widetilde M$ is built up as a reduced moduli scheme of
$S$-semistable admissible pairs. In \cite{Tim6} a corresponding possibly nonreduced 
moduli scheme is constructed. The latter scheme was also denoted as $\widetilde M$.
From now on $\widetilde M$ is understood as a possibly nonreduced moduli scheme for semistable admissible pairs.

The scheme  $\widetilde M$ contains an open subscheme  $\widetilde
M_0$ which is isomorphic to the subscheme $M_0$ of
Gieseker-semistable vector bundles in the Gieseker--Maruyama
moduli scheme $\overline M$ of torsion-free semistable sheaves
whose Hilbert polynomial is equal to  $\chi(E \otimes L^n)=rp(n)$.

Let  $E$ be a semistable locally free sheaf. Then the
sheaf  $I={{\mathcal F}itt}^0 {{\mathcal E}xt}^1(E, {\mathcal
O}_S)$ is trivial and $\widetilde S \cong S$. In this case
$((\widetilde S, \widetilde L), \widetilde E) \cong ((S, L),E)$,
and we have a bijective correspondence $\widetilde M_0 \cong M_0$. Moreover, it is a scheme isomorphism.

Let $E$ be a semistable coherent sheaf which is not locally free. Then the
scheme  $\widetilde S$ contains a reduced irreducible component
$\widetilde S_0$ such that the morphism
$\sigma_0:=\sigma|_{\widetilde S_0}\colon \widetilde S_0 \to S$ is a
morphism of blowing up of the scheme $S$ in the sheaf of ideals
$I= {{\mathcal F}itt}^0 {{\mathcal E}xt}^1(E, {\mathcal O}_S).$
Formation of the sheaf~$I$ is a way to  characterize the singularities of the sheaf~$E$, i.e. its difference from a
locally free sheaf. Indeed, the quotient sheaf $\varkappa:=
E^{\vee \vee}/ E$ is Artinian of length not greater than
$c_2=c_2(E)$, and ${{\mathcal E}xt}^1(E, {\mathcal O}_S) \cong
{{\mathcal E}xt}^2(\varkappa, {\mathcal O}_S).$ Let $Z$ be the subscheme corresponding to the sheaf of ideals  ${{\mathcal F}itt}^0 {{\mathcal E}xt}^2(\varkappa, {\mathcal O}_S)$. In general, $Z$ is nonreduced and has  bounded length \cite{Tim6}. The subscheme $Z$ is supported at a finite set of points on the surface~$S$. As shown in \cite{Tim3}, in general, the other 
components $\widetilde S_i,$ $i>0$, of the scheme  $\widetilde S$ 
 can carry nonreduced scheme structures.

Each semistable coherent torsion-free sheaf $E$ corresponds to a
pair  $((\widetilde S, \widetilde L), \widetilde E)$ where
$(\widetilde S, \widetilde L)$ is defined as described above.

Now we turn to the construction of the subsheaf~$\tors$ in (\ref{quasiideal}). Let~$U$ be a Zariski-open subset in one of the components  $\widetilde S_i,$ $i\ge 0$, and $\sigma^{\ast}E|_{\widetilde S_i}(U)$ be the correspond\-ing group of sections. This group is an ${\mathcal O}_{\widetilde S_i}(U)$-module.
Let $\tors_i(U)$ be the submodule in $\sigma^{\ast}E|_{\widetilde S_i}(U)$ which consists of the sections $s\in \sigma^{\ast}E|_{\widetilde S_i}(U)$ such that $s$ is annihilated by some prime ideal of positive codimension in ${\mathcal O}_{\widetilde S_i}(U)$. The cor\-respondence
 $U \mapsto \tors_i(U)$ defines a subsheaf $\tors_i
\subset \sigma^{\ast}E|_{\widetilde S_i}.$ Note that the associated primes of positive codimensions which annihilate the sections $s\in \sigma^{\ast}E|_{\widetilde S_i}(U)$ correspond
to the subschemes supported in the preimage 
$\sigma^{-1}({\rm Supp\,} \varkappa)=\bigcup_{i>0}\widetilde S_i.$
Since the scheme  $\widetilde S=\bigcup_{i\ge
0}\widetilde S_i$ is connected by the construction \cite{Tim3}, the subsheaves $\tors_i,
i\ge 0,$ allow us to construct a subsheaf $\tors \subset
\sigma^{\ast}E$. The latter subsheaf is defined as follows. A
section  $s\in \sigma^{\ast}E|_{\widetilde S_i}(U)$ satisfies the
condition $s\in \tors|_{\widetilde S_i}(U)$ if and only if
\begin{itemize}
\item{there exists a section  $y\in {\mathcal O}_{\widetilde S_i}(U)$ such that
 $ys=0$,}
\item{at least one of the following two conditions is satisfied:
either  $y\in {\mathfrak p}$, where $\mathfrak p$ is a prime ideal
of positive codimension, or there exist a Zariski-open subset
$V\subset \widetilde S$ and a section  $s' \in \sigma^{\ast}E (V)$
such that  $V\supset U$, $s'|_U=s$, and $s'|_{V\cap \widetilde
S_0} \in$ $\tors (\sigma^{\ast}E|_{\widetilde S_0})(V\cap
\widetilde S_0)$. In the latter expression the torsion subsheaf
$\tors(\sigma^{\ast}E|_{\widetilde S_0})$ is understood in the usual
sense.}
\end{itemize}

The role of the subsheaf $\tors \subset \sigma^{\ast}E$ in our
construction is analogous to the role of a torsion subsheaf in the
case of a reduced and irreducible base scheme. Since no confusion
occurs, the symbol  $\tors$ is understood everywhere in the described
sense. The subsheaf $\tors$ is called the {\it torsion subsheaf \textup{(}in the modified sense\textup{)}}.

In \cite{Tim4} it is proven that the sheaves $\sigma^{\ast} E/\tors$
are locally free. The sheaf $\widetilde E$ included in the pair
$((\widetilde S, \widetilde L), \widetilde E)$ is defined by the
formula $\widetilde E = \sigma^{\ast}E/\tors$. In this
situation there is an isomorphism $H^0(\widetilde S, \widetilde
E \otimes \widetilde L) \cong H^0(S,E\otimes L).$

Also in \cite{Tim4} it was proven that the restriction of the
sheaf  $\widetilde E$ to each of the components $\widetilde S_i$,
$i>0,$ is given by the quasi-ideality relation (\ref{quasiideal}),
where $q_0 \colon  {\mathcal O}_S^{\oplus r}\twoheadrightarrow \varkappa$
is an epimorphism defined by the exact triple  $0\to E \to E^{\vee
\vee} \to \varkappa \to 0$. Here the sheaf
$E^{\vee \vee}$ is locally free.

Resolution of singularities of a semistable sheaf $E$ can be
globalized in a flat family by means of the construction developed
in various versions in  \cite{Tim1, Tim2, Tim4, Tim8}. Let $T$ be
a  scheme, ${\mathbb E}$ a sheaf of  ${\mathcal O}_{T\times
S}$-modules, ${\mathbb L}$ an invertible ${\mathcal O}_{T\times
S}$-sheaf very ample relative to $T$ and such that  ${\mathbb
L}|_{t\times S}=L$. Also let $\chi({\mathbb E} \otimes {\mathbb
L}^n|_{t\times S})=rp(n)$ for all closed points $t\in T$. In  \cite{Tim1, Tim2, Tim4, Tim8} the sheaves $\E$ and
$\L$ were assumed to be flat relative to $T$, and $T$ was assumed
to contain a nonempty open subset $T_0$ such that ${\mathbb
E}|_{T_0\times S}$ is a locally free ${\mathcal O}_{T_0 \times
S}$-module. Then the following objects were defined:
\begin{itemize}
\item{$\pi \colon  \widetilde \Sigma \to \widetilde T$ is a flat family of admissible
schemes with an invertible ${\mathcal O}_{\widetilde \Sigma}$-module
$\widetilde {\mathbb L}$ such that  $\widetilde {\mathbb
L}|_{t\times S}$ is a distinguished polarization of the scheme
$\pi^{-1}(t)$,}
\item{$\widetilde {\mathbb E}$ is a locally free
${\mathcal O}_{\widetilde \Sigma}$-module and $((\pi^{-1}(t),
\widetilde {\mathbb L}|_{\pi^{-1}(t)}),  \widetilde {\mathbb
E}|_{\pi^{-1}(t)})$ is an $S$-semi\-stable admiss\-ible pair.}
\end{itemize}
In this situation there is a blowup morphism $\Phi: \widetilde
\Sigma \to \widetilde T \times S$, and the scheme $\widetilde T$ is
birational to the initial base scheme $T$. In \cite{Tim8} the
procedure  is modified so that $\widetilde T
\cong T$.

The mechanism described was called in \cite{Tim4} the {\it standard resolution}.

In the present paper we admit that the open subset $T_0$ such that
$\E|_{T_0\times S}$ is locally free can be empty. 

\smallskip

In \cite{Tim12} the following result was proven.
\begin{theorem}[\cite{Tim12}]
Let $T$ be an algebraic scheme of finite type over~$k$ and $L$ be a (big enough) very ample invertible sheaf on a nonsingular projective algebraic variety~$S$. Set $\Sigma=T\times S$.
Let $\E$ be a $T$-flat coherent $\OO_\Sigma$-module of rank~$r$ and such that $\E|_{t\times S}$ is a torsion-free coherent sheaf and the fiberwise Hilbert polynomial $\chi(\E|_{t\times S} \otimes L^n)=rp(n)$ is independent of the choice of~$t\in T$. Then there are \begin{itemize} \item{a flat family $\pi \colon \widetilde \Sigma \to T$ with a  morphism $\sigma\!\!\!\sigma \colon \widetilde \Sigma \to \Sigma$ of $T$-schemes,}
\item{an invertible $\OO_{\widetilde \Sigma}$-sheaf $\widetilde \L$,}
\end{itemize}
such that \begin{itemize}\item{$\widetilde \E=\sigma\!\!\!\sigma^{\ast}\E/\tors$ is a locally free $\OO_{\widetilde \Sigma}$-module,} \item{$\chi(\widetilde \L^n|_{\pi^{-1}(t)})$ is uniform over $t\in T$,} \item{$\chi (\widetilde \E \otimes \widetilde \L^n|_{\pi^{-1}(t)})=rp(n)$.}
\end{itemize}
\end{theorem}
The resolution described in \cite{Tim12}
takes any coherent torsion-free $\OO_S$-sheaf $E$ on a polarized projective scheme $S$ to an {\it admissible pair} of the form $((\widetilde S, \widetilde L),\widetilde E)$, which is a generalization of the one in the dimension two (resp., homological dimension 1) case. In particular, $\widetilde S$ is an {\it admissible scheme}; its structure depends on the structure of the sheaf $E$ under the resolution. In the present paper we  extend the notion of stability (semistability)  for such pairs and  provide a functorial approach to their moduli scheme.

Using the results of \cite{Tim12}, in the present paper we prove the following theorem.
\begin{theorem} \label{thfunc} 
\textup{(}i\textup{)} There is a natural transformation
$\underline \kappa \colon {\mathfrak f}^{GM} \to {\mathfrak f}$ of the
Gieseker--Maruyama moduli functor ${\mathfrak f}^{GM}$ for a nonsingular projective algebraic variety $S$ to the moduli functor ${\mathfrak f}$ of
admissible semistable pairs with the same rank and Hilbert polynomial.

\textup{(}ii\textup{)} There is a natural transformation $\underline \tau \colon
{\mathfrak f} \to {\mathfrak f}^{GM}$ of the moduli functor of
admissible semistable pairs  to the Gieseker--Maruyama moduli
functor for sheaves with the same rank and Hilbert polynomial.

\textup{(}iii\textup{)} The natural transformations $\underline \kappa$ and $\underline
\tau$ satisfy the relation $$\underline
\tau \circ \underline \kappa = \id_{{\mathfrak f}^{GM}}.$$

\textup{(}iv\textup{)}  The  nonreduced moduli scheme $\widetilde M$ for ${\mathfrak f}$
is isomorphic to the  nonreduced Gieseker--Maruyama scheme
$\overline M$ for sheaves with the same rank and Hilbert polynomial.
\end{theorem}

In Section \ref{s2} we give an outline of the  multidimensional version of the standard resolution developed in \cite{Tim12}. It leads to the functorial morphism $\underline \kappa \colon {\mathfrak f}^{GM} \to {\mathfrak f}$ and proves the part ({\it \!i}) of  Theorem \ref{thfunc}.

Section 3 is devoted to the notion of (semi)stability of an admissible pair and to its relation to the (semi)stability of a coherent sheaf $E$ in the sense of Gieseker when the admissible pair arises from the standard resolution of $E$. The class of semistable admissible pairs, as we define it, in dimension greater than two is much wider than the one in dimension two. 

In Section 4 we introduce and examine the notion of  M-equivalence (monoidal equivalence) for admissible pairs and its relation to the S-equivalence of coherent sheaves. 

Section 5 contains an explicit decsription of the functor  ${\mathfrak f}$ of moduli of admissible pairs and of the classical Gieseker--Maruyama functor ${\mathfrak f}^{GM}$ of moduli of coherent sheaves.

Section 6 contains a construction of the natural transformation $\underline \tau\colon {\mathfrak f} \to  {\mathfrak f}^{GM}$  taking the functor $\mathfrak f$ of moduli for admissible pairs (\ref{funcmy}), (\ref{class}) to the functor ${\mathfrak f}^{GM}$ of moduli of coherent sheaves in the sense of Gieseker and Maruyama (\ref{funcGM}), (\ref{famGM}). 
The rank~$r$ and the polynomial~$p(n)$ are fixed and equal for both moduli functors. We give a description of the
transformation of a family of semistable admissible pairs
 $((\pi\colon \widetilde \Sigma \to T, \widetilde \L),
\widetilde \E)$ with a (possibly, nonreduced) base scheme~$T$ to a
family $\E$ of coherent torsion-free semistable sheaves with the
same base~$T$. The transformation provides a morphism $\underline \tau$ of the
functor ${\mathfrak f}$ of admissible semistable pairs
 to the Gieseker--Maruyama functor $\mathfrak f^{GM}$. This proves the part ({\it \!ii}) of Theorem \ref{thfunc}.

In Section \ref{s5} we investigate the relations among the both functors by examining their composites. We prove that $\underline
\tau \circ \underline \kappa = \id_{{\mathfrak f}^{GM}}$ but  $\underline \kappa \circ \underline
\tau $ takes any family of pairs to some other family of pairs but these two families are M-equivalent (monoidally equivalent). This shows that the moduli schemes $\widetilde M$ and  $\overline M$ are isomorphic although the morphisms of functors
$\underline \kappa \colon {\mathfrak f}^{GM} \to {\mathfrak f}$ and
$\underline \tau \colon {\mathfrak f} \to  {\mathfrak f}^{GM}$ we
constructed are not mutually inverse. 
This completes the proof of Theorem \ref{thfunc}.

{\bf Acknowledgements.} 
This work was carried out within the framework of a development program for the Regional Scientific and Educational Mathematical Center of the P.G. Demidov Yaroslavl State University with financial support from the Ministry of Science and Higher Education of the Russian Federation (Agreement No. 075-02-2020-1514/1 additional to the agreement on provision of subsidies from the federal budget No. 075-02-2020-1514).

\smallskip
%%%%%%%%%%%%%%%%%%%%%%%%%%%%%%%%%%%%%%%%%%%%%%%%%%%%%%%%%%
\section{Outline of the standard resolution and the GM-to-Pairs transformation}\label{s2}

%%%%%%%%%%%%%%%%%%%%%%%%%%%%%%%%%%%%%%%%%%%%%%%%%%%%%%%%%%%%%%%%%
%%%%%%%%%%% STANDARD RESOLUTION FOR ARBITRARY DIMENSION %%%%%%%%%
%%%%%%%%%%%%%%%%%%%%%%%%%%%%%%%%%%%%%%%%%%%%%%%%%%%%%%%%%%%%%%%%%

In the previous paper \cite{Tim12} the transformation of a family
$(T, \L, \E)$ of semistable coherent torsion-free sheaves to the
family  $((T, \pi\colon  \widetilde \Sigma \to T, \widetilde
\L), \widetilde \E)$ of admissible semistable pairs was developed for an $N$-dimensional
ground variety $S$. Here we give an outline of the construction from \cite{Tim12}. 

The initial family of sheaves $\E$ can contain no locally free
sheaves and the codimension of the singular locus $\Sing \E$ in
$\Sigma=T\times S$ is more than or equals  two. Hence, if the blowing
up $\sigma\!\!\!\sigma\colon \widetilde \Sigma \to \Sigma$ in the sheaf of
ideals $\FFitt^0 \EExt^1(\E, \OO_{\Sigma})$ (or in any other sheaf of
ideals of some subscheme  supported in the singular locus of
$\E$) is considered, then the fibers of the composite $p\circ
\sigma\!\!\!\sigma$ are not obliged to be equidimensional. Such a blowing up 
cannot produce a family of admissible schemes. Equidimensionality is needed because
admissible schemes must be included in flat families.
%%%%%%%%%%%%%%%%%%%%%%%%%%%%%%%%%%%%%%%%%%%%%%%%%%%%%%%%%%%%%%%%%
%%%%%%%%%%%%%%%%%%%%%%%%%%%%%%%%%%%%%%%%%%%%%%%%%%%%%%%%%%%%%%%%%

To overcome this difficulty we perform a trick (base extending) done in \cite{Tim10}. Consider the product $\Sigma'=\Sigma \times \A^1$ and fix the closed immersion $i_0\colon \Sigma \hookrightarrow \Sigma'$ which identifies $\Sigma$ with the zero fiber~$\Sigma \times 0.$ Now let $Z\subset \Sigma$  be a subscheme defined by a sheaf of ideals $\I=\FFitt^0 \EExt^1(\E, \OO_{\Sigma})$. Then consider the
sheaf of ideals $\I':=\ker (\OO_{\Sigma'} \twoheadrightarrow i_{0\ast} \OO_Z)$ and the blowing up morphism $\sigma\!\!\!\sigma'\colon\widehat \Sigma' \to \Sigma'$ defined by the sheaf $\I'$. Denote the
projection onto the product of factors $\Sigma' \to T\times \A^1=: T'$ by $p'$ and the composite $p' \circ \sigma\!\!\!\sigma'$ by $\widehat \pi'$. We are interested in the induced morphism $\pi\colon
\widetilde \Sigma \to T$ where we have introduced the notation 
$\widetilde \Sigma:=i_0(\Sigma) \times_{\Sigma'} \widehat \Sigma'$. Under the identification $\Sigma \cong i_0(\Sigma)$ we denote by $\sigma\!\!\!\sigma$ the induced morphism $\widetilde \Sigma \to
\Sigma.$ Set $\L':= \L \boxtimes \OO_{\A^1}$. Obviously, there
exists $m\gg 0$ such that the invertible sheaf $\widehat \L':=
\sigma\!\!\!\sigma'^{\ast}{\L'}^m \otimes
({\sigma\!\!\!\sigma'}^{-1} \I') \cdot \OO_{\widehat \Sigma'}$ is
ample relatively to $\widehat \pi'$. For brevity of notations we fix
this $m$ and replace $L$ by its $m$-th tensor power throughout
the further text. Denote $\widetilde \L:=\widehat \L'|_{\widetilde
\Sigma}.$
The following proposition is formulated and proven analogously to Proposition 1 in \cite{Tim10} but has  wider range of applicability.
%%%%%%%%%%%%%%%%%%%%%%%%%%%%%%%%%%%%%%%%%%%%%%%%%%%%%%%%%%
%%%%%%%%%%%%%% FLAT-TO-FLAT RESOLUTION %%%%%%%%%%%%%%%%%%%
%%%%%%%%%%%%%%%%%%%%%%%%%%%%%%%%%%%%%%%%%%%%%%%%%%%%%%%%%%
\begin{proposition}[\cite{Tim12}, Proposition 1] The morphism $\pi\colon\widetilde \Sigma \to
T$ is flat and its fiberwise Hilbert polynomial computed with respect
to $\widetilde \L$, i.e. $\chi (\widetilde \L^n|_{\pi^{-1}(t)})$,
is uniform over $t\in T$.
\end{proposition}

We denote $\sigma\!\!\!\sigma:=\sigma\!\!\!\sigma'|_{\widetilde
\Sigma}\colon\widetilde \Sigma \to \Sigma.$

Let $T$ be an arbitrary  (possibly nonreduced) $k$-scheme of finite type. We assume that its reduction  $T_{\red}$ is irreducible. If $S$ is a surface and  $\E$ is a family of coherent torsion-free sheaves with a base $T$ on $S$,  then the homological dimension of $\E$ as an $\OO_{T\times S}$-module is not greater then one. In our case $S$ has a bigger dimension and we have to work with a locally free resolution of a higher length. Start with a shortest locally free resolution of the family of sheaves $\E$ and cut the corresponding exact $\OO_{\Sigma}$-sequence of length $\ell$
$$ 0\to \widehat E_\ell \to \widehat E_{\ell-1}\to \dots \to \widehat E_0 \to \E \to 0
$$ with locally free $\OO_\Sigma$-modules $\widehat E_\ell, \dots \widehat E_0$ into triples:
\begin{equation}
0\to W_i \to \widehat E_{i-1} \to W_{i-1} \to 0. \label{triple}
\end{equation}
Here $W_\ell=\widehat E_\ell$, $W_1=\ker(\widehat E_0 \to \E)$ and $W_i=\ker(\widehat E_{i-1} \to \widehat E_{i-2})=\coker (\widehat E_{i+1}\to \widehat E_i)$ for $i=2,\dots, \ell-1.$ Also we keep in mind that all the sheaves $W_i$ except $W_\ell$ are not locally free (otherwise the resolution must be shorter). Since $S$ is assumed to be regular, then $\E$ possesses a locally free resolution of the length not greater than $\dim S$.

We resolve the singularities of the sheaf $\E$ by the consequent blowups in the sheaves of ideals $\I'_i=\I_i[t_i]+(t_i),$ $i=1, \dots, \ell$, for the sheaves of ideals
\begin{eqnarray} \I_1=\FFitt^0 \EExt ^1(W_{\ell-1}, \OO_{\Sigma}), &&\I_2=\FFitt^0 \EExt^1(\sigma\!\!\!\sigma_1^\ast
W_{\ell-2},\OO_{\Sigma_1}), \nonumber \\ \I_3=\FFitt^0
\EExt^1(\sigma\!\!\!\sigma_2^\ast\sigma\!\!\!\sigma_1^\ast W_{\ell-3},\OO_{\Sigma_2}), &...,& \I_{\ell}=\FFitt^0 \EExt^0(\sigma\!\!\!\sigma_{\ell-1}^\ast \dots
\sigma\!\!\!\sigma_1^\ast W_0, \OO_{\Sigma_\ell}).\nonumber
\end{eqnarray} Each morphism $\sigma\!\!\! \sigma_i \colon 
\Sigma_i \to \Sigma_{i-1}$ is induced by the sheaf of ideals $\I_i,$
$i=1,\dots, \ell$, and these morphisms form a sequence
\begin{equation}\label{seqmor}
\widetilde \Sigma:=\Sigma_{\ell}
\stackrel{\sigma\!\!\!\sigma_\ell}{\longrightarrow}\dots
\stackrel{\sigma\!\!\!\sigma_1}{\longrightarrow} \Sigma_0:=\Sigma.
\end{equation}
Each segment (\ref{triple}) is resolved by the morphism $\sigma\!\!\!\sigma_{\ell-i+1}$ at
the $(\ell-i+1)^{\rm st}$ step of the process.

%%%%%%%%%%%%%%%%%%%%%%%%%%%%%%%%%%%%%%%%%%%%%%%%%%%%%%%%%%%%%%%%%%%%%%%%%%%%%%
%%%%  NESESSARITY TO DECOMPOSE INTO BLOWING UP AND GROWING UP A COMPONENT %%%%
%%%%%%%%%%%%%%%%%%%%%%%%%%%%%%%%%%%%%%%%%%%%%%%%%%%%%%%%%%%%%%%%%%%%%%%%%%%%%%

\begin{remark}\cite{Tim12} If the scheme $ X=\Sigma_1$ has an irreducible
reduction, then we conclude directly
that  $\hd \sigma \!\!\! \sigma_1 ^{\ast}\EExt^1(W_{\ell-1},
\OO_\Sigma)=1.$ If $\Sigma_1$ has a reducible reduction, then there is
a natural decomposition $\Sigma_1 = \Sigma_1^0 \cup \D_1'$. Here
$\Sigma_1^0 \cong \Proj \bigoplus _{s\ge 0} \I_1^s$ is isomorphic to the scheme
obtained by blowing up $\Sigma$ in the sheaf of ideals $\I_1$ and
$\D_1'$ is an exceptional divisor of the blowing up morphism $\sigma
\!\!\! \sigma_1'\colon\Sigma'_1 \to \Sigma \times \A^1.$ Their
scheme-theoretic intersection equals the exceptional divisor $\D_1$
of the blowing up morphism $\sigma \!\!\! \sigma_1^0\colon\Sigma_1^0 \to
\Sigma$:
$$\Sigma_1^0 \cap \D_1'=\D_1.$$ One comes to the
decomposition of the morphism $\sigma\!\!\! \sigma_1$:
$$
\Sigma_1 \stackrel{\delta \!\!\delta_1}{\longrightarrow} \Sigma_1^0
\stackrel{\sigma \!\!\! \sigma_1^0}{\longrightarrow} \Sigma,
$$
where $\delta \!\!\delta_1$ acts identically on $ \Sigma_1^0$ and
its action on $\D_1'$
$$
\delta \!\!\delta_1 |_{\D_1'}\colon\D_1' \to  \Sigma_1^0
$$
factors through the exceptional divisor $\D_1=\Proj \bigoplus_{s\ge
0}\I_1^s/\I_1^{s+1}$ of the morph\-ism $\sigma \!\!\! \sigma_1^0$
and is defined by the structure of a $\bigoplus_{s\ge 0}
\I_1^s/\I_1^{s+1}$-algebra on the graded ring $\bigoplus_{s\ge 0}
\I_1'^s/\I_1'^{s+1}.$
\end{remark}
\begin{convention}
For uniformity of the notation we use the decomposition $\sigma\!\!\!\sigma_i = \sigma\!\!\!\sigma^0_i \circ \delta\!\!\delta_i$ for all $i=1,\dots,\ell$ with no reference if $\delta\!\!\!\,\delta_i$ is either non-identity or identity morphism.
\end{convention}

%%%%%%%%%%%%%%%%%%%%%%%%%%%%%%%%%%%%%%%%%%%%%%%%%%%%%%%%%%%%%%%%%%%%
The composites $
\delta\!\!\delta_{i} \circ \sigma\!\!\!\sigma_{i+1}^0 $ are  ``interchangeable'' in the sense of the commutative diagram
\begin{eqnarray}\label{interch}\xymatrix{
\Sigma_{i+1}^0 \ar[d]_{\delta\!\!\delta_i^{-1}}
\ar[r]^{\sigma\!\!\!\sigma_{i+1}^0}&
\Sigma_i \ar[d]^{\delta\!\!\delta_i}\\
\Sigma_{i+1}^{-1}\ar[r]^{\sigma\!\!\!\sigma_{i+1}^{-1}}& \Sigma_i^0
}\end{eqnarray} 
\begin{remark} We use double indexing with lower and upper indices for
schemes and morph\-isms. For example, $-1$ in the notation
$\delta\!\!\delta_i^{-1}$ or $\sigma\!\!\!\sigma_{i+1}^{-1}$ is just
an index but neither a sign of the inverse of the map nor a notation for
the inverse image of a sheaf.
\end{remark}

The resolution of singularities of the sheaf $\E$ is done by the
sequence of morphisms $\sigma\!\!\!\sigma_i=\delta\!\!\delta_i \circ
\sigma\!\!\!\sigma_i^0$, $i=1,\dots,\ell.$ Their composite can be
decomposed into the following diagram by iterating the square
(\ref{interch}):
\begin{equation}\xymatrix{
\Sigma_\ell \ar[d]_{\delta\!\!\delta^0_\ell}\ar[dr]^{\sigma\!\!\!\sigma_\ell}\\
\Sigma_\ell^0 \ar@{.}[d]\ar[r]^{\sigma\!\!\!\sigma_\ell^0}
&\Sigma_{\ell-1}
\ar@{.}[d] \ar@{.}[rd]\\
\Sigma_\ell^{-i+2} \ar[d]_{\delta\!\!\delta_{i+1}^{-i+1}}
\ar[r]^{\sigma\!\!\!\sigma_{\ell}^{-i+2}} &\Sigma_{\ell-1}^{-i+3}
\ar[d]_{\delta\!\!\delta_{i+1}^{-i+2}} \ar@{.}[r] &\Sigma_{i+1}
\ar[d]_{\delta\!\!\delta_{i+1}^0}
\ar[rd]^{\sigma\!\!\!\sigma_{i+1}}\\
\Sigma_\ell^{-i+1}
\ar[d]_{\delta\!\!\delta_i^{-i}}\ar[r]^{\sigma\!\!\!\sigma_\ell^{-i+1}}
&\Sigma_{\ell-1}^{-i+1} \ar[d]_{\delta\!\!\delta_i^{-i+1}}\ar@{.}[r]
&\Sigma_{i+1}^0 \ar[d]_{\delta\!\!\delta_i^{-1}}
\ar[r]^{\sigma\!\!\!\sigma_{i+1}^0}&
\Sigma_i \ar[d]_{\delta\!\!\delta_i^0} \ar[rd]^{\sigma\!\!\!\sigma_i}\\
\Sigma_\ell^{-i}
\ar@{.}[d]\ar[r]^{\sigma\!\!\!\sigma_\ell^{-i}}&\Sigma_{\ell-1}^{-i+1}
\ar@{.}[d]\ar@{.}[r]&\Sigma_{i+1}^{-1}
\ar[r]^{\sigma\!\!\!\sigma_{i+1}^{-1}}\ar@{.}[d] &\Sigma_i^0
\ar[r]^{\sigma\!\!\!\sigma_i^0}\ar@{.}[d]
&\Sigma_{i-1} \ar@{.}[d] \ar@{.}[rd]\\
\Sigma_{\ell}^{ -\ell+2} \ar[d]_{\delta\!\!\delta_1^{-\ell+1}}
\ar[r]^{\sigma\!\!\!\sigma_\ell^{-\ell+2}}&\Sigma_{\ell-1}^{-\ell+3}
\ar[d]_{\delta\!\!\delta_1^{-\ell+2}}\ar@{.}[r]&\Sigma_{i+1}^{-i+1}
\ar[d]_{\delta\!\!\delta_1^{-i}}\ar[r]^{\sigma\!\!\!\sigma_{i+1}^{-i+1}}&\Sigma_i^{-i+2}
\ar[d]_{\delta\!\!\delta_1^{-i+1}}
\ar[r]^{\sigma\!\!\!\sigma_i^{-i+2}}&\Sigma_{i-1}^{-i+3}
\ar[d]_{\delta\!\!\delta_1^{-i+2}} \ar@{.}[r]
&\Sigma_1 \ar[d]_{\delta\!\!\delta_1^0} \ar[rd]^{\sigma\!\!\!\sigma_1}\\
\Sigma_{\ell}^{-\ell+1}
\ar[r]^{\sigma\!\!\!\sigma_\ell^{-\ell+1}}&\Sigma_{\ell-1}^{-\ell+2}
\ar@{.}[r]&\Sigma_{i+1}^{-i}
\ar[r]^{\sigma\!\!\!\sigma_{i+1}^{-i}}&\Sigma_i^{-i+1}
\ar[r]^{\sigma\!\!\!\sigma_i^{-i+1}}&\Sigma_{i-1}^{-i+2}
\ar@{.}[r]&\Sigma_1^0
\ar[r]^{\sigma\!\!\!\sigma_1^0}&\Sigma}\label{bigdiag}
\end{equation}
Each square of this diagram has a view \begin{equation*}
\xymatrix{\Sigma_i^{-j} \ar[d]_{\delta\!\!\delta_{i-1}^{-j-1}}
\ar[r]^{\sigma\!\!\!\sigma_i^{-j}}& \Sigma_{i-1}^{-j+1}
\ar[d]^{\delta \!\!\delta_{i-1}^{-j}}\\
\Sigma_i^{-j-1} \ar[r]_{\sigma
\!\!\!\sigma_i^{-j-1}}&\Sigma_{i-1}^{-j}}
\end{equation*}
for $i=0, \dots, \ell,$ $j=-1, \dots,i-1$, where $\Sigma_0:=\Sigma$,
$\Sigma_i^1:=\Sigma_i$, $\delta\!\!\delta_i^0:=\delta\!\!\delta_i$.

 In this diagram the
bottom horizontal row is a composite of consequent blowups and the
left vertical column $\delta\!\!\delta_1^{-\ell+1} \circ \dots \circ
\delta\!\!\delta_\ell^0$ provides flatness of the scheme
$\Sigma_\ell=\widetilde \Sigma$ over its base $T$ by consequent
growing ups additional components of several levels and
$(\delta\!\!\delta_1^{-\ell+1} \circ \dots \circ
\delta\!\!\delta_\ell^0)^\ast$ produces just an inverse image of
a locally free sheaf. Each square of the diagram is an analog of
(\ref{interch}).
%%%%%%%%%%%%%%%%%%%%%%%%%%%%%%%%%%%%%%%%%%%%%%%%%%%%%%%%%%%%%%%%%%%%%%%%%%%

In arbitrary homological dimension we act
inductively and do interchanging in the inductive step.

Now we pass to the next segment
\begin{equation} \label{seg}
0\to W_{\ell-1} \to \widehat E_{\ell -2} \to W_{\ell -2 }\to 0
\end{equation}
and to its ``inverse image'' under $\sigma \!\!\! \sigma_1 ^0$:
\begin{equation} \label{seg'}0\to W'_{\ell -1} \to \sigma \!\!\! \sigma_1 ^{0\ast} \widehat E_{\ell
-2} \to \sigma \!\!\! \sigma_1 ^{0\ast} W_{\ell -2} \to 0.
\end{equation}
The ``inverse image'' of (\ref{seg}) under $\sigma \!\!\! \sigma_1 =\sigma \!\!\!
\sigma_1 ^0 \circ \delta \!\! \delta_1 ^0 $ equals the inverse image of (\ref{seg'}) under $\delta \!\! \delta_1 ^0$ which is exact because the kernel $W'_{\ell-1}$ of (\ref{seg'}) is
locally free. In this and other
consequent  segments the cokernel sheaf contains a torsion (in the  modified
sense) but finally the resolution eliminates a torsion.

Next steps are similar to each other and involve inverse images of
consequent segments
\begin{eqnarray}
0\to W'_{\ell-2} \to \sigma \!\!\! \sigma_2 ^{0\ast}\sigma \!\!\!
\sigma_1 ^\ast \widehat E_{\ell-3} \to \sigma \!\!\! \sigma_2 ^{0\ast}\sigma
\!\!\! \sigma_1 ^\ast W_{\ell -3}
\to 0,\nonumber \\
.\quad .\quad .\quad .\quad .\quad .\quad .\quad .\quad .\quad
.\quad .\quad .\quad .\quad .\quad .\quad .\quad .\nonumber \\
0\to W'_{\ell-i} \to \sigma \!\!\! \sigma_i ^{0\ast}
\sigma\!\!\!\sigma_{i-1}^\ast\dots \sigma \!\!\! \sigma_1 ^\ast
\widehat E_{\ell-i-1} \to \sigma \!\!\! \sigma_i ^{0\ast}
\sigma\!\!\!\sigma_{i-1}^\ast \dots \sigma \!\!\! \sigma_1 ^\ast
W_{\ell -i-1}
\to 0,\label{v}\\
.\quad .\quad .\quad .\quad .\quad .\quad .\quad .\quad .\quad
.\quad .\quad .\quad .\quad .\quad .\quad .\quad .\nonumber \\
0\to W'_1 \to \sigma \!\!\! \sigma_{\ell-1} ^\ast\dots \sigma \!\!\!
\sigma_1 ^\ast \widehat E_0 \to \sigma \!\!\! \sigma_{\ell-1} ^\ast\dots
\sigma \!\!\! \sigma_1 ^\ast W_0 \to 0, \nonumber
\end{eqnarray}
where $W_0=\E$ and the kernel sheaf $W'_{\ell -i}$ in the next
triple is a locally free $\OO_{\Sigma_{i-1}}$-module which is produced
by the standard resolution of $\sigma \!\!\! \sigma_i ^\ast\dots \sigma
\!\!\! \sigma_1 ^\ast W_{\ell -i}$ in the previous triple.

As it is proven in \cite{Tim12}, the composite morphism $$\sigma\!\!\!\sigma_{[i]}:=\sigma
\!\!\! \sigma_1 \circ \dots \circ \sigma \!\!\! \sigma_i$$
can be rewritten as $$
\sigma\!\!\!\sigma_{[i]}=\sigma\!\!\!\sigma_{[i-1]}\circ \sigma\!\!\!\sigma_i=
\sigma\!\!\!\sigma_{[i-1]}\circ \sigma\!\!\!\sigma_i^0 \circ \delta\!\!\delta_i^0=
\sigma\!\!\!\sigma_{[i]}^0 \circ \delta\!\!\delta_{[i]}^0,
$$
where we have introduced the notations $$%\sigma\!\!\!\sigma_{[i]}:=\sigma
%\!\!\! \sigma_1 \circ \dots \circ \sigma \!\!\! \sigma_i,\quad
\sigma\!\!\!\sigma_{[i]}^0:=\sigma\!\!\!\sigma_1^0 \circ \sigma\!\!\!\sigma_2^{-1} \circ \dots \circ
\sigma\!\!\!\sigma_i^{-i+1},\quad
\delta\!\!\delta_{[i]}^0:=\delta\!\!\delta_1^{-i+1}\circ \delta\!\!\delta_2^{-i+2}\circ  \dots \circ
\delta\!\!\delta_i^0,\quad \delta\!\!\delta_{[i-1]}^{-1}:=\delta\!\!\delta_1^{-i+1}\circ \delta\!\!\delta_{2}^{-i+2} \circ \dots \circ \delta\!\!\delta_{i-1}^{-1}.$$ 

The resolution leads to the locally free $\OO_{\Sigma_\ell^{-\ell+1}}$-sheaf $\widehat \E:= \sigma\!\!\!\sigma_{[\ell]}^{0\ast} \E /\tors$
and the $\OO_{\Sigma_\ell}$-sheaf $\widetilde \E:=\delta\!\!\delta_{[\ell]}^{0\ast}\widehat \E.$

We will use the notation $\widetilde \Sigma:=\Sigma_\ell$ analogous to the one used for a family of admissible schemes in previous papers.

\begin{remark} Since  $\widetilde \E$ is locally free as
a $\OO_{\widetilde
\Sigma}$-module and $\OO_{\widetilde \Sigma}$ is $\OO_T$-flat, then
$\widetilde \E$ is also flat over $T$.
\end{remark}

The transformation of families we constructed has a form $$(T,
\L, \E) \mapsto (\pi\colon\widetilde \Sigma \to T, \widetilde \L,
\widetilde \E)$$ and is defined by the commutative diagram
\begin{equation*}\xymatrix{ T \ar@{=}[d]\ar@{|->}[r]& \{(T, \L,
\E)\} \ar[d]\\
 T \ar@{|->}[r]& \{(\pi\colon\widetilde \Sigma \to
 T, \widetilde \L, \widetilde \E)\} }
\end{equation*}
 The right vertical arrow is the map of sets.
Their elements are families of objects to be parameterized. The map
is determined by the procedure of the resolution as it was developed in
\cite{Tim12}.
\begin{remark} As we conclude below, the
resolution as it is constructed now  defines a morphism of
functors $\underline \kappa \colon {\mathfrak f}^{GM}\to {\mathfrak f}$, where $\mathfrak f$ is the functor of moduli for admissible pairs (to be described below) and ${\mathfrak f}^{GM}$ the functor of moduli for semistable torsion free sheaves in the setting of Gieseker and Maruyama.
The morphism of the functors  $\underline \kappa \colon {\mathfrak f}^{GM} \to
\mathfrak f$ is defined by the class of commutative diagrams
\begin{equation*}\label{morfun}\xymatrix{T \ar@{|->}[rd] \ar@{|->}[r]&
\mathfrak F^{GM}_T/\sim \ar[d]^{\underline \kappa (T)}\\
& \mathfrak F_T/\sim}
\end{equation*}
where $T\in {\mathcal O}b (Schemes)_k$ and $\underline
\kappa(T)\colon(\mathfrak F^{GM}_T/\sim) \to ( \mathfrak F_T/\sim)$ is
a morphism in the category of sets (mapping).
\end{remark}
\begin{remark}
The procedure of the resolution involves a choice of a locally free resolution of the sheaf under the resolution. But it is proven \cite{Tim12} that the sequence of morphisms (\ref{seqmor}) does not depend on this choice.
\end{remark}

The structure of fibers of the morphism  $\pi\colon\widetilde \Sigma \to T$ is produced directly from  
the composite of the morphisms of $T$-schemes $\pi_i\colon\Sigma_i \to T$, $i=0,\dots,\ell$: 
$$\xymatrix{\widetilde \Sigma \ar@{=}[r]&\!\Sigma_\ell \ar[rd]_{\delta\!\!\delta_\ell^0} \ar[rr]^{\sigma\!\!\!\sigma_\ell}&&\Sigma_{\ell-1} \ar[rd]_{\delta\!\!\delta_{\ell-1}^0}\ar[rr]^{\sigma\!\!\!\sigma_{\ell-1}}&&\dots \ar[rr]^{\sigma\!\!\!\sigma_2}\ar[rd]&&\Sigma_1 \ar[dr]_{\delta\!\!\delta_1^0}\ar[rr]^{\sigma\!\!\!\sigma_1}&&\Sigma_0 \ar@{=}[r]&\Sigma\\
&&\Sigma_{\ell}^0\ar[ur]_{\sigma\!\!\!\sigma_\ell^0}&&\Sigma_{\ell-1}^0\ar[ur]&&\Sigma_2^0 \ar[ur]_{\sigma\!\!\!\sigma_2^0}&& \Sigma_1^0 \ar[ur]_{\sigma\!\!\!\sigma_1^0}}
$$
When restricted to a fiber over any fixed closed point $t\in T$ (or, equivalently, one can set $T$ to be a reduced point) this chain gives rise to the chain of morphisms among fibers $S_i=\pi_i^{-1}(t)$:
$$\xymatrix{\widetilde S \ar@{=}[r]&S_\ell \ar[rd]_{\delta_\ell^0} \ar[rr]^{\sigma_\ell}&&S_{\ell-1} \ar[rd]_{\delta_{\ell-1}^0}\ar[rr]^{\sigma_{\ell-1}}&&\dots \ar[rr]^{\sigma_2}\ar[rd]&&S_1 \ar[dr]_{\delta_1^0}\ar[rr]^{\sigma_1}&&S_0 \ar@{=}[r]&S\\
&&S_{\ell}^0\ar[ur]_{\sigma_\ell^0}&&S_{\ell-1}^0\ar[ur]&&S_2^0 \ar[ur]_{\sigma_2^0}&& S_1^0 \ar[ur]_{\sigma_1^0}}
$$
Each of $\sigma_i^0$ is a blowing up morphism and each of $\delta_i^0$ contracts the additional component onto the exceptional divisor of $\sigma_i^0$. When moving against the arrows we can say that $\sigma_i^0$ blows up the scheme $S_{i-1}$ and $\delta_i^0$ grows the additional component.
Since these two types of morphisms alternate, next blowup is applied to the scheme which can consist of several connected components.

\begin{convention}
We use the notations for a single fiber  which are completely parallel to the notations for families in (\ref{bigdiag}). The fiberwise version for (\ref{bigdiag}) can be obtained when one replaces  $\Sigma_i$ by $S_i$ with all the indices preserved for the schemes.  The double letters $\sigma\!\!\!\sigma $ and $\delta\!\!\delta$ are replaced by the respective usual single letters $\sigma$ and $\delta$ with all the indices preserved for morphisms. The reader should keep in mind that each $\delta$ projects its source scheme to the component of this scheme.
 \end{convention}

We start with the initial non-singular variety $S$. When passing to $S_1^0$ and after that to $S_1$ we see the scheme $S_1$ consisting of a principal component $S_1^0$ and an addit\-ional component $S_1^{add}$. The principal component $S_1^0=(\sigma_1^0)^{-1}S$ is an algebraic variety which is obtained by blowing up $S$. The addit\-ional component $S_1^{add}$ can carry a non-reduced scheme structure and can have a reducible reduction. In the previous papers this closed subscheme appeared as a union of addit\-ional components of an admissible scheme. The addit\-ional component $S_1^{add}$ can consist of several connected components.

Passing to $S_2^0$ leads to the transformation of both $S_1^0$ and $S_1^{add}$. We come to the algebraic variety $S_2^0=(\sigma_2^{-1})^{-1}S_1^0=(\sigma_2^{-1})^{-1}(\sigma_1^0)^{-1}S$ obtained by blowing up the principal component $S_1^0$ of $S_1$ and to the scheme $(\sigma_2^0)^{-1}S_1^{add}$. Passing to $S_2$ against $\delta_2^0$ grows up the additional component $S_2^{add}$, and we can write $S_2=S_2^0 \cup (\sigma_2^0)^{-1}S_1^{add} \cup S_2^{add}$. 

Analogously, on $\ell^{\mathrm{th}}$ step we have the following scheme 
$$
\widetilde S:=S_\ell=(\sigma_1^0\circ \dots \circ \sigma_\ell^{-\ell+1})^{-1}S \cup 
(\sigma_2^0\circ \dots \circ \sigma_\ell^{-\ell+2})^{-1}S_1^{add}\cup \dots \cup (\sigma_\ell^0)^{-1}S_{\ell-1}^{add}\cup S_\ell^{add}.$$

Depending on the structure of the initial  $\OO_S$-sheaf $E$, several morphisms $\sigma_i$ can turn to be identities and hence the actual length of the chain 
$S_1, \dots, S_\ell$ can vary from 0 (for the case when $E$ is locally free) to the maximal value equal to $\ell.$

To measure the numerical invariants of the objects we obtained in the standard resolution we need to fix an appropriate ample invertible sheaf $\widetilde L$ on each $\widetilde S$. The sheaf of analogous role was called a {\it distinguished polarization} in previous papers, where $\hd$-one case was developed. This class of invertible sheaves provides, in particular, fiberwise uniform Hilbert polynomials in flat families of admissible schemes. Strictly speaking, if $\widetilde \L$ is invertible $\OO_{\widetilde \Sigma}$-sheaf which is very ample relatively to the base $T$, then for any closed point $t\in T$ and for any integer $n\gg 0$ Hilbert polynomial of the fiber  $\pi^{-1}(t)$ computed as $\chi (\widetilde \L^n|_{\pi^{-1}(t)})$ is independent of the choice of $t\in T$. 

Let $\Sigma$ carries an invertible sheaf $\L$ which is very ample relatively to $T$. In $\hd$-one case there is a one-step resolution by $T$-morphism $\sigma\!\!\!\sigma\colon\widetilde \Sigma \to \Sigma$ associated with the sheaf of ideals $\I \subset \OO_\Sigma$. In this case is was shown that a distinguished polarization can be chosen as $\widetilde \L=\sigma\!\!\!\sigma^\ast \L^m \otimes \sigma\!\!\!\sigma^{-1}\I \cdot \OO_{\widetilde \Sigma}$, when $m$ is sufficiently big to make $\widetilde \L$ to be ample relatively to $T$.

In the case of a bigger homological dimension $\ell$ this step of constructing a family of polarizations on the $T$-scheme $\widetilde \Sigma$ is iterated $\ell$ times until one comes to 
\begin{equation} \label{polarizfam}
    \widetilde \L:=\L_\ell=[\sigma\!\!\!\sigma_\ell^\ast \dots  [\sigma\!\!\!\sigma_2^\ast[\sigma\!\!\!\sigma_1^\ast \L^{m_1} \otimes (\sigma\!\!\!\sigma_1)^{-1} \I_1 \cdot \OO_{\Sigma_1}]^{m_2} \otimes (\sigma\!\!\!\sigma_2)^{-1}\I_2 \cdot \OO_{\Sigma_2}]^{m_3}\dots ]^{m_\ell}\otimes (\sigma\!\!\!\sigma_\ell)^{-1} \I_\ell \cdot \OO_{\Sigma_\ell}.
\end{equation}

The corresponding $T$-scheme $\Sigma_i$ has fiberwise constant Hilbert polynomial with respect to $\L_i$ after $i$-th step of the resolution.

A distinguished polarization on a single admissible scheme $\widetilde S$ has a view 
\begin{equation}\label{polarizfib}
\widetilde L:=L_\ell=[\sigma_\ell^\ast \dots  [\sigma_2^\ast[\sigma_1^\ast L^{m_1} \otimes (\sigma_1)^{-1} I_1 \cdot \OO_{S_1}]^{m_2} \otimes (\sigma_2)^{-1}I_2 \cdot \OO_{S_2}]^{m_3}\dots ]^{m_\ell}\otimes (\sigma_\ell)^{-1} I_\ell \cdot \OO_{S_\ell}
\end{equation}
The distinguished polarization $\widetilde L$ is assumed to be fixed for each admissible scheme $\widetilde S$. If $\widetilde S=S$, then $\widetilde L=L$ and polarizations $\widetilde L$ for all possible $\widetilde S$ are chosen in such a way that for all admissible schemes $(\widetilde S, \widetilde L)$ their Hilbert polynomials 
are uniform:
$$
\chi(\widetilde L^n)=\chi(L^n), \quad n\gg 0.
$$
Now we redenote $\L^{m_1\cdot m_2 \cdot \dots \cdot m_\ell}$ as $\L$ and $L^{m_1\cdot m_2 \cdot \dots \cdot m_\ell}$ as $L$. From now we work with these new polarizations. Also we come to the following shorthand notations for (\ref{polarizfam}) and (\ref{polarizfib}) respectively: 
\begin{eqnarray}
\widetilde \L=\sigma\!\!\!\sigma^\ast \L \otimes \E \mathrm{xc};\\ \E \mathrm{xc}:=[\sigma\!\!\!\sigma_\ell^\ast\dots [\sigma\!\!\!\sigma_2^\ast[ (\sigma\!\!\!\sigma_1)^{-1} \I_1 \cdot \OO_{\Sigma_1}]^{m_2} \otimes (\sigma\!\!\!\sigma_2)^{-1}\I_2 \cdot \OO_{\Sigma_2}]^{m_3}\dots ]^{m_\ell}\otimes (\sigma\!\!\!\sigma_\ell)^{-1} \I_\ell \cdot \OO_{\Sigma_\ell};
\end{eqnarray}
\begin{eqnarray}
\widetilde L=\sigma^\ast L \otimes \mathrm{Exc}_{S_\ell};\label{polshort}\\\mathrm{Exc}_{S_\ell}:=[\sigma_\ell^\ast \dots [\sigma_2^\ast[ (\sigma_1)^{-1} I_1 \cdot \OO_{S_1}]^{m_2} \otimes (\sigma_2)^{-1}I_2 \cdot \OO_{S_2}]^{m_3}\dots ]^{m_\ell}\otimes (\sigma_\ell)^{-1} I_\ell \cdot \OO_{S_\ell}. \label{polshortexc}
\end{eqnarray}
%%%%%%%%%%%%%%%%%%%%%%%%%%%%%%%%%%%%%%%%%%%%%%%%%%%%%%%%%%%%%%%%%
\section{(Semi)stability} \label{s3}

\begin{definition}
An $\OO_{\widetilde S}$-sheaf $\widetilde E$ {\it is obtained from $E$ by the standard resolution} if for some locally free $\OO_S$-resolution  $\widehat E_\bullet \to E\to 0$
with intermediate quotients $W_i$, $i=0,\dots,\ell,$ $W_0=E,$ $W_\ell =\widehat E_\ell$ and for a sequence of morphisms $\widetilde S=S_\ell \stackrel{\sigma_\ell}{\longrightarrow} S_{\ell-1}\stackrel{\sigma_{\ell-1}}{\longrightarrow} \dots \stackrel{\sigma_2}{\longrightarrow} S_1 \stackrel{\sigma_1}{\longrightarrow} S$ there is an isomorphism
$\widetilde E=\sigma_\ell^\ast \dots \sigma_1^\ast E/\tors$. Morphisms $\sigma_i,$ $i=1,\dots, \ell,$ are specified as in Sec. \ref{s2}: $S_i= \Proj \bigoplus_{s\ge 0} (I_i[t]+(t))^s/(t)^{s+1}$ and $\sigma_i \colon S_i \to S_{i-1}$ is a structure morphism.
\end{definition}
%%%%%%%%%%%%%%%%%%%%%%%%%%%%%%%%%%%%%%%%%%%%%%%%%%%%%%%%%%%%%%%%%
\begin{definition}[\cite{Gies}]
The $\OO_{\widetilde S}$-sheaf $\widetilde E$ is {\it Gieseker-stable} (resp., {\it Gieseker-semistable}) with respect to the polarization $\widetilde L$ if for any subsheaf $\widetilde F \subset \widetilde E$ and  for any proper subsheaf
 $\widetilde F \subset \widetilde E$
for $m\gg 0$ one has
\begin{eqnarray*}
\frac{h^0(\widetilde F\otimes \widetilde L^{m})}{\rank \widetilde F}&<&
\frac{h^0(\widetilde E\otimes \widetilde L^{m})}{\rank \widetilde E},
\\ (\mbox{\rm respectively,} \;\;
\frac{h^0(\widetilde F\otimes \widetilde L^{m})}{\rank \widetilde F}&\leq&
\frac{h^0(\widetilde E\otimes \widetilde L^{m})}{\rank \widetilde E}\;);
\end{eqnarray*} 
\end{definition}

\begin{remark} When $\dim S > 2$, the need to characterize a sequence of centers for blowups in the standard resolution becomes a difficult problem if one wants to eliminate the data of the initial sheaf $E$ under the resolution for formulation of a moduli problem for pairs with no relation to the Gieseker--Maruyama moduli problem. The way to overcome this difficulty is to make the class of pairs wider without any direct restrictions imposed on the blowing up morphisms. This step brings an additional profit: this approach includes some other ways to compactify families of stable vector bundles on the initial algebraic variety $S$. \end{remark}

Let $\sigma_i\colon S_i \to S_{i-1}$ be a birational morphism such that 
\begin{itemize}
    \item $U \subset S_{i-1}$ is the biggest open subset such that $\sigma_i|_{(\sigma_i)^{-1}(U)}\colon(\sigma_i)^{-1}(U) \to U$ is an isomorphism. Denote $H:=S_{i-1}\setminus U.$
    \item  Denote $S^0_i:=\overline{(\sigma_i)^{-1}(U)}$ and call it a {\it principal component} of $S_i$. Also denote $$\sigma_i^0:= \sigma_i|_{S_i^0}\colon S_i^0 \to S_{i-1}.$$ The restriction $\sigma_i|_{S_i^0}\colon S_i^0 \to S_{i-1}$ is a blowup morphism with an exceptional divisor $\sigma_i^{-1}(H).$ 
    \item The restriction $\sigma_i|_{\overline{S_i\setminus S_i^0}}\colon \overline{S_i \setminus S_i^0} \to S_{i-1}$ factors through the closed immersion $H \hookrightarrow S_{i-1}$. This implies that
    \item there is a decomposition of the morphism $\sigma_i$ as $\sigma_i=\sigma_i^0 \circ \delta_i$, where $\delta_i\colon S_i \to S_i^0$ is such that $\delta_i|_{S_i^0}\colon S_i^0 \to S_i^0$ is an identity morphism and there is a decomposition  $$\delta_i|_{\overline{S_i \setminus S_i^0}}\colon\overline{S_i \setminus S_i^0} \to H \hookrightarrow S_{i-1}.$$
    \end{itemize}

    Admissible scheme is included in $$\widetilde S=S_\ell \stackrel{\sigma_\ell}{\longrightarrow} S_{\ell-1} \longrightarrow \dots \longrightarrow S_1 \stackrel{\sigma_1}{\longrightarrow}S=S_0$$
where $\sigma_i,$ $i=1,\dots, \ell$ are as described above.

We will consider $S_i=\Proj \bigoplus_{s\ge 0} (I_i[t]+(t))^s/(t)^{s+1}$ for $I_i \subset \OO_{S_{i-1}}$ and $\sigma_i$ as its structure morphism to $S_{i-1}$. We assume no specification for $I_i,$ $i=1, \dots. \ell.$

A very ample invertible sheaf $L_i$ on $S_i$ is fixed by the equality \begin{equation*}
L_i=[\sigma_i^\ast \dots  [\sigma_2^\ast[\sigma_1^\ast L^{m_1} \otimes (\sigma_1)^{-1} I_1 \cdot \OO_{S_1}]^{m_2} \otimes (\sigma_2)^{-1}I_2 \cdot \OO_{S_2}]^{m_3}\dots ]^{m_i}\otimes (\sigma_i)^{-1} I_\ell \cdot \OO_{S_i}
\end{equation*} and after redenoting $L:=L^{m_1\cdot \dots \cdot m_i}$ it satisfies $h^0(S_i, L_i^m)=h^0(S,L^m)$ for $m\gg 0.$ The sheaf $\widetilde L=L_\ell$ is still called a {\it distinguished polarization} on $\widetilde S.$

If $E_\ell=\widetilde E$ is obtained from the $\OO_S$-sheaf $E=E_0$ then $E_i=\delta_i^{\ast} E_i|_{S_i^0}/\tors$ and $E_i=\sigma_i^\ast E_{i-1}/\tors.$ The torsion is understood in the modified sense. In the case of a sheaf on an integral scheme the modified torsion becomes the usual torsion.

All the considerations about the structure of a sheaf and transformations of a scheme in the process of standard resolution and about interchanging the order of the morphisms $\sigma_i^j$ and $\delta_i^j$ (see Sec. \ref{s2} and \cite{Tim12} for details) hold for a single sheaf $E$ on $S$. This leads to the parallel procedure with consequent $E_i \cong \delta_i^\ast E_i^0$, $E_i^0=E_i|_{S_i^0}$. Also interchanging the order of morphisms $\sigma_i^j$ and $\delta_i^j$ holds for arbitrary ideals $I_i$ being not only Fitting ideals and leads to the resolution for subsequent short exact sequences analogous to (\ref{v}). Hence we rewrite the sequence of the morphisms $\sigma_i$ by $\delta_i$ followed by $\sigma_i^0$. We get $\sigma:= \sigma_\ell \circ \dots \circ \sigma_1= \sigma^0_{[\ell]} \circ  \delta^0_{[\ell]}$. In this case the principal component $\widetilde S^0$ of the whole of the scheme $\widetilde S$ is naturally distinguished: this is the unique component $\widetilde S^0=(\sigma^0_{[\ell]})^{-1} S$. Now by \cite[Sec.~5]{Tim12} $\widetilde E=( \delta^0_{[\ell]}) ^\ast \widetilde E|_{\widetilde S_0}$.

\begin{definition}\label{semistable} An $S$-{\it (semi)stable pair}
$((\widetilde S,\widetilde L), \widetilde E)$ is the following
data:
\begin{itemize}
\item{$\widetilde S=\bigcup_{i\ge 0} \widetilde S_i$ 
an admissible scheme, $S_i$, $i\ge 0$ its components, $\sigma \colon \widetilde S \to S$ a canonical
morphism;}
\item{$\widetilde E$ is a vector bundle on the scheme
$\widetilde S$;}
\item{$\widetilde L \in Pic\, \widetilde S$ is a distinguished polarization;}
\end{itemize}
such that
\begin{itemize}
\item{$\chi (\widetilde E \otimes \widetilde
L^{n})=rp_E(n);$}
\item{the sheaf $\widetilde E$ is {\it Gieseker-stable} (resp., {\it Gieseker-semistable}) with respect to the polarization $\widetilde L$ on
the scheme~$\widetilde S$. }
\item{
the sheaf  $\widetilde E$ is a
{\it quasi-ideal sheaf,} namely it has a description of the form
$\widetilde E= (\delta^0_{[\ell]} )^\ast \widetilde E|_{\widetilde S_0}$. }\end{itemize}
\end{definition}

\begin{remark} If  $\widetilde S \cong S,$ then the stability (resp., semistability)
of a pair $(\widetilde S, \widetilde E)$
is equivalent to the Gieseker-stability (resp., semistability) of the vector bundle
$\widetilde E$ on the variety $\widetilde S\cong S$ with respect to the
polarization $\widetilde L \in Pic\, \widetilde S.$\end{remark}

To investigate the relation of an $S$-stability (resp., $S$-semistability) of the sheaf
 $\widetilde E$ to the Gieseker-stability (resp., Gieseker-semistability) of
 the corresponding sheaf $E$ on the variety $S$ note that for $m\gg 0$ one has
 $rp_E(m)=h^0(E\otimes
L^{m})$. For the Gieseker-stability the behaviour of the Hilbert
polynomial under $m\gg 0$ is important. Therefore we assume that
$m$ is big enough.

%%%%%%%%%%%%%%%%%%%%%%%%%%%%%%%%%%%%%%%%%%%%%%%%%%%%%%%%%%%%%%%%%%%%%%%%
%%% HERE ABOUT DISTINGUISHED ISOMORPHISM OF GLOBAL SECTIONS GROUPS!!! %%
%%%%%%%%%%%%%%%%%%%%%%%%%%%%%%%%%%%%%%%%%%%%%%%%%%%%%%%%%%%%%%%%%%%%%%%%
Now we regret for a time to a discussion of the standard resolution procedure.
We consider a family $\Sigma=T\times S$ whose base $T$ is a projective curve, and this  family is supplied with a sheaf~$\E$ which is flat over~$T$ and provides a torsion-free sheaf when restricted on any fiber. Also let $\Sigma$ carries an invertible sheaf $\L$ which is very ample relatively to~$T$ such that fiberwise Hilbert polynomial $\chi((\E\otimes \L^u)|_{t\times S})$ is uniform over~$T$ and equals to $rp(u)$. Let a general fiber  $t\times S$ over $t\in T$ carry a locally free sheaf $\E|_{t\times S}$ and a special fiber carry a non locally free sheaf $\E|_{t\times S}$. 
Applying the standard resolution to this family we come to a family $\pi\colon\widetilde \Sigma \to T$ supplied with a locally free sheaf $\widetilde \E.$ Since any projective morphism with a base being a projective curve is necessarily flat, then  $\pi$ is to be flat. Since general fibers do not overcome any transformation, the Hilbert polynomials of fiberwise restrictions of $\widetilde \E$ with respect to $\widetilde \L$ remain equal to the initial Hilbert polynomial
$\chi((\widetilde \E\otimes \widetilde \L^u)|_{\pi^{-1}(t)}=\chi((\E\otimes \L^u)|_{t\times S})=rp(u)$.

If we regret to the iterative process of the standard resolution then each step is a standard resolution of a sheaf of the homological dimension equal to 1 by the morphism $\sigma\!\!\!\sigma_i=\sigma\!\!\!\sigma_i^0 \circ \delta\!\!\delta_i \colon \Sigma_i\to \Sigma_{i-1}$. This step transforms the flat family of schemes $\Sigma_{i-1}$ with a distinguished polarization $\L_{i-1}$ to the flat family of schemes $\Sigma_i$ with a distinguished polarization $\L_i$.
In the triple $$0\to \delta\!\!\delta_i^\ast W'_{\ell-i} \to \sigma \!\!\! \sigma_i ^{\ast}
\sigma\!\!\!\sigma_{i-1}^\ast\dots \sigma \!\!\! \sigma_1 ^\ast
\widehat E_{\ell-i-1} \to \sigma \!\!\! \sigma_i ^{\ast}
\sigma\!\!\!\sigma_{i-1}^\ast \dots \sigma \!\!\! \sigma_1 ^\ast
 W_{\ell -i-1}
\to 0,
$$
 $\delta\!\!\delta_i^\ast W'_{\ell-i}$ being locally free on $\Sigma_i$ turns to be $T$-flat and hence the  cokernel of the injective morphism of flat sheaves is also necessarily flat. Since uniformity of fiberwise Hilbert polynomial holds for the distinguished polarization $\widetilde \L_i$ for the kernel and the extension. The same holds for the cokernel.

Now we get confirmed that the standard resolution takes flat families to flat families with uniformity of fiberwise Hilbert polynomial computed with respect to the distinguished polarization.

\begin{proposition} \label{sectionsrel}
The standard resolution of the family $((\Sigma, \L),\E)$ induces the fixed isomorphism of $k$-vector spaces $H^0(\widetilde S, \widetilde E \otimes \widetilde L^m)\cong H^0(S,E\otimes L^m)$, for all $m\gg 0$, for $\widetilde E=\widetilde \E|_{\pi^{-1}(t)}$, $E=\E|_{t\times S}$ and for each closed point $t\in T.$
\end{proposition}
\begin{remark}
In the case when $\hd \E=1$ this result was proven in \cite{Tim3} using the fact that in this case  $\E$ is a reflexive sheaf. But reflexivity is not guaranteed when $\hd \E > 1.$
\end{remark}
\begin{proof}
Start with a commutative diagram
\begin{equation*}
    \xymatrix{\widetilde \Sigma \ar[d]_\pi \ar[r]^{\sigma\!\!\!\sigma}& \Sigma \ar[d]^p\\
    T\ar@{=}[r]&T}
\end{equation*}
and consider the  morphisms of  sheaves
\begin{eqnarray*}
\E \otimes \L^m \hookrightarrow \sigma\!\!\!\sigma_\ast \sigma\!\!\!\sigma^{\ast}(\E\otimes \L^m)=\sigma\!\!\!\sigma_\ast (\sigma\!\!\!\sigma^{\ast}\E\otimes \sigma\!\!\!\sigma^{\ast}\L^m) \to 
\sigma\!\!\!\sigma_\ast (\sigma\!\!\!\sigma^{\ast}\E/\tors\otimes \sigma\!\!\!\sigma^{\ast} \L^m)=\sigma\!\!\!\sigma_\ast \widetilde\E\otimes \L^m,\\
\sigma\!\!\!\sigma_{\ast}(\widetilde \E \otimes \widetilde \L^m)\hookrightarrow \sigma\!\!\!\sigma_{\ast}\widetilde \E \otimes \L^m.
\end{eqnarray*}   
The first row composite morphism is injective. Now apply the direct image $p_\ast$ and by the equality $p_\ast \sigma\!\!\!\sigma_{\ast}=\pi_{\ast}$ we come to the following two monomorphisms 
$$p_\ast (\E \otimes \L^m) \hookrightarrow p_\ast (\sigma\!\!\!\sigma_\ast \widetilde\E\otimes \L^m) \hookleftarrow \pi_\ast (\widetilde \E \otimes \widetilde \L^m).
$$
Now we are to prove that images of both locally free sheaves $p_\ast (\E \otimes \L^m)$ and $\pi_\ast (\widetilde \E \otimes \widetilde \L^m)$ coinside in $p_\ast (\sigma\!\!\!\sigma_\ast \widetilde\E\otimes \L^m)$.
We will prove it by restrictions onto fibers.
Since $\pi$ is a projective morphism of Noetherian schemes, the sheaf $\widetilde \E \otimes \widetilde \L^m$ is flat over $T$ and functions $t\mapsto H^0(\pi^{-1}(t), (\widetilde \E \otimes \widetilde \L^m)|_{\pi^{-1}(t)})$ for $m\gg0$ are constant in $t\in T$, then by \cite[ch.III, Corollary 12.9]{Hart} there is an isomorphism 
$$\pi_\ast (\widetilde \E \otimes \widetilde \L^m) \otimes k_t \cong H^0(\pi^{-1}(t), (\widetilde \E \otimes \widetilde \L^m)|_{\pi^{-1}(t)}). 
$$
Also by the analogous reason 
$$p_\ast (\E\otimes \L^m) \otimes k_t \cong H^0(t\times S, (\E \otimes\L^m)|_{t\times S}).
$$
Note that the morphism $\sigma_t\colon\pi^{-1}(t) \to t\times S$, which is induced by the restriction of $\sigma\!\!\!\sigma$ to the fiber  over $t\in T$, is an isomorphism on a nonempty open subset. Namely, for $\widetilde S_t:=\pi^{-1}(t)$ there are open subschemes $\widetilde S_{0t} \subset \widetilde S_t$ and $S_{0t}\subset t\times S$ such that $\sigma|_{\widetilde S_{0t}}\colon\widetilde S_{0t} \to S_{0t}$ is an isomorphism. Also for $\widetilde \E|_{\pi^{-1}(t)}=\widetilde E_t$, $\widetilde L_t=\widetilde \L|_{\pi^{-1}(t)}$ and for $E_t=\E|_{t\times S}$ the corresponding restrictions are isomorphic locally free sheaves: $\widetilde E_t|_{\widetilde S_{0t}} \cong \sigma_{0t}^{\ast} E_t|_{S_{0t}}$. Here $S_{0t}$ is precisely the maximal subset where the sheaf $E_t$ is locally free.
The restriction maps $$H^0(t\times S, E_t \otimes L^m) \hookrightarrow H^0(S_{0t}, (E_t \otimes L^m)|_{S_{0t}})$$ and $$H^0(\widetilde S_t, \widetilde E_t \otimes \widetilde L_t^m) \hookrightarrow H^0(\widetilde S_{0t}, (\widetilde E_t \otimes \widetilde L_t^m)|_{\widetilde S_{0t}})$$ together with coincidence 
$H^0(S_{0t}, (E_t \otimes L^m)|_{S_{0t}})=H^0(\widetilde S_{0t}, (\widetilde E_t \otimes \widetilde L_t^m)|_{\widetilde S_{0t}})$ lead to the bijective  correspondence $$H^0(t\times S, E_t \otimes L^m) \sim H^0(\widetilde S_t, \widetilde E_t \otimes \widetilde L_t^m)$$ also as images in $\sigma\!\!\!\sigma_\ast \widetilde\E\otimes \L^m \otimes k_t$. 
\end{proof}
Hence Proposition \ref{sectionsrel} fixes the isomorphism $\upsilon \colon H^0(\widetilde S, \widetilde E \otimes
\widetilde L^{m}) \stackrel{\sim}{\to} H^0(S,E\otimes L^{m})$.
\begin{proposition}\label{ssc} Let the locally free  $\OO_{\widetilde
S}$-sheaf $\widetilde E$ is obtained from a coherent $\OO_S$-sheaf
$E$ by its standard resolution. The sheaf $\widetilde E$ is
stable (resp., semistable) on the scheme
 $\widetilde S$ if and only if the sheaf $E$ is stable (resp., semistable).
\end{proposition}
\begin{proof} Let  $E$ be  a Gieseker-semistable sheaf on
$(S,L)$ and $\widetilde E$ be a locally free sheaf on the scheme
$\widetilde S$. Let $\widetilde E$ be obtained from $E$ by the 
standard resolution.  Consider a proper subsheaf  $\widetilde
F \subset \widetilde E.$ Since  $m\gg 0$ we assume that both the
sheaves  $\widetilde E \otimes \widetilde L^{m}$ and  $\widetilde
F\otimes \widetilde L^{m}$ are globally generated. Fix an
epimorphism $H^0(\widetilde S, \widetilde E \otimes \widetilde
L^{m}) \otimes \widetilde L^{(-m)} \twoheadrightarrow \widetilde
E$. The subsheaf~$\widetilde F$ is generated by a subspace of
global sections $V_{\widetilde F}=H^0(\widetilde S, \widetilde F
\otimes \widetilde L^{m})\subset H^0(\widetilde S, \widetilde E
\otimes \widetilde L^{m}).$ Then a subspace $V_F\subset H^0(S,
E\otimes L^{m})$ which is isomorphic to $V_{\widetilde F}$ and
generates some subsheaf $F\in E$, is given by the 
isomorphism $\upsilon \colon H^0(\widetilde S, \widetilde E \otimes
\widetilde L^{m}) \stackrel{\sim}{\to} H^0(S,E\otimes L^{m})$ from Proposition \ref{sectionsrel} by
the equality $V_F=\upsilon (V_{\widetilde F})$. Since the sheaves
 $\widetilde F$ and $F$ are canonically isomorphic on 
 corresponding open subsets of schemes $\widetilde S$ and $S$,
 then their ranks are equal. Clearly,  $V_F= H^0(S, F\otimes L^{ m})$
 and
\begin{eqnarray*}\frac{h^0(\widetilde S,
\widetilde E \otimes \widetilde L^{m})}{r}&-& \frac{h^0(\widetilde
S,
\widetilde F \otimes \widetilde L^{m})}{r'}\nonumber \\
=\frac{h^0(S, E \otimes L^{m})}{r}&-& \frac{h^0( S, F \otimes
L^{m})}{r'}>(\ge) 0.\nonumber
\end{eqnarray*}
This implies semistability of $\widetilde E.$ The opposite
implication is proven similarly.
\end{proof}
\begin{remark} \label{sscr} This shows that there is a bijection among
subsheaves of $\OO_S$-sheaf $E$ and subsheaves of the
corresponding $\OO_{\widetilde S}$-sheaf $\widetilde E.$ This
bijection preserves Hilbert polynomials.
\end{remark}
%%%%%%%%%%%%%%%%%%%%%%%%%%%%%%%%%%%%%%%%%%%%%%%%%%%%%%%%%%%%%%%%%
\section{An equivalence of semistable admissible pairs}\label{s4}
In this section we introduce and examine an equivalence relation for admissible pairs of an arbitrary dimension. This equivalence generalizes S-equivalence for semistable coherent sheaves in the Gieseker--Maruyama moduli problem and is an analog of M-equivalence (monoidal equivalence) for the case $\dim S=2$ \cite{Tim4}. The major part of proofs concerning the equivalence are very similar to the ones in \cite{Tim4} but all the consideration is included here for the convenience of the reader. 
Since the class of admissible schemes in the case of $\dim S >2$ is sufficiently wider then the class of admissible schemes for the case of $\dim S=2$, the notion of  M-equivalence for higher dimensions is strengthened. 

Now we need some additional construction. Let $\sigma\colon \widetilde S \to S$ and $\sigma'\colon \widetilde S' \to S$ be two admissible schemes with their structure morphisms.

In this section we investigate the behaviour of a Jordan--H\"{o}lder 
filtration for a semistable coherent sheaf under the
standard resolution. Also a notion of M-equi\-val\-ence for
semi\-stable pairs is introduced and a relation of M-equivalence of resolutions to
S-equival\-ence for initial semistable coherent sheaves is examined. In
particular, it is proven that S-equivalent coherent sheaves on the
variety $S$ lead to M-equivalent pairs of the form
$((\widetilde S, \widetilde L), \widetilde E)$ when resolved.

Recall some classical notions from the theory of semistable coherent
sheaves.
\begin{definition}[\cite{HL}, Definition 1.5.1]
A {\it Jordan--H\"{o}lder filtration} for a semistable sheaf
 $E$ with a reduced Hilbert polynomial $p_E(n)$ on the
 polarized projective scheme  $(X,L)$ is a sequence of subsheaves
 \begin{equation*} 0=F_0\subset F_1
\subset \dots \subset F_s=E
\end{equation*}
such that quotient sheaves  $\gr_i(E)=F_i / F_{i-1}$, $i=1, \dots, s$, are stable
with reduced Hilbert polynomials equal to $p_E(n)$.
\end{definition}
Denote  $\gr(E):=\bigoplus \gr_i (E_i)$. The well-known theorem \cite[Prop.
1.5.2]{HL} claims that the isomorphism class of the sheaf
 $\gr(E)$ has no dependence on a choice of the Jordan--H\"{o}lder
 filtration of $E$.
\begin{definition}[\cite{HL}, Definition 1.5.3] Semistable
sheaves $E$ and $E'$ are called  {\it S-equivalent} if
$\gr(E)=\gr(E').$
\end{definition}
\begin{remark} Obviously, S-equivalent stable sheaves are
isomorphic.\end{remark}

Define a Jordan--H\"{o}lder filtration for an $S$-semistable sheaf on
a reducible admissible polarized scheme $(\widetilde S, \widetilde
L)$. This definition will be completely analogous to the classical
definition for a Gieseker-semistable sheaf.

\begin{definition} A {\it Jordan--H\"{o}lder filtration}
for a sheaf $\widetilde E$ on the polarized project\-ive reducible
scheme $(\widetilde S, \widetilde L)$ such that a pair
$((\widetilde S, \widetilde L), \widetilde E)$ is semistable in
the sense of Definition \ref{semistable}, and with reduced Hilbert
polynomial  $p_{E}(n),$ is a sequence of subsheaves
\begin{equation*} 0=\widetilde F_{0}\subset \widetilde F_1
\subset \dots \subset \widetilde F_{\ell}=\widetilde E,
\end{equation*}
such that the quotients  $\gr_i(\widetilde E)=\widetilde F_i
/\widetilde F_{i-1}$ are Gieseker-stable with reduced Hilbert
polynomials equal to $p_E(n)$.
\end{definition}

Example 2 in \cite{Tim4} shows that S-equivalent coherent sheaves can
have different associated sheaves of Fitting ideals yielding
non-iso\-morphic schemes $\widetilde S$ (even in the case of  $\hd E=1$).

Example 3 in \cite{Tim4} shows that  fibered product cannot be used
to construct a notion of equivalence for semistable pairs.

Since a resolution can consist of several morphisms $\sigma_i \colon S_i \to S_{i-1}$, $i=1,\dots, \ell,$ followed by each other,  we develop the notion of equivalence which is an analog  of S-equivalence in classical sheaf moduli theory. Also our notion of equivalence is a straightforward  generalization of M-equivalence of semistable admissible pairs in the case of $\hd E=1$. It is enough to analyze one step of the resolution, i.e. one morphism $\sigma_i$. The construction of an admissible scheme provides all $S_i$ being equidimensional schemes. 

Now we restrict ourselves by one step but all the construction holds when $\sigma$'s are iterated because the reasoning below does not depend on nonsingularity, irreducibility and reducedness of the scheme $S$ (it can be replaced by $S_i$, $i=2,\dots, \ell-1$).

Now consider schemes \begin{eqnarray*}\widetilde S_1&=&\Proj \bigoplus _{s\ge
0}(I_1[t]+(t))^s/(t)^{s+1},\\\widetilde S_2&=& \Proj
\bigoplus _{s\ge 0}(I_2[t]+(t))^s/(t)^{s+1} \end{eqnarray*} with their canonical
morphisms $\sigma _1\colon \widetilde S_1 \to S$ and $\sigma_2\colon 
\widetilde S_2 \to S$ to the scheme  $S$. Form inverse images of
sheaves of ideals $I'_2=\sigma_1^{-1} I_2 \cdot \OO_{\widetilde
S_1}\subset \OO_{\widetilde S_1}$ and $I'_1=\sigma_2^{-1} I_1
\cdot \OO_{\widetilde S_2}\subset \OO_{\widetilde S_2}$, and the corresponding
projective spectra \begin{eqnarray*}\widetilde S_{12}&=&\Proj (\bigoplus _{s\ge
0}(I'_2[t]+(t))^s/(t)^{s+1}),\\
\widetilde S_{21}&=&\Proj (\bigoplus _{s\ge 0}(I'_1[t]+(t))^s/(t)^{s+1}).\end{eqnarray*} There are
canonical morphisms  $\sigma'_2\colon\widetilde S_{12}\to \widetilde
S_1$ and $\sigma'_1\colon\widetilde S_{21}\to \widetilde S_2$.

\begin{proposition}\label{coversch} $\widetilde S_{12}$ and $\widetilde
S_{21}$ are equidimensional schemes. Moreover, $\widetilde S_{12}
\cong \widetilde S_{21}.$
\end{proposition}
\begin{proof}
First we prove that $\widetilde S_{12}\cong \widetilde S_{21}$, и
and that these schemes can be included into flat families with a
general fiber  isomorphic to $S$, or to $\widetilde S_1$, or to
 $\widetilde S_2$. This implies that all the components of the scheme
$\widetilde S_{12}$ have dimension not bigger then $\dim S$. Then we will
give the scheme-theoretic characterization of the scheme $\widetilde
S_{12}$. It proves that  $\widetilde S_{12}$ is an equidimensional
scheme. Namely, all reduced schemes corresponding to its
components have dimension equal to $\dim S$.

Let $T= \Spec k[t].$ Turn to the trivial 2-parameter family of
schemes  $T\times T \times S$ with natural projections $T \times S
\stackrel{p_{13}}{\longleftarrow}T\times T\times S
\stackrel{p_{23}}{\longrightarrow} T\times S$ onto the product of $1^{\mathrm{st}}$ and $3^{\mathrm{rd}}$ factors  and onto the product of $2^{\mathrm{nd}}$ and $3^{\mathrm{rd}}$ factors respectively. Introduce the
notations $\I_1:=\OO_T \boxtimes I_1 \subset \OO_{T \times S},$
$\I_2:= \OO_T \boxtimes I_2\subset \OO_{T\times S}$. Form inverse
images  $p_{13}^{\ast }\I_1$ and $p_{23}^{\ast} \I_2$. These are
sheaves of ideals on the scheme $T\times T\times S.$ Consider the
morphism $$\sigma \!\!\! \sigma_1 \times \id_T \colon \widehat \Sigma_1
\times T \to T\times T \times S$$ with an identity map on the second
factor. Form a sheaf of ideals $(\sigma \!\!\! \sigma_1 \times
\id_T)^{-1} p_{23}^{\ast}\I_2 \cdot \OO_{\widehat \Sigma_1 \times
T}$  and consider the corresponding morphism of blowing up $\sigma\!\!\! \sigma_{12} \colon
\Sigma \!\!\! \Sigma_{12} \to \widehat \Sigma_1 \times T$. Now
restrict the sheaf $(\sigma \!\!\! \sigma_1 \times \id_T)^{-1}
p_{23}^{\ast}\I_2 \cdot \OO_{\widehat \Sigma_1 \times T}$ on the
fiber  of the composite map $$\widehat \Sigma_1 \times T
\stackrel{\sigma\!\!\! \sigma_1 \!\times
\id_T}{-\!\!-\!\!\!\longrightarrow} T\times T\times S
\stackrel{p_{12}}{\longrightarrow} T\times T$$ over a closed point~$(t_1,
t_2).$ Let $\widetilde i \colon \widetilde S_1 \hookrightarrow \widehat
\Sigma_1 \times T$ be the morphism of the closed immersion of this fiber.
Commutativity of the diagram
\begin{equation*}\xymatrix{\widehat \Sigma_1 \times T
\ar[rr]^{\sigma \!\!\!\sigma_1 \times \id_T} &&T\times T\times S
\ar[rr]^{p_{12}}&& T\times T\\
\widetilde S_1 \ar@{^(->}[u]^{\widetilde i} \ar[rr]^{\sigma_1}&& S
\ar@{^(->}[u]_i \ar[rr]&& (t_1, t_2)\ar@{^(->}[u]}
\end{equation*}
leads to the equalities $\widetilde i^{-1}((\sigma \!\!\! \sigma_1 \times
\id_T)^{-1} p_{23}^{\ast} \I_2 \cdot \OO_{\widehat \Sigma_1 \times
T})\cdot \OO_{\widetilde S_1}=\sigma_1^{-1}
i^{-1}(p_{23}^{\ast}\I_2)\cdot \OO_{\widetilde S_1} =\sigma_1^{-1}
I_2 \cdot \OO_{\widetilde S_1}$.

Now consider an immersion of the line $j_T\colon T \hookrightarrow T
\times T$ which is fixed by the equation $at_1+bt_2+c=0$, $a,b,c \in k$.
The corresponding fibered diagram
\begin{equation}\label{famil}\xymatrix{\Sigma \!\!\! \Sigma_{12}
\ar[r]^{\!\!\sigma \!\!\! \sigma_{12}\;\;}&\widehat \Sigma_1\times
T \ar[rr]^{\sigma \!\!\! \sigma_1 \times \id_T} &&T\times T\times
S
\ar[r]^{\;\;\;\;p_{12}}&T\times T\\
\Sigma \!\!\! \Sigma_{12j} \ar@{^(->}[u]^{j_{12}} \ar[r]&\Sigma_{1j}
\ar@{^(->}[u]_{j_1} \ar[rr]^{\sigma \!\!\! \sigma'_1}&& T\times S
\ar@{^(->}[u]^{j_T \times  \id_S} \ar[r]&T\ar@{^(->}[u]_{j_T}}
\end{equation}
fixes notations. If the immersion $j_T$ does not correspond to the
case $b=0$, then $\Sigma_{1j}\simeq \widehat \Sigma _1$ and \linebreak
$j_1^{-1}((\sigma \!\!\! \sigma_1 \times \id_T)^{-1} p_{23}^{\ast}
\I_2 \cdot \OO_{\widehat \Sigma \times T})\cdot \OO_{\Sigma_{1j}}=
\sigma \!\!\! \sigma_1^{-1}\I_2 \cdot \OO_{\Sigma_1}.$ Otherwise
 (for $b=0$) we have $\Sigma_{1j}\cong T \times S.$

The morphism $\sigma_{1j} \colon \widehat \Sigma_{1j} \to \Sigma_{1j}$
of the blowing up of the sheaf of ideals $\sigma \!\!\!
\sigma_1^{-1} \I_2 \cdot \OO_{\Sigma_1}$ is included into the
commutative diagram
\begin{equation*}\xymatrix{\Sigma \!\!\! \Sigma_{12}
\ar[rr]^{\sigma\!\!\! \sigma_{12}}&& \widehat \Sigma_1 \times T\\
\widehat \Sigma_{1j} \ar@{^(->}[u] \ar[rr]^{\sigma_{1j}} && \Sigma_{1j}
\ar@{^(->}[u]_{j_1}}
\end{equation*}
By the universal property of the left hand side fibered product in
(\ref{famil}) there is a scheme morphism  \linebreak $u\colon \widehat \Sigma_{1j}\to
\Sigma \!\!\! \Sigma_{12j}$.

The morphism of blowing up  $\sigma \!\!\! \sigma'_2\colon\widehat
\Sigma_{12} \to \widehat \Sigma_1$ of the sheaf of ideals $\sigma
\!\!\! \sigma_1^{-1}\I_2 \cdot \OO_{\widehat \Sigma_1}$ is included
into the commutative diagram
\begin{equation}\label{gldiamond}\xymatrix{\widehat \Sigma_{12}
\ar[r]^{\sigma \!\!\! \sigma'_1}\ar[d]_{\sigma \!\!\! \sigma'_2}
& \widehat \Sigma_2 \ar[d]^{\sigma \!\!\! \sigma_2}\\
\widehat \Sigma_1 \ar[r]^{\sigma \!\!\! \sigma_1}& T\times S}
\end{equation}
Note that in this diagram $\sigma\!\!\!\sigma'_1$ is a morphism of
blowing up of the sheaf of ideals $\sigma \!\!\! \sigma_2^{-1}\I_1
\cdot \OO_{\widehat \Sigma_2}$ and it follows that  $\widehat
\Sigma_{21} =\widehat \Sigma_{12}$. Also $\widehat \Sigma_1,$
$\widehat \Sigma_2$, and $\widehat \Sigma_{12}$ are the schemes which are obtained from $\Sigma$ by one or two consequent blowing up morphisms  and then each of them consists of the same number of components as $\Sigma$ does. Also if $\Sigma'_i$ denotes the scheme which is obtained by removing  the fiber  containing an exceptional divisor of corresponding $\sigma\!\!\!\sigma_i,$ $i=1,2,$  then $\widehat \Sigma_i:= \overline{\Sigma'_i}$. The closure in the latter expression is taken in a $T$-based relative projective space $\P_{T,i}$ with a closed immersion $T$-morphism  $\P_{T,i}\hookleftarrow \widehat \Sigma_i$. Then $\widehat \Sigma_i$ is a $T$-flat scheme. The same reasoning and the same conclusion are true for $\widehat \Sigma_{12}$. Each of the schemes $\widehat \Sigma_1,$ $\widehat \Sigma_2,$  and $\widehat \Sigma_{12}$ is fibered over the regular one-dimensional base $T$ with fibers being isomorphic to projective schemes. Hence \cite[ch. III, Proposition 9.8]{Hart} the schemes $\widehat \Sigma_1,$ $\widehat \Sigma_2$, and $\widehat \Sigma_{12}$ are flat families of projective schemes over $T$. Each of these families has a fiber  which is isomorphic to the
scheme $S$ at general enough closed point of $T$. This implies that
each fiber  of the family $\widehat \Sigma _{12}$ has a form of
projective spectrum $\Proj \bigoplus_{s\ge
0}(I[t]+(t))^s/(t)^{s+1}$ for an appropriate sheaf of ideals
$I\subset \OO_{S}$. Fibers of a flat family of projective schemes
admit polarizations with the following property. Hilbert polynomials
of fibers computed with respect to these polarizations remain
uniform over the base. By the construction, such polarizations on
fibers of schemes $\widehat \Sigma_1,$ $\widehat \Sigma_2$, and
$\widehat \Sigma_{12}$ are exactly the same as polarizations
computed in (\ref{polarizfam}), (\ref{polarizfib}).

Now we prove that  $\Sigma \!\!\!\Sigma_{12j}$ is a family of
schemes which is flat over $T$. Consider the exact $\OO_{\widehat \Sigma_1
\times T}$-triple induced by the sheaf of ideals $(\sigma \!\!\!
\sigma_1 \times \id_T)^{-1} p_{23}^{\ast}\I_2 \cdot \OO_{\widehat
\Sigma_1 \times T}$:
$$
0\to (\sigma \!\!\! \sigma_1 \times \id_T)^{-1} p_{23}^{\ast}\I_2
\cdot \OO_{\widehat \Sigma_1 \times T} \to \OO_{\widehat \Sigma_1
\times T} \to \OO_Z \to 0
$$
for the corresponding closed subscheme  $Z$. Apply 
$j_1^{\ast}$ and note that there is an isomorphism \linebreak $j_1^{-1}((\sigma
\!\!\! \sigma_1 \times \id_T)^{-1} p_{23}^{\ast}\I_2 \cdot
\OO_{\widehat \Sigma_1 \times T}) \cdot \OO_{\Sigma_{1j}}\cong j_1^{\ast}(\sigma \!\!\!
\sigma_1 \times \id_T)^{-1} p_{23}^{\ast}\I_2 \cdot \OO_{\widehat
\Sigma_1 \times T}/\tors,$ where the torsion subsheaf is given by the
equality $$\tors= \TTor_1^{j_1^{-1}\OO_{\widehat \Sigma_1 \times
T}} (j_1^{-1} \OO_Z, \OO_{\Sigma_{1j}}).$$ Note that
$\Sigma_{1j}\cong \widehat \Sigma_1$ and $j_1^{\ast}\OO_{\widehat
\Sigma_1 \times T}\cong \OO_{\Sigma_{1j}}.$ With the latter two
isomorphisms taken into account we have 
$\TTor_1^{j_1^{-1} \OO_{\widehat \Sigma_1 \times T}}(j_1^{-1}
\OO_Z, \OO_{\Sigma_{1j}})=
\TTor_1^{\OO_{\Sigma_{1j}}}(j_1^{-1}\OO_Z, \OO_{\Sigma_{1j}})=0.$
Then  $$j_1^{\ast}(\sigma \!\!\! \sigma_1 \times \id_T)^{-1}
p_{23}^{\ast}\I_2 \cdot \OO_{\widehat \Sigma_1 \times T}=
j_1^{-1}((\sigma \!\!\! \sigma_1 \times \id_T)^{-1}
p_{23}^{\ast}\I_2 \cdot \OO_{\widehat \Sigma_1 \times T}) \cdot
\OO_{\Sigma_{1j}}=\sigma \!\!\! \sigma_1^{-1} \I_2 \cdot
\OO_{\Sigma_1}.$$ Also for the blowing up morphisms one has $$\Sigma \!\!\!
\Sigma_{12j}=\Proj \bigoplus_{s\ge 0} (j_1^{\ast}(\sigma \!\!\!
\sigma_1 \times \id_T)^{-1} p_{23}^{\ast}\I_2 \cdot \OO_{\widehat
\Sigma_1 \times T})^s= \Proj \bigoplus_{s\ge 0} (\sigma \!\!\!
\sigma_1^{-1} \I_2 \cdot \OO_{\Sigma_1})^s=\widehat \Sigma_{12}.$$
Since $\widehat \Sigma_{12}$ is a flat family over $T$ then the
scheme $\Sigma \!\!\! \Sigma_{12j}$ is also flat over $T$.

Any two points on  $T\times T$ can be connected by a chain of two
lines satisfying the condition $b\ne 0.$ Then Hilbert polynomials
of fiber s of the scheme $\Sigma \!\!\! \Sigma_{12} \to T\times T$
are constant over the base  $T\times T$ which is an integral scheme. Hence the scheme
$\Sigma\!\!\! \Sigma_{12}$ is flat over the base $T\times T.$

To characterize the scheme structure of the special fiber  of the
scheme $\Sigma_{12}$ (and consequently the corresponding fiber  of
the scheme  $\Sigma \!\!\! \Sigma_{12}$) it is enough to consider
the embedding  $j_T$ defined by the equation $t_2=0,$ and a
subscheme $\widetilde \Sigma _1=j_1(\Sigma_{1j})$. It is a flat
family of subschemes with fiber  $\widetilde S_1 =\Proj
\bigoplus_{s\ge 0} (I_1[t]+(t))^s/(t)^{s+1}$. As proven before,
the preimage  $\widetilde \Sigma_{12}= \sigma \!\!\!
\sigma_{12}^{-1}(\widetilde \Sigma_1)$ is also flat over $j_T(T)
\cong T$ with generic fiber  isomorphic to $\widetilde S_1=\Proj
\bigoplus_{s\ge 0} (I_1[t]+(t))^s/(t)^{s+1}$. Applying the reasoning 
of the article \cite{Tim4} in this situation, we obtain that
the special fiber $\widetilde S_{12}$ of the scheme  $\widetilde
\Sigma_{12}$ has the following scheme-theoretic characterization:
$\widetilde S_{12}= \Proj \bigoplus_{s\ge
0}(I'_2[t]+(t))^s/(t)^{s+1}$ for the sheaf of ideals $I'_2 \subset
\OO_{\widetilde S_1}$ defined as $I'_2=\sigma_1^{-1}I_1 \cdot
\OO_{\widetilde S_1}.$
\end{proof}

Hence, for any two schemes $$\widetilde S_1=\Proj \bigoplus_{s\ge
0}(I_1[t]+(t))^{s}/(t)^{s+1}$$ and $$\widetilde S_2=\Proj
\bigoplus_{s\ge 0}(I_2[t]+(t))^{s}/(t)^{s+1}$$ the scheme
$\widetilde S_{12}=\Proj \bigoplus_{s\ge
0}(I'_1[t]+(t))^{s}/(t)^{s+1}=\Proj \bigoplus_{s\ge
0}(I'_2[t]+(t))^{s}/(t)^{s+1}$ is defined. It has two  morphisms
$\widetilde S_1\stackrel{\sigma'_1}{\longleftarrow} \widetilde
S_{12} \stackrel{\sigma'_2}{\longrightarrow} \widetilde S_2$ such
that the diagram
\begin{equation*}\xymatrix{\widetilde S_{12}
\ar[r]^{\sigma'_2}\ar[d]_{\sigma'_1}& \widetilde S_2
\ar[d]^{\sigma_2}\\
\widetilde S_1 \ar[r]_{\sigma_1}&S}
\end{equation*}
commutes. Since in our (higher-dimensional) case the resolution consists of several 
consequent morphisms $\sigma_i$, then Proposition \ref{coversch} is to be applied iteratively.
A binary operation $(\widetilde S_1, \widetilde S_2) \mapsto
\widetilde S_1 \diamond \widetilde S_2=\widetilde S_{12}$ defined
in this fashion is obviously associative. Moreover,  for any
morphism $\sigma \colon \widetilde S\to S$ there are
equalities $\widetilde S \diamond S =S \diamond \widetilde S
=\widetilde S$. Then canonical morphisms $\sigma \colon \widetilde S \to S$ 
arising in resolution of sheaves in each class $[E]$ of
S-equivalent semistable coherent sheaves generate a commut\-ative
monoid  $\diamondsuit[E]$ with the binary operation $\diamond$ and the 
neutral element $\id_S \colon S \to S.$

Note that by Proposition  \ref{ssc} and Remark \ref{sscr} there is
a bijective correspondence among subsheaves of a coherent
$\OO_S$-sheaf $E$ and subsheaves of the corresponding locally free
$\OO_{\widetilde S}$-sheaf $\widetilde E$ arising in the resolution of $E$. 
This correspondence pre\-serves Hilbert polynomials. Let there is a fixed
Jordan-H\"{o}lder filtration in $E$ formed by subsheaves  $F_i$.
Then there is a sequence of semistable subsheaves $\widetilde F_i$
with the same reduced Hilbert polynomial and $\rank \widetilde
F_i=\rank F_i$ distinguished in $\widetilde E$ by the described
correspondence.

Let $X$ be a projective scheme, $L$ an ample invertible
$\OO_X$-sheaf, $E$  a coherent $\OO_X$-sheaf. Let the sheaf
$E\otimes L^{m}$ be globally generated, namely, there is an evaluation
epimorph\-ism $q\colon H^0(X, E\otimes L^{m}) \otimes L^{ (-m)}
\twoheadrightarrow E.$ Fix a subspace $H \subset H^0(X, E\otimes
L^{m}).$ A subsheaf $F\subset E$ is said to be {\it generated by
the subspace $H$} if it is an image of the composite map  $H
\otimes L^{(-m)}\subset H^0(X, E\otimes L^{ m}) \otimes L^{
(-m)}\stackrel{q} {\twoheadrightarrow }E.$

\begin{proposition} \label{corr} The transformation $E \mapsto
\sigma^{\ast}E /\tors$ is compatible on all the sub\-sheaves $F \subset
E$ with the isomorphism  $\upsilon \colon H^0(\widetilde S, \widetilde E \otimes
\widetilde L^{m}) \stackrel{\sim}{\to} H^0(S,E\otimes L^{m})$ for all $m\gg 0.$
\end{proposition}
\begin{proof} Take an arbitrary subsheaf $F\subset E$
of rank $r'$. It is necessary to check that the subsheaf $\widetilde
F\subset \widetilde E=\sigma^{\ast}E /\tors \!$, which is generated in
 $\widetilde E$ by the subspace $\upsilon^{-1}
H^0(S, F\otimes L^{m})$, coincides with the subsheaf
$\sigma^{\ast}F /\tors \!.$

It is known that the sheaf $\sigma^\ast E/\tors$ is generated by the vector space $H^0(\widetilde S, \widetilde E \otimes L^m)=\upsilon^{-1}H^0(S, E\otimes L^m)$, i.e. there is an epimorphism of $\OO_{\widetilde S}$-modules 
$$
\upsilon^{-1}H^0(S, E\otimes L^m) \otimes \widetilde L^{-m} \twoheadrightarrow \sigma^\ast E/\tors \!.
$$

Also we  know that the subsheaf $\sigma^{\ast} F/\tors \!\subset \sigma^\ast E/\tors$ is included into the commutative diagram 
\begin{equation}\label{d1}
\xymatrix{\upsilon^{-1}H^0(S, E\otimes L^m) \otimes \sigma^\ast L^{-m} \ar@{->>}[r] &\sigma^\ast E/\tors\\
\upsilon^{-1}H^0(S, F\otimes L^m) \otimes \sigma^\ast L^{-m} \ar@{^(->}[u]\ar@{->>}[r] &\sigma^\ast F/\tors \ar@{^(->}[u]}
\end{equation}
as well as the subsheaf $\widetilde F \subset \widetilde E$ is included into the commutative diagram 
\begin{equation}\label{d2}
\xymatrix{\upsilon^{-1}H^0(S, E\otimes L^m) \otimes \widetilde L^{-m} \ar@{->>}[r] &\widetilde E\\
\upsilon^{-1}H^0(S, F\otimes L^m) \otimes \widetilde L^{-m} \ar@{^(->}[u]\ar@{->>}[r] &\widetilde F \ar@{^(->}[u]}
\end{equation}
Combining (\ref{d1},\ref{d2}) with an inclusion of invertible sheaves $\sigma^\ast L^{-m} \subset \widetilde L^{-m}$, we come to the inclusion $\sigma^\ast F/\tors \subset \widetilde F$ and the commutative diagram
\begin{equation}\label{*}
\xymatrix{&\ar@{^(->}[ld]\upsilon^{-1}H^0(S, E\otimes L^m) \otimes \sigma^\ast L^{-m} \ar@{->>}[rr] &&\ar@{=}[ld]\sigma^\ast E/\tors\\
\upsilon^{-1}H^0(S, E\otimes L^m) \otimes \widetilde L^{-m} \ar@{->>}[rr] &&\widetilde E\\
&\ar@{^(->}[ld]\upsilon^{-1}H^0(S, F\otimes L^m) \otimes \sigma^\ast L^{-m} \ar@{^(->}[uu]\ar@{->>}[rr] &&\ar@{^(->}[ld]\sigma^\ast F/\tors \ar@{^(->}[uu]\\
\upsilon^{-1}H^0(S, F\otimes L^m) \otimes \widetilde L^{-m} \ar@{^(->}[uu]\ar@{->>}[rr] &&\widetilde F \ar@{^(->}[uu]}    
\end{equation}

Now consider global sections of sheaves $\sigma^\ast F/\tors \otimes \widetilde L^s$ and $\widetilde F \otimes \widetilde L^s$ for $s\gg 0$. They coincide along any open subscheme  $\widetilde U \subset \widetilde S$ such that $\sigma \colon \widetilde S \to S$
becomes an isomorphism $\sigma|_{\widetilde U}\colon\widetilde U \stackrel{\sim}{\longrightarrow} U$ being restricted to $\widetilde U$. Then we have an isomorphism $$H^0(\widetilde S^0, \sigma^\ast F/\tors \otimes \widetilde L^s|_{\widetilde S^0}) \cong H^0(\widetilde S^0, \widetilde F \otimes \widetilde L^s|_{\widetilde S^0})$$ for all $s\gg 0$. Hence, an inclusion $\sigma^{0\ast} F/\tors \subset \widetilde F|_{\widetilde S^0}$ is an inclusion of sheaves with equal Hilbert polynomials. From this we conclude that the sheaves $\sigma^{0\ast} F/\tors$ and $\widetilde F|_{\widetilde S^0}$ are equal along the principal component $\widetilde S^0$, i.e. 
$\sigma^{0\ast}F/\tors =\widetilde F|_{\widetilde S^0}$.
Also we have an analog of the commutative  diagram (\ref{*}) for the principal component~$\widetilde S^0$:
\begin{equation}\label{**}
\xymatrix{&\ar@{^(->}[ld]\upsilon^{-1}H^0(S, E\!\otimes \!L^m) \!\otimes \!\sigma^{0\ast} L^{-m} \ar@{->>}[rr] &&\ar@{=}[ld]\sigma^{0\ast} E/\tors\!\\
\upsilon^{-1}H^0(S, E\!\otimes \! L^m) \!\otimes \!\widetilde L^{-m}|_{\widetilde S^0} \ar@{->>}[rr] &&\widetilde E|_{\widetilde S^0}\\
&\ar@{^(->}[ld]\upsilon^{-1}H^0(S, F\!\otimes \!L^m) \!\otimes \!\sigma^{0\ast} L^{-m} \ar@{^(->}[uu]\ar@{->>}[rr] &&\ar@{=}[ld]\sigma^{0\ast} F/\tors \! \ar@{^(->}[uu]\\
\upsilon^{-1}H^0(S, F\!\otimes \!L^m)\! \otimes \!\widetilde L^{-m}|_{\widetilde S^0} \ar@{^(->}[uu]\ar@{->>}[rr] &&\widetilde F|_{\widetilde s^0} \ar@{^(->}[uu]}    
\end{equation}
Now  we take into account that $\widetilde E=\delta^\ast (\widetilde E|_{\widetilde S^0})$, bring together the evaluation morphism $$\upsilon ^{-1} H^0(S, E\otimes L^m)\otimes \widetilde L^{-m} \twoheadrightarrow \widetilde E$$ and the equality $\delta^\ast (\sigma^{0\ast} F/\tors)=\sigma^\ast F/\tors$, and then we come from (\ref{**}) to the diagram:
\begin{equation}\label{***} \xymatrix{&\ar@{^(->}[ld]\upsilon^{-1}H^0(S, E\otimes L^m) \otimes \sigma^{\ast} L^{-m} \ar@{->>}[rr] &&\ar@{=}[ld]\widetilde E\\
\upsilon^{-1}H^0(S, E\otimes L^m) \otimes \widetilde L^{-m} \ar@{->>}[rr] &&\widetilde E\\
&\ar@{^(->}[ld]\upsilon^{-1}H^0(S, F\otimes L^m) \otimes \sigma^\ast L^{-m} \ar@{^(->}[uu]\ar@{->>}[rr] &&\ar@{=}[ld]\sigma^{\ast} F/\tors \ar@{^(->}[uu]\\
\upsilon^{-1}H^0(S, F\otimes L^m) \otimes \widetilde L^{-m} \ar@{^(->}[uu]\ar@{->>}[rr] &&\delta^\ast (\widetilde F|_{\widetilde S^0}) \ar@{^(->}[uu]} 
\end{equation}
From the left hand side rectangle of (\ref{***}) it follows that $\delta^\ast \widetilde F|_{\widetilde S^0}$ is the image of the sheaf \linebreak $\upsilon^{-1}H^0(S, F\otimes L^m) \otimes \widetilde L^{-m}$ under the evaluation morphism generating $\widetilde E$ but this image is exactly $\widetilde F,$ i.e. $\delta^\ast (\widetilde F|_{\widetilde S^0})=\widetilde F.$ This leads to the equality $\widetilde F = \sigma^\ast F/\tors$ as required. This completes the proof.
\end{proof}
We return to a Jordan--H\"{o}lder filtration $$ 0=F_0 \subset F_1 \subset \dots \subset F_s =E
$$
of a semistable sheaf $E$.

From now $\sigma$ is understood as a composite morphism $\sigma=\sigma_1\circ \dots \sigma_\ell$, i.e. it is a structure morphism of an admissible scheme $\widetilde S$.
\begin{corollary} \label{sssh} Each of the sheaves $\widetilde F_j=\sigma^{\ast}
F_j /\tors,$ $j=1,\dots,s,$ is semistable of rank $r_j=\rank F_j$, with reduced
Hilbert polynomial equal to $p_E(n)$.
\end{corollary}
\begin{proof} The equality of the ranks follows from the equality
$\widetilde F_j=\sigma^{\ast} F_j /\tors$ and from the fact that
the morphism  $\sigma$ is an isomorphism on some open subscheme in
$\widetilde S$. The Hilbert polynomial for all $n\gg 0$ is fixed
by the equalities $\chi (\widetilde F_j \otimes \widetilde L^{n})=
h^0(\widetilde S, \widetilde F_j \otimes \widetilde L^{ n})=
h^0(S, F_j \otimes L^{ n} )= \chi (F_j \otimes L^{ n})=r_j p_E(n).$
\end{proof}

For $\widetilde E= \sigma^{\ast} E/\tors$ consider epimorphisms
$\widetilde q \colon \!H^0(\widetilde S, \widetilde E \otimes \widetilde
L^m) \!\otimes \! \widetilde L^{(-m)}\!\twoheadrightarrow \!\widetilde E$
and  $q\colon \!H^0( S, E \otimes  L^m)\! \otimes\!
L^{(-m)}\!\twoheadrightarrow \!E$.
\begin{definition} Subsheaves  $\widetilde F \subset \widetilde E$ и $F \subset
E$ are called  {\it $\upsilon$-correspond\-ing} if there exist
subspaces  $\widetilde V \subset H^0(\widetilde S, \widetilde E
\otimes \widetilde L^m)$ and $V = \upsilon (\widetilde V)\subset
H^0( S, E \otimes L^m)$ such that  $\widetilde q(\widetilde V
\otimes \widetilde L^{-m})=\widetilde F$, $q(V \otimes L^{-m})=F$.
Notation: $F= \upsilon (\widetilde F).$ The corresponding quotient
sheaves  $\widetilde E /\widetilde F$ и $ E/F$ will be also called
$\upsilon$-{\it corresponding } and denoted
 $E/F=\upsilon(\widetilde E /\widetilde F).$\end{definition}

\begin{proposition} \label{satur} The transformation $E \mapsto
\sigma^{\ast}E /\tors$ takes saturated subsheaves to saturated
subsheaves.
\end{proposition}
\begin{proof} Let $F_{j-1}\subset F_j$ be a saturated subsheaf and let $F_{j-1}= \upsilon (\widetilde F_{j-1})$ and $F_j= \upsilon (\widetilde F_j)$.
Assume that the quotient sheaf $\widetilde F_j/ \widetilde
F_{j-1}$ has a subsheaf of torsion $\tau.$ This subsheaf is
generated by the vector subspace  $\widetilde T \subset H^0
(\widetilde S, \widetilde F_j \otimes \widetilde
L^{m})/H^0(\widetilde S, \widetilde F_{j-1} \otimes \widetilde
L^{m}).$ Let $\widetilde T'$ be its preimage in $H^0(\widetilde S,
\widetilde F_j \otimes \widetilde L^{m})$ and $\TT' \subset
\widetilde F_j$ be a subsheaf generated by the subspace $\widetilde
T'.$ Then there is a sheaf  epimorphism $\TT' \twoheadrightarrow
\tau$ with the kernel $\TT' \cap \widetilde F_{j-1}.$ Let $\widetilde
K= H^0(\widetilde S, (\TT' \cap \widetilde F_{j-1})\otimes
\widetilde L^{ m})\subset H^0(\widetilde S, \widetilde
F_{j-1}\otimes \widetilde L^{ m})$ be its generating subspace.
Then the isomorphism $\upsilon$ leads to the exact diagram of
vector spaces
\begin{equation*}\xymatrix{0\ar[r]& \upsilon(\widetilde K)
\ar[d]_{\wr}^{\upsilon} \ar[r]& \upsilon(\widetilde T')
\ar[d]_{\wr}^{\upsilon} \ar[r]& \upsilon(\widetilde
T')/\upsilon(\widetilde K)
\ar[d]_{\wr}^{\overline \upsilon} \ar[r]&0\\
0\ar[r]& \widetilde K \ar[r]& \widetilde T' \ar[r]& \widetilde T'
/\widetilde K \ar[r]&0}
\end{equation*}
with the morphism $\overline \upsilon$ induced by the morphism
$\upsilon.$ Also there are exact sequences of
$\upsilon$-corresponding coherent sheaves
\begin{equation*}\xymatrix{
0\ar[r]&\upsilon(\TT' \cap \widetilde F_{j-1}) \ar[r]&
\upsilon(\TT') \ar[r]& \upsilon(\tau) \ar[r]&0,\\
0\ar[r]& \TT' \cap \widetilde F_{j-1} \ar[r]& \TT' \ar[r]& \tau
\ar[r]& 0.}
\end{equation*}
The sheaves $\TT' \cap \widetilde F_{j-1},$ $\upsilon(\TT' \cap
\widetilde F_{j-1}),$ $\TT'$ and $\upsilon(\TT')$ coincide under
restriction on open subsets  $W$ and $\sigma (W)$ respectively where $W\subset \widetilde S$ is the maximal open subscheme such that  $\sigma|_W$ is an isomorphism.
Then if $\tau$ is a torsion sheaf, then $\upsilon (\tau)$ is also
torsion sheaf. This contradicts the assumption that  the
subsheaf $F_{j-1}$ is saturated.\end{proof}

\begin{corollary} There are isomorphisms $\sigma^{\ast}
(F_j/F_{j-1})/\tors \cong \widetilde F_j/ \widetilde F_{j-1}.$
\end{corollary}
\begin{proof} Take an exact triple
$$
0\to F_{j-1} \to F_j \to F_j/F_{j-1} \to 0
$$
and apply the functor $\sigma^{\ast}.$ This yields
$$
0 \to \sigma^{\ast} F_{j-1}/\tau \to \sigma^{\ast} F_j \to
\sigma^{\ast} (F_j/F_{j-1})\to 0$$ where the symbol $\tau$ denotes
the subsheaf of torsion violating exactness. Factoring  first two
sheaves by torsion and applying Proposition \ref{corr} one has
an exact diagram
\begin{equation*}\xymatrix{&0\ar[d]&0\ar[d]&0\ar[d]&\\
0\ar[r]& N \ar[d]\ar[r]& \tors(\sigma^{\ast}F_i)\ar[r] \ar[d] &\tau'\ar[r] \ar[d]& 0\\
0\ar[r]& \sigma^{\ast} F_{j-1}/\tau \ar[r]\ar[d]& \sigma^{\ast}
F_j\ar[r] \ar[d]&
\sigma^{\ast}(F_j /F_{j-1}) \ar[r] \ar[d]&0\\
0\ar[r]& \widetilde F_{j-1} \ar[r] \ar[d]& \widetilde F_j \ar[r]
\ar[d]& \widetilde F_j/
\widetilde F_{j-1} \ar[r] \ar[d]&0\\
&0&0&0&}
\end{equation*}
where the sheaf $N$ is defined as $\ker (\sigma^{\ast}F_{j-1}/\tau
\to \widetilde F_{j-1})$. It rests to note that the sheaf $\tau'$
is a torsion sheaf. Also since the subsheaf $F_{j-1}\subset F_j$ is
saturated then by  Proposition \ref{satur} the sheaf
$\widetilde F_j/ \widetilde F_{j-1}$ has no torsion. Then
$\widetilde F_j/\widetilde F_{j-1} \cong \sigma^{\ast}(F_j
/F_{j-1}) /\tors$.
\end{proof}
\begin{corollary} \label{stabquot} Quotient sheaves $\widetilde F_j/\widetilde
F_{j-1}$ are stable and their reduced Hilbert polynomial is equal
to $p_E(n).$
\end{corollary}
\begin{proof} Consider a subsheaf $\widetilde R \subset
\widetilde F_j/\widetilde F_{j-1}$ and the space of global
sections $$H^0(\widetilde S, \widetilde R \otimes \widetilde
L^{m})\subset H^0(\widetilde S, (\widetilde F_j/\widetilde
F_{j-1})\otimes \widetilde L^{m})= H^0(\widetilde S, \widetilde
F_j \otimes \widetilde L^{m})/ H^0(\widetilde S, \widetilde
F_{j-1} \otimes \widetilde L^{m}).$$ We assume as usually $m$ to
be so big as higher cohomology groups vanish. Let $\widetilde H$
be the preimage of the subspace $H^0(\widetilde S, \widetilde R
\otimes \widetilde L^{m})$ in $H^0(\widetilde S, \widetilde F_j
\otimes \widetilde L^{m})$, and $H:= \upsilon(\widetilde H)$.
Denote by $\HH$ a subsheaf in $F_j$ which is generated
by the subspace $H$. It is clear that $\HH /F_{j-1} \subset F_j
/F_{j-1}.$ From the chain of obvious equalities
\begin{eqnarray} h^0(\widetilde S, \widetilde R \otimes \widetilde L
^{m})=\dim \widetilde H - h^0(\widetilde S, \widetilde
F_{j-1}\otimes \widetilde L^{ m})=\dim H-h^0(S, F_{j-1} \otimes
L^{m})\nonumber\\
=h^0(S, \HH \otimes L^{m})- h^0(S, F_{j-1}\otimes L^{m})=h^0(S,
(\HH/F_{j-1}) \otimes L^{m})\nonumber \end{eqnarray} it follows
that
\begin{eqnarray*}\frac{h^0(\widetilde S, (\widetilde F_j
/\widetilde F_{j-1})\otimes \widetilde L^{m})} {\rank (\widetilde
F_j/\widetilde F_{j-1})}&-& \frac{h^0(\widetilde S,\widetilde R
\otimes \widetilde
L^{m})}{\rank \widetilde R}=\nonumber \\
\frac{h^0(S, (F_j/F_{j-1})\otimes L^{m})} {\rank (F_j/F_{j-1})}&-&
\frac{h^0(S,(\HH/F_{j-1}) \otimes L^{m})}{\rank (\HH/F_{j-1})} >0.
\end{eqnarray*}
This proves stability of the quotient sheaf $\widetilde F_j /
\widetilde F_{j-1}.$
\end{proof}

Now consider the exact triple $0\to F_1 \to E \to \gr_1(E) \to 0$
and the corresponding triple of spaces of global sections
\begin{equation*} 0\to H^0(S, F_1\otimes L^{m})\to H^0(S, E\otimes
L^{m})\to H^0(S, \gr_1(E)\otimes L^{m})\to 0.
\end{equation*}
It is also exact for $m\gg 0.$ Transition to the corresponding
 $\OO_{\widetilde S}$-sheaf $\widetilde
E$, to its subsheaf $\widetilde F_1$, to global sections, and
application of the isomorphism $\upsilon$ lead to the commutative
diagram of vector spaces
\begin{equation*} \xymatrix{
0\ar[r]& H^0(S, F_1\otimes L^{m})\ar[r]& H^0(S, E\otimes
L^{m})\ar[r]&
H^0(S, \gr_1(E)\otimes L^{m})\ar[r]& 0\\
0\ar[r]& H^0(\widetilde S, \widetilde F_1\otimes \widetilde
L^{m})\ar[r] \ar[u]_{\wr}^{\upsilon}& H^0(\widetilde S, \widetilde
E\otimes \widetilde L^{m})\ar[r]\ar[u]_{\wr}^{\upsilon}&
H^0(\widetilde S, \gr_1(\widetilde E)\otimes \widetilde L^{
m})\ar[r]\ar[u]_{\wr}^{\overline \upsilon}& 0}
\end{equation*}
where the isomorphism $\overline \upsilon$ is induced by the
isomorphism $\upsilon.$ Continuing the reasoning inductively for
other subsheaves of Jordan--H\"{o}lder filtration of the sheaf $E$,
we get that bijective correspondence of subsheaves is continued
onto quotients of filtrations. Then the transition from quotient
sheaves $F_j/F_{j-1}$ to  $\widetilde F_j/\widetilde F_{j-1}$
preserves Hilbert polynomials and stability.

Now we can formulate the analog of a Jordan--H\"{o}lder filtration 
for a vector bundle $\widetilde E$ on an admissible scheme $\widetilde S$
for the notions of stability and semistability from Definition~\ref{sst}.

\begin{definition} {\it Jordan--H\"{o}lder filtration}
of a semistable vector bundle $\widetilde E$ with Hilbert polynomial
equal to $rp(n)$ is a sequence of semistable subsheaves $0
\subset \widetilde F_1 \subset \dots \subset \widetilde F_s
\subset \widetilde E$ with reduced Hilbert polynomials equal to
$p(n)$, such that quotient sheaves $\gr_j(\widetilde E)=\widetilde
F_j/\widetilde F_{j-1}$ are stable.

The sheaf $\bigoplus_j \gr_j(\widetilde E)$ is called  {\it
associated polystable sheaf} for the vector bundle $\widetilde E.$
\end{definition}

Then it follows from the results of propositions \ref{corr},
\ref{satur} and corollaries \ref{sssh} -- \ref{stabquot} that the
transformation \linebreak $E \mapsto \sigma^{\ast} E /\tors$ takes a
Jordan--H\"{o}lder filtration of the sheaf $E$ to a Jordan--H\"{o}lder filtration of the vector bundle $\sigma^{\ast} E /\tors \!.$

Let $(\widetilde S, \widetilde E)$ and $(\widetilde S', \widetilde
E')$ be semistable pairs.

\begin{definition} Semistable pairs  $(\widetilde S,
\widetilde E)$ and $(\widetilde S', \widetilde E')$ are called
{\it $M$-equivalent (monoidally equivalent)} if for the morphisms of
$\diamond$-product $\widetilde S \diamond \widetilde S'$ to the
factors $\overline \sigma'\colon \widetilde S \diamond \widetilde S'
\to \widetilde S$ and $\overline \sigma\colon \widetilde S \diamond
\widetilde S' \to \widetilde S'$ and for the associated polystable
sheaves  $\bigoplus_j \gr_j (\widetilde E)$ and $\bigoplus_j \gr_j
(\widetilde E')$ there are isomorphisms \begin{equation*}\overline
\sigma'^{\ast} \bigoplus_j \gr_j (\widetilde E)/\tors \cong
\overline \sigma^{\ast} \bigoplus_j \gr_j (\widetilde
E')/\tors\!.\end{equation*}
\end{definition}

\begin{proposition} S-equivalent semistable coherent sheaves $E$
and $E'$ are taken by the resolution to M-equivalent semistable pairs $(\widetilde
S, \widetilde E)$ and $(\widetilde S', \widetilde E')$ respectively.
\end{proposition}
\begin{proof} The resolution takes a semistable coherent sheaf
$E$ to a semistable pair  $(\widetilde S, \widetilde E)$. A Jordan--H\"{o}lder filtration 
of the sheaf $E$ is taken to a Jordan--H\"{o}lder filtration of the vector bundle $\widetilde E.$ Then the polystable sheaf $\bigoplus_j \gr_j(E)$ is taken to the associated
polystable sheaf  $\bigoplus _j \gr_j(\widetilde E)$. Hence for the sheaf $E$ we have
\begin{equation}\label{grE}\overline \sigma '^{\ast}\bigoplus_j
\gr_j(\widetilde E)/\tors= \overline \sigma
'^{\ast}[\sigma^{\ast}\bigoplus_j gr_j(E)/\tors]/\tors =
\overline \sigma '^{\ast}\sigma^{\ast}\bigoplus_j
\gr_j(E)/\tors\!.\end{equation} Analogously, for a sheaf $E'$ which
is S-equivalent to the sheaf  $E$ one has
\begin{equation}\label{grE'}\overline \sigma ^{\ast}\bigoplus_j
\gr_j(\widetilde E')/\tors=\overline \sigma
^{\ast}[\sigma'^{\ast}\bigoplus_j \gr_j(E')/\tors]/\tors
=\overline \sigma ^{\ast} \sigma'^{\ast}\bigoplus_j
\gr_j(E')/\tors\!.\end{equation} Right hand sides of (\ref{grE}) and
(\ref{grE'}) are isomorphic by the isomorphism of polystable
$\OO_S$-sheaves $\bigoplus_i \gr_i(E)\cong \bigoplus_i \gr_i(E')$
and by commutativity of the diagram \begin{equation*}
\xymatrix{\widetilde S \diamond \widetilde S' \ar[d]_{\overline
\sigma'} \ar[r]^{\overline \sigma }&\widetilde S'
\ar[d]^{\sigma'}\\
\widetilde S \ar[r]_{\sigma}& S}\end{equation*} for
$\diamond$-product. The proposition is proven. \end{proof}

We turn again to the polystable sheaf $\gr(E)\!=\! \bigoplus_j \gr_j(E)$
in the given S-equival\-ence class. Let $\sigma_{\gr}\colon \!\widetilde
S_{\gr}\! \to \!S$ be the corresponding canonical morphism for a standard resolution defined by
the sequence of sheaves of ideals $(I_{\gr})_i$, $i=1, \dots, \ell$.
\begin{proposition} For all $E \in [E]$ the sheaves
$\sigma_{\gr}^{\ast}E/\tors$ are locally free.
\end{proposition}
\begin{proof}
The sheaf $\sigma_{\gr}^{\ast} \gr(E)/\tors$ is locally free. This
implies that all the direct summands  $\sigma_{\gr}^{\ast}
\gr_j(E)/\tors$ are also locally free. Consider inductively
the following exact sequences
\begin{eqnarray}0\to \sigma_{\gr}^{\ast} F_1/\tors
=\!\!\!= \sigma_{\gr}^{\ast} \gr_1(E)/\tors\to 0,\nonumber\\
0\to \sigma_{\gr}^{\ast} F_1/\tors\to \sigma_{\gr}^{\ast}
F_{2}/\tors \to \sigma_{\gr}^{\ast} \gr_{2}(E)/\tors \to 0,\nonumber\\
.\;.\;.\;.\;.\;.\;.\;.\;.\;.\;.\;.\;.\;.\;.\;.\;.\;.\;.\;.\;.\;.\;.\;.\;.\;.\;.\;.\;.\;.\;.\;.\;.\;.\;.\;.\;.\;.\;.\;.\;.\;.\;.\;.\; \nonumber\\
0\to \sigma_{\gr}^{\ast} F_{s-1}/\tors \to \sigma_{\gr}^{\ast}
E/\tors \to \sigma_{\gr}^{\ast} \gr_s(E)/\tors \to 0,\nonumber
\end{eqnarray}
induced by Jordan -- H\"{o}lder filtration of any semistable sheaf
 $E$ of the given S-equi\-valence class. It follows that the sheaf
$\sigma_{\gr}^{\ast} E/\tors$ is locally free. \end{proof}

Now we come to the global version of M-equivalence.
We prove that on the set of all $T$-flat families of admissible schemes of the view $(\pi\colon \widetilde \Sigma \to T, \widetilde \L)$ there is a commutative binary operation $$\diamond\colon ((\pi\colon \widetilde \Sigma \to T, \widetilde \L),(\pi'\colon \widetilde \Sigma' \to T, \widetilde \L')) \mapsto (\pi_\Delta\colon \widetilde \Sigma_\Delta \to T, \widetilde \L_\Delta)$$
\begin{itemize}
    \item there is a commutative diagram of schemes 
    $$\xymatrix{\widetilde \Sigma_\Delta \ar[d]_{\widehat \sigma\!\!\!\sigma} \ar[dr]_{\pi_\Delta}\ar[r]^{\widehat \sigma\!\!\!\sigma'}& \widetilde \Sigma' \ar[d]^{\pi'} \\
    \widetilde \Sigma \ar[r]_\pi &T}
    $$
    \item $(\pi_\Delta\colon \widetilde \Sigma_\Delta \to T, \widetilde \L_\Delta)$ is flat over $T$;
    \item if $\chi (\widetilde \L^n|_{\pi^{-1}(t)})=\chi (\widetilde \L{'^ n}|_{\pi{'^{-1}}(t)})$ then $\chi (\widetilde \L_\Delta^n|_{\pi_\Delta^{-1}(t)})=\chi (\widetilde \L^n|_{\pi^{-1}(t)})=\chi (\widetilde \L{'^n}|_{\pi{'^{-1}}(t)})$;
    \item if $\chi (\widetilde \E \otimes \widetilde \L^n|_{\pi^{-1}(t)})=\chi (\widetilde \E' \otimes \widetilde \L{'^n}|_{\pi{'^{-1}}(t)})$ then 
    $\chi (\widehat \sigma\!\!\!\sigma^\ast \widetilde \E \otimes \widetilde \L_\Delta^n|_{\pi_\Delta^{-1}(t)})=\chi (\widetilde \E \otimes \widetilde \L^n|_{\pi^{-1}(t)})$ and $\chi (\widehat \sigma\!\!\!\sigma{'^\ast} \widetilde \E' \otimes \widetilde \L_\Delta^n|_{\pi_\Delta^{-1}(t)})=\chi (\widetilde \E' \otimes \widetilde \L{'^n}|_{\pi{'^{-1}}(t)})$.
\end{itemize}

Let $T,S$ be  schemes over a field $k$, $\pi\colon \widetilde \Sigma
\to T$ a morphism of $k$-schemes. We recall the following

\begin{definition}[\cite{Tim8}, Definition 5] \label{bitriv} A family of schemes
$\pi\colon \widetilde \Sigma \to T$ is {\it birationally $S$-trivial}
if there exist isomorphic open subschemes $\widetilde \Sigma_0
\subset \widetilde \Sigma$ and $\Sigma_0 \subset T\times S$ and
there is a scheme equality $\pi(\widetilde \Sigma_0)=T$.
\end{definition}
In particular, the latter equality  means that all fibers of the
morphism $\pi$ have nonempty intersections with the open subscheme
$\widetilde \Sigma_0$.

In the case when $T=\Spec k$ and $\pi$ is a constant morphism
 $\widetilde \Sigma_0 \cong \Sigma_0$ are isomorphic open subschemes in $\widetilde S$ and $S$.

Since in the present paper we consider only birationally
$S$-trivial families, they will be called {\it
birationally trivial} families.

Since we are interested in birationally trivial families of admissible schemes, we can assume that there are commutative diagrams 
$$\xymatrix{\widetilde \Sigma \ar[rd]_\pi \ar[r]^{\sigma\!\!\!\sigma}&\Sigma \ar[d]_p\\
&T} \quad \xymatrix{\widetilde \Sigma' \ar[rd]_{\pi'} \ar[r]^{\sigma\!\!\!\sigma'}&\Sigma \ar[d]_p\\
&T}
$$
where the morphisms $\sigma\!\!\!\sigma$ and $\sigma\!\!\!\sigma'$ are birational and projective.

Now we concentrate on the  case when $T=\Spec k$ and $\widetilde \Sigma=\widetilde S$, $\widetilde \Sigma'=\widetilde S'$ are admissible schemes. Then they can be obtained as zero fibers by consequent blowups $\sigma\!\!\!\sigma_i=bl \I_i,$ $\I_i\subset \OO_{\Sigma_{i-1}},$ $i=1,\dots,\ell,$ of the one-parameter family $\Sigma=\Sigma_0=T \times S$ for $T=\Spec k[t]$:
\begin{eqnarray*}
\Sigma_0 \stackrel{\sigma\!\!\!\sigma_1}{\longleftarrow} \Sigma_1\stackrel{\sigma\!\!\!\sigma_2}{\longleftarrow} \Sigma_2 \stackrel{\sigma\!\!\!\sigma_3}{\longleftarrow} \dots \stackrel{\sigma\!\!\!\sigma_\ell}{\longleftarrow} \Sigma_\ell=\widetilde  \Sigma, \quad \sigma\!\!\!\sigma_i=bl \I_i,\;\; \I_i\subset \OO_{\Sigma_{i-1}}, \;\; i=1,\dots,\ell;\\
\Sigma_0 \stackrel{\sigma\!\!\!\sigma'_1}{\longleftarrow} \Sigma'_1\stackrel{\sigma\!\!\!\sigma'_2}{\longleftarrow} \Sigma'_2 \stackrel{\sigma\!\!\!\sigma'_3}{\longleftarrow} \dots \stackrel{\sigma\!\!\!\sigma'_{\ell'}}{\longleftarrow} \Sigma'_{\ell'}=\widetilde  \Sigma', \quad \sigma\!\!\!\sigma'_i=bl \I'_i,\;\; \I'_i\subset \OO_{\Sigma'_{i-1}}, \;\; i=1,\dots,\ell'.
\end{eqnarray*}
For the first step consider the morphisms $\sigma\!\!\!\sigma_1$ and $\sigma\!\!\!\sigma'_1$, the inverse images $$\widehat \I_1=\sigma\!\!\!\sigma_1^{-1}\I'_1 \cdot \OO_{\Sigma_1}, \quad \widehat \I'_1=\sigma\!\!\!\sigma{'^{-1}} \I_1 \cdot \OO_{\Sigma'_1}$$
and the corresponding blown up schemes 
$$ \widehat \Sigma_{11}=Bl_{\Sigma_1} \I'_1 \stackrel{\widehat \sigma\!\!\!\sigma'_1}{\longrightarrow} \Sigma_1, \quad \widehat \Sigma'_{11}=Bl_{\Sigma'_1} \I_1 \stackrel{\widehat \sigma\!\!\!\sigma_1}{\longrightarrow} \Sigma'_1.
$$
By the universality of blowups these schemes are supplied with the morphisms 
$$\widehat \Sigma_{11} \to \Sigma'_1, \quad \widehat \Sigma'_{11} \to \Sigma_1.
$$
The schemes $\widehat \Sigma_{11}$ and $\widehat \Sigma'_{11}$ are irreducible and reduced. They have morphisms to $\Sigma_1 \times_T \Sigma'_1$. Each of the schemes arises as a closure of the diagonal immersion of the open set $\Sigma_1 \supset U \subset \Sigma'_1$, $U=(T\setminus 0) \times S$. Hence $\widehat \Sigma_{11}=\widehat \Sigma'_{11}$, and  we come to the commutative square 
$$
\xymatrix{\Sigma_1 \ar[d]_{\sigma\!\!\!\sigma_1} &\ar[l]_{\widehat \sigma\!\!\!\sigma'_1}  \ar[ld]_{\sigma\!\!\!\sigma_{11}} \widehat \Sigma_{11} \ar[d]^{\widehat \sigma\!\!\!\sigma_1}\\
\Sigma_0&\ar[l]^{\sigma\!\!\!\sigma'_1} \Sigma'_1}
$$
and to the projection morphism $\pi_{11}\colon \widehat \Sigma_{11} \stackrel{\sigma\!\!\!\sigma_1 \circ \widehat \sigma\!\!\!\sigma'_1}{\longrightarrow} \Sigma_0 \stackrel{p}{\longrightarrow} T$.
If $\Sigma_1$ has fiberwise uniform Hilbert polynomial $\chi(\L_1^n|_{\pi_1^{-1}(t)})$  computed with respect to $\L_1=\sigma\!\!\!\sigma_1^\ast \L \otimes \sigma\!\!\!\sigma_1^{-1} \I_1 \cdot \OO_{\Sigma_1}$ and
$\Sigma'_1$ has fiberwise uniform Hilbert polynomial $\chi({\L'_1} ^n|_{{\pi'_1} ^{-1}(t)})$  computed with respect to $\L'_1={\sigma\!\!\!\sigma'_1}^\ast \L \otimes {\sigma\!\!\!\sigma'_1}^{-1} \I_1'\cdot \OO_{\Sigma'_1}$, then $\Sigma_{11}$ has also fiberwise uniform Hilbert polynomial $\chi(\L_{11}^n|_{\pi_{11}^{-1}(t)})$  computed with respect to $$\L_{11}=\sigma\!\!\!\sigma_{11}^\ast \L \otimes {\widehat\sigma\!\!\!\sigma'_1}^{-1}\sigma\!\!\!\sigma_1^{-1} \I_1\cdot \OO_{\Sigma_{11}}\cdot \widehat\sigma\!\!\!\sigma^{-1}_1{\sigma\!\!\!\sigma'_1}^{-1} \I'_1 \cdot\OO_{\Sigma_{11}}=\sigma\!\!\!\sigma_{11}^\ast \L \otimes {\sigma\!\!\!\sigma_{11}}^{-1}\I_1 \cdot\OO_{\Sigma_{11}}\cdot \sigma\!\!\!\sigma^{-1}_{11} \I'_1 \cdot\OO_{\Sigma_{11}}.$$
Iterating the process, one comes to the scheme $\widehat \Sigma_{ij}$ included in the commutative square 
$$
\xymatrix{\Sigma_i \ar[d]_{\sigma\!\!\!\sigma_1\circ \dots \circ \sigma\!\!\!\sigma_i} &&\ar[ll]_{\widehat \sigma\!\!\!\sigma'_1\circ \dots \circ \widehat \sigma\!\!\!\sigma'_{i'}}  \ar[dll]_{\sigma\!\!\!\sigma_{ii'}}\widehat \Sigma_{ii'} \ar[d]^{\widehat \sigma\!\!\!\sigma_1\circ \dots \circ \widehat\sigma\!\!\!\sigma_i}\\
\Sigma_0&&\ar[ll]^{\sigma\!\!\!\sigma'_1\circ \dots \circ \sigma\!\!\!\sigma'_{i'}} \Sigma'_{i'}}
$$
with the projection morphism $\pi_{ii'}=p\circ \sigma\!\!\!\sigma_{ii'}\colon \widehat \Sigma_{ii'} \to T$ for any pair $i,i'$. 
Also the morphism $\pi_i\colon \Sigma_i \to T$ is flat with fiberwise Hilbert polynomial being uniform over $t\in T$ if it is computed with respect to the polarization \begin{equation} \label{Li}
   \L_i=\sigma\!\!\!\sigma_i^\ast \dots \sigma\!\!\!\sigma_1^\ast \L \otimes  \sigma\!\!\!\sigma_i^{-1} \dots \sigma\!\!\!\sigma_1^{-1} I_1 \otimes \dots \otimes \sigma\!\!\!\sigma_i^{-1} I_i \cdot \OO_{\Sigma_i}, 
\end{equation} the morphism $\pi'_{i'}\colon \Sigma'_{i'} \to T$ is also flat with fiberwise Hilbert polynomial being uniform over $t\in T$ if it is computed with respect to the polarization \begin{equation}\label{L'j}
\L'_{i'}={\sigma\!\!\!\sigma'_{i'}} ^\ast \dots {\sigma\!\!\!\sigma'_1}^\ast \L \otimes  {\sigma\!\!\!\sigma'_{i'}}^{-1} \dots {\sigma\!\!\!\sigma'_1}^{-1} \I'_1 \otimes \dots \otimes {\sigma\!\!\!\sigma'_{j'}}^{-1} \I_j \cdot \OO_{\Sigma_{i'}},
\end{equation}
 and as well the morphism $\pi_{ii'}\colon \Sigma_{ii'} \to T$ is flat with fiberwise Hilbert polynomial being uniform over $t\in T$ if it is computed with respect to the polarization \begin{eqnarray}&&\L_{ii'}=\sigma\!\!\!\sigma_{ii'}^\ast  \L \otimes (\widehat \sigma\!\!\!\sigma'_1 \circ \dots \circ \widehat \sigma\!\!\!\sigma'_{i'})^{-1}(\sigma\!\!\!\sigma_i^{-1}\dots \sigma\!\!\!\sigma_1^{-1} \I_1 \otimes \dots \otimes \sigma\!\!\!\sigma_i^{-1}\I_i) \nonumber\\
&&\otimes (\widehat \sigma\!\!\!\sigma_1 \circ \dots \circ \widehat \sigma\!\!\!\sigma_i)^{-1}({\sigma\!\!\!\sigma'_{i'}}^{-1}\dots {\sigma\!\!\!\sigma'_1}^ {-1} \I'_1 \otimes \dots \otimes {\sigma\!\!\!\sigma'_{i'}}^{-1}\I'_{i'}) \cdot \OO_{\Sigma_{ii'}}. \label{Lij}\end{eqnarray}
If continue on $i,i'$ we come to $\widetilde\L_\Delta$ as fiberwise polarization for $\widetilde \Sigma_\Delta.$
\begin{remark}
If necessary, as in (\ref{polarizfam}) there can be  some degrees of inverse images of ideals $\I_i$, $i=1, \dots, \ell$, $\I'_{i'}$, $i'=1,\dots, \ell$, involved in (\ref{Li}), (\ref{L'j}), (\ref{Lij}). 
\end{remark}

Now we turn our attention to the preimages of locally free sheaves of $\OO_{\widetilde \Sigma}$-modules $\widetilde \E$ and of $\OO_{\widetilde \Sigma'}$-modules $\widetilde \E'$ on $\widetilde \Sigma_\Delta$. Since  $\widetilde \E$ is locally free, then $\sigma\!\!\!\sigma{'^\ast} \widetilde \E$ is also locally free; $\pi_\Delta\colon \widetilde \Sigma_\Delta \to T$ is flat morphism and hence $\sigma\!\!\!\sigma{'^\ast} \widetilde \E$ is flat over $T$, by \cite[Proposition 7.9.14]{EGAIII} $\pi_{\Delta \ast} (\sigma\!\!\!\sigma{'^\ast} \widetilde \E \otimes \widetilde \L_\Delta^n)$ is locally free sheaf and by \cite[Corollary 7.9.13]{EGAIII} fiberwise Hilbert polynomial $\chi(\sigma\!\!\!\sigma{'^\ast} \widetilde \E \otimes \widetilde \L_\Delta^n|_{\pi_\Delta^{-1}(t)})$ is constant over $t\in T.$
%the fiberwise Hilbert polynomial $\chi(\widetilde \L_\Delta^n|_{\pi^{-1}_\Delta}(t))$ does not depend on $t\in T$ then induction over rank proves that the same is true for any locally free sheaf $\F$. 
The sheaf of  $\OO_{\pi_\Delta^{-1}(t)}$-modules $\widehat \sigma\!\!\!\sigma{'^\ast}\widetilde \E \otimes \widetilde \L_\Delta^n|_{\pi_\Delta^{-1}(t)}$ can be obtained as a zero fiber of the appropriately blown up family $\Spec k[t']\times \pi^{-1}(t) \stackrel{\overline \sigma}{\longleftarrow} \overline \Sigma$ (the center of blowing up is concentrated in the zero fiber) with an inverse image  $\overline \sigma^\ast pr_2^\ast(\widetilde \E|_{\pi^{-1}(t)})$. The blowup morphism $\overline \sigma$ is assumed to be a composite of consequent blowups in the sequence of sheaves of ideals $I_1[t']+(t'), \dots , I_s[t']+(t')$. The distinguished polarization $\overline \L$ for the flat family $\overline \pi\colon \overline \Sigma \to \Spec k[t']$ takes a view $\overline \L= \overline \sigma^\ast pr_2^\ast \widetilde \L|_{\pi^{-1}(t)} \otimes \overline \sigma^{-1}[I_1,\dots, I_s] \cdot \OO_{\overline \Sigma}$, where $\sigma^{-1}[I_1,\dots, I_s]$ is some product of inverse images of $I_i[t']+(t')$, $i=1,\dots,s$. The restriction of $\overline \L$ to the zero fiber takes the view $\overline \L |_{\overline \pi^{-1} (0)}=\widetilde \L_\Delta|_{\pi_\Delta^{-1}(t)}$. The inverse image of $\widehat \sigma\!\!\!\sigma^\ast \widetilde \E |_{\pi_\Delta^{-1}(t)}$ has Hilbert polynomial $\chi(\widehat \sigma\!\!\!\sigma{'^\ast} \widetilde \E \otimes \widetilde \L_\Delta^n|_{\pi_\Delta^{-1}(t)})=\chi(\overline \sigma^\ast \widetilde\E|_{\pi^{-1}(t)}\otimes \overline \L^n|_{\overline \pi(0)})=\chi(\overline \sigma^\ast \widetilde\E|_{\pi^{-1}(t)}\otimes \overline \L^n|_{\overline \pi(u)})=\chi( \widetilde\E|_{\pi^{-1}(t)}\otimes \widetilde \L^n|_{\overline \pi(u)}),$ where $u\ne 0$.
By the construction, the binary operation $\diamond$ is commutative and it leads to M-equivalence of pairs which are obtained by resolutions of the same sheaf $E$ or of two S-equivalent sheaves $E,$ $E'$ on $S$.
%%%%%%%%%%%%%%%%%%%%%%%%%%%%%%%%%%%%%%%%%%%%%%%%%%%%%%%%%%%%%%%%%

\section{Moduli functors}\label{s5}

Following \cite[ch. 2, sect. 2.2]{HL}, we recall some definitions.
Let ${\mathcal C}$ be a category, ${\mathcal C}^o$ its dual,
 ${\mathcal C}'={{\mathcal F}unct}({\mathcal C}^o, Sets)$ a category of
 functors to the category of sets. By Yoneda's lemma, the functor  $${\mathcal C} \to
{\mathcal C}'\colon F\mapsto (\underline F\colon X\mapsto {\rm Hom
\,}_{{\mathcal C}}(X, F))$$ includes ${\mathcal C}$ into ${\mathcal
C}'$ as a full subcategory.

\begin{definition}\label{corep}\cite[ch. 2, Definition 2.2.1]{HL}
A functor  ${\mathfrak f} \in {\OO}b\, \CC'$ is {\it
corepresented by an object} $M \in {\OO}b \,\CC$, if there exist
a $\CC'$-morphism $\psi \colon {\mathfrak f} \to \underline M$ such
that any morphism $\psi'\colon {\mathfrak f} \to \underline F'$ factors
through the unique morphism  $\omega\colon \underline M \to \underline
F'$.
\end{definition}

\begin{definition} The scheme  $\widetilde M$ is a {\it coarse moduli space}
for the functor $\mathfrak f$ if  $\mathfrak f$ is corepresented
by the scheme $\widetilde M.$
 \end{definition}

We consider sets of families of semistable pairs
\begin{equation}\label{class}{\mathfrak F}_T= \left\{
\begin{array}{l}\pi\colon \widetilde \Sigma \to T \mbox{\rm \;\; is a birationally $S$-trivial family};\\
\widetilde \L\in \Pic \widetilde \Sigma \mbox{\rm \;\;is flat over
}T;\\
\mbox{\rm for }m\gg 0 \; \widetilde \L^m \; \mbox{\rm is very ample
relative to }T;\\ \forall t\in T \;\widetilde L_t=\widetilde
\L|_{\pi^{-1}(t)}
\mbox{\rm \; is ample;}\\
(\pi^{-1}(t),\widetilde L_t) \mbox{\rm \;is an admissible scheme with
a distinguished polarization}; \\
\chi (\widetilde L_t^n) \mbox{\rm \; does not depend on }t;\\
 \widetilde \E \;\; \mbox{\rm is a locally free } \OO_{\Sigma}\mbox{\rm
-sheaf flat over } T;\\
 \chi(\widetilde \E\otimes\widetilde \L^{n})|_{\pi^{-1}(t)})=
 rp(n);\\
 ((\pi^{-1}(t), \widetilde L_t), \widetilde \E|_{\pi^{-1}(t)}) 
 \mbox{\rm \;\;is a semistable pair}
 \end{array} \right\} \end{equation}  and a functor
\begin{equation}\label{funcmy}\mathfrak f\colon (Schemes_k)^o
\to (Sets)\end{equation} from the category of $k$-schemes to the
category of sets. It attaches to any scheme $T$ the set of
equivalence classes of families of the form $({\mathfrak
F}_T/\sim).$

The equivalence relation  $\sim$ is defined as follows. Families
 $((\pi\colon \widetilde \Sigma \to T, \widetilde \L),
\widetilde \E)$ and  $((\pi'\colon \widetilde \Sigma \to T, \widetilde
\L'), \widetilde
 \E')$ from the class  $\mathfrak F_T$ are said to be equivalent
 (notation:
 $((\pi\colon \widetilde \Sigma \to T, \widetilde \L),
 \widetilde \E) \sim ((\pi'\colon \widetilde \Sigma \to T, \widetilde \L'), \widetilde
 \E')$), if\\
 1) there exist an isomorphism $ \iota\colon \widetilde \Sigma \stackrel{\sim}{\longrightarrow}
 \widetilde \Sigma'$ such that the diagram \begin{equation*}
 \xymatrix{\widetilde \Sigma  \ar[rd]_{\pi}\ar[rr]_{\sim}^{\iota}&&\widetilde \Sigma ' \ar[ld]^{\pi'}\\
&T }
 \end{equation*} commutes;\\
 2) there exist line bundles   $L', L''$ on the scheme  $T$ such
 that $\iota^{\ast}\widetilde \E' = \widetilde \E \otimes \pi^{\ast} L',$
 $\iota^{\ast}\widetilde \L' = \widetilde \L \otimes \pi^{\ast} L''.$

Now we discuss what is the ``size'' of $\widetilde \Sigma_0$ which is a maximal under  inclusion
among the open sub\-schemes  in  $\widetilde \Sigma$ which are isomorphic to
appropriate open subschemes in $T\times S$ in Definition~\ref{bitriv}. The set $F=\widetilde \Sigma \setminus \widetilde
\Sigma_0$ is closed. If $T_0$ is an open subscheme in $T$ whose
points carry fibers isomorphic to $S$, then $\widetilde \Sigma_0
\supsetneqq \pi^{-1}T_0$ (inequality is true because in Definition~\ref{bitriv} $\pi
(\widetilde \Sigma_0)=T$ is required). The
subscheme~$\Sigma_0$, which is open in $T\times S$ and isomorphic
to $\widetilde \Sigma_0$, is such that $\Sigma_0 \supsetneqq
T_0\times S$. If $\pi\colon \widetilde \Sigma \to T$ is a family of
admissible schemes then $\widetilde \Sigma_0 \cong \widetilde
\Sigma \setminus F$, and $F$ is (set-theoretically) a union of
additional components of those fibers which are not isomorphic to $S$.
Particularly, this means that $\codim_{T\times S} (T\times
S)\setminus \Sigma_0 \ge 2.$

\smallskip

The Gieseker -- Maruyama functor
\begin{equation}\label{funcGM}{\mathfrak f}^{GM}\colon 
(Schemes_k)^o \to Sets,\end{equation} attaches to any scheme $T$
a set of equivalence classes of families of the following form
${\mathfrak F}_T^{GM}/\sim$, where
\begin{equation}\label{famGM} \mathfrak
F_T^{\,GM}= \left\{
\begin{array}{l} \E \;\;\mbox{\rm is a sheaf of } \OO_{T\times S}-
\mbox{\rm modules flat over } T;\\
\L \;\;\mbox{\rm is an invertible sheaf of } \OO_{T\times S}-\mbox{\rm
modules,}\\ \mbox{\rm  ample relatively to } T\\
\mbox{\rm and such that } L_t:=\L|_{t\times S}\cong L\; \mbox{\rm for any point } t\in T;\\
E_t:=\E|_{t\times S} \;\mbox{\rm is a torsion-free and Gieseker-semistable;}\\
\chi(E_t \otimes L_t^n)=rp(n).\end{array}\right\}
\end{equation}

The families   $(\E, \L)$ and $(\E',\L')$ from the class $\mathfrak
F^{GM}_T$ are said to be equivalent (notation: $ (\E, \L)\sim
(\E',\L')$), if there exist line bundles  $L', L''$ on the scheme
$T$ such that $\E' = \E \otimes p^{\ast} L',$ $ \L' =  \L \otimes
p^{\ast} L''$, where $p\colon T\times S \to T$ is the projection onto the
first factor.

\begin{remark} Since $\Pic (T \times S) =\Pic T \times \Pic
S$, our definition of the moduli functor ${\mathfrak f}^{GM}$ is
equivalent to the standard definition which can be found, for
example, in \cite{HL}: the difference in choice of polarizations
$\L$ and $\L'$ having isomorphic restrictions on corresponding fibers over the
base $T$ is avoided by the equivalence which is induced by
tensoring by the inverse image of an invertible sheaf $L''$ from the
base $T$.
\end{remark}
\begin{remark}
The procedure of standard resolution developed in \cite{Tim12} and outlined in Sec.2 provides a natural transformation $\underline \kappa\colon {\mathfrak f}^{GM} \to \mathfrak f$ mentioned in part ({\it i}) of Theorem \ref{thfunc}.
\end{remark}
\section{A Pairs-to-GM transformation}\label{s4}

%%%%%%%%%%%%%%%%%%%%%%%%%%%%%%%%%%%%%%%%%%%%%%%%%%%%%%%%%%%%%%%%%
%%%%%%%%%%%%%%%%%%%%%%%%%%%%%%%%%%%%%%%%%%%%%%%%%%%%%%%%%%%%%%%%%

Further we will show that there is a morphism of the nonreduced moduli
functor of admiss\-ible semistable pairs to the nonreduced
Gieseker--Maruyama moduli functor. Namely, for any scheme $T$ we
build up a correspondence $((\pi\colon \widetilde \Sigma \to T,
\widetilde \L), \widetilde \E)\mapsto (\L, \E)$. It leads to
 a mapping of sets
 $$(\{((\pi\colon \widetilde \Sigma \to T, \widetilde
\L), \widetilde \E)\}/\sim)\to (\{(\L,\E)\}/\sim).$$ This means
that a family of semistable coherent torsion-free sheaves $\E$
with the same base $T$ can be constructed by any family
$((\pi\colon \widetilde \Sigma \to T, \widetilde \L), \widetilde \E)$ of
admissible semistable pairs which is birationally trivial and flat
over $T$.

First we construct a $T$-morphism $\phi\colon \widetilde \Sigma \to
T\times S$. Since the family $\pi\colon \widetilde  \Sigma \to T$ is
birationally trivial, there is a fixed isomorphism $\phi_0 \colon 
\widetilde \Sigma_0 \stackrel{\sim}{\to}\Sigma_0$ of maximal open
subschemes $\widetilde \Sigma_0 \subset \widetilde \Sigma$ and
$\Sigma_0 \subset T\times S$. Define an invertible $\OO_{T\times
S}$-sheaf $\L$ by the equality
$$\L(U):= \widetilde \L(\phi_0^{-1}(U \cap \Sigma_0)).$$ Identifying
$\widetilde \Sigma_0$ with $\Sigma_0$ by the isomorphism $\phi_0$
one comes to the conclusion that sheaves $\L|_{\Sigma_0}$ and
$\widetilde \L|_{\widetilde \Sigma_0}$ are also isomorphic.

For each closed point $t\in T$ there is a canonical morphism of
the fiber $\sigma_t\colon \widetilde S_t \to S$, where $\widetilde
S_t=\pi^{-1}(t).$
\begin{proposition} \label{polariz} For any closed point $t\in T$ and for any open subset $V\subset S$
$$
\L\otimes (k_t\boxtimes \OO_S)(V)=\widetilde L_t(\sigma_t^{-1}(V)
\cap \widetilde \Sigma_0).
$$
In particular, $\L\otimes (k_t \boxtimes \OO_S)=L.$
\end{proposition}
\begin{proof} The restriction $\L|_{t\times S}$ is the sheaf associated
to the presheaf $$V \mapsto \L(U) \otimes _{\OO_{T\times S}(U)}
(k_t\boxtimes \OO_S)(U\cap t\times S),$$
 for any open subset $U\subset T\times S$ such that $U\cap (t\times S)=V.$
Since $\codim T\times S \setminus \Sigma_0 \ge 2,$
$$\OO_{
T\times S}(U)=\OO_{T\times S}(U \cap \Sigma_0)$$ and
$$(k_t \boxtimes \OO_S)(U\cap t\times S)=
(k_t \boxtimes \OO_S)(U\cap \Sigma_0\cap t\times S)
=\OO_{\widetilde S_t}(\phi_0^{-1}(U\cap \Sigma_0)\cap \pi^{-1}
(t)).$$ Hence, $\L|_{t\times S}$ is associated with the presheaf $$
V\mapsto \widetilde \L(\phi_0^{-1}(U\cap \Sigma_0))
\otimes_{\OO_{\widetilde \Sigma}(\phi_0^{-1}(U\cap \Sigma_0))}
\OO_{\pi^{-1}(t)}(\phi_0^{-1}(U\cap \Sigma_0)\cap \pi^{-1}(t)),
$$ or, equivalently,
$$V\mapsto
\widetilde L_t (\phi_0^{-1}(U\cap \Sigma_0)\cap \widetilde S_t)
=L(U\cap \Sigma_0\cap t\times S)=L(U\cap t\times S).$$ We keep in
mind that $\phi_0^{-1}(U\cap \Sigma_0)\cap \widetilde
S_t=\sigma_t^{-1}(V)\cap \widetilde \Sigma_0$, what completes the
proof.\end{proof}
%%%%%%%%%%%%%%%%%%%%%%%%%%%%%%%%%%%%%%%%%%%%%%%%%%%%%%%%%%%%%%%%%
%%%%%%%%%%%%%%%%%%%%%%%%%%%%%%%%%%%%%%%%%%%%%%%%%%%%%%%%%%%%%%%%%

Define a sheaf $\L'$ by the correspondence $U \mapsto \widetilde
\L(U\cap \widetilde \Sigma_0)$ for any open subset $U\subset \widetilde
\Sigma.$ It carries a natural structure of an invertible
$\OO_{\widetilde \Sigma}$-module. This structure is induced by the
commutative diagram
$$\xymatrix{\OO_{\widetilde \Sigma}(U)\times \L'(U)
\ar[d]_{res} \ar[r]&\L'(U) \ar@{=}[d]\\
\OO_{\widetilde \Sigma}(U\cap \widetilde \Sigma_0)\times
\widetilde \L(U\cap \widetilde \Sigma_0)\ar[r] & \widetilde
\L(U\cap \widetilde \Sigma_0)}
$$
where the vertical arrow is induced by the natural restriction map in the structure sheaf
$\OO_{\widetilde \Sigma}$. Compare the direct images $p_{\ast}\L$ and
$\pi_{\ast} \L'$; for any open subset $V\subset T$ one has
$$p_{\ast}\L (V)=\L(p^{-1}V)=\L(p^{-1}V \cap \Sigma_0).
$$
By the definition of $\L'$
$$
\L(p^{-1}V \cap \Sigma_0)=\widetilde \L(\pi^{-1}V \cap \widetilde
\Sigma_0)=\L'(\pi^{-1}V)=\pi_{\ast}\L'(V).
$$
Now $\pi_{\ast} \L'=p_{\ast} \L$.

The invertible sheaf $\L'$ induces a morphism $\phi'\colon \widetilde
\Sigma \to \P(\pi_{\ast}\L')^\vee$ which is included into the
commutative diagram of $T$-schemes
$$\xymatrix{\widetilde\Sigma \ar[rr]^{\phi'}&&\P(\pi_{\ast}\L')^{\vee} \ar@{=}[dd] \\
\widetilde \Sigma_0=\Sigma_0 \ar@{^(->}[u] \ar@{_(->}[d]\\
T\times S \ar@{^(->}[rr]^{i_{\L}}&& \P(p_{\ast}\L)^{\vee}}
$$
where $i_{\L}$ is a closed immersion induced by $\L$ and
$\phi'|_{\widetilde \Sigma_0}$ is also an immersion. From now we
identify $\P(\pi_{\ast}\L')^{\vee}$ with $\P(p_{\ast}\L)^{\vee}$
and use the common notation $\P$ for these projective bundles. Formation of
scheme closures of the images of $\widetilde \Sigma_0$ and $\Sigma_0$
in $\P$ leads to $\overline{\phi'(\widetilde
\Sigma_0)}=\overline{i_\L(T\times S)} =T\times S.$ Also by the
definition of the sheaf $\L'$ for any open subset $U\subset \widetilde
\Sigma$ and $V\subset T\times S$ such that $U\cap \widetilde
\Sigma_0 \cong V\cap \Sigma_0$ the following chain of equalities
holds: \begin{equation} \label{corrsec} \L'(U)= \L'(U\cap
\widetilde \Sigma_0)= \L(V\cap \Sigma_0)=\L(V).\end{equation} Now
for a moment we suppose that $T$ is affine scheme: let $T=\Spec A$ for some
commutative algebra $A$, $\P=\Proj A[x_0:\dots :x_N]$ where
$x_0,\dots,x_N\in H^0(\P, \OO_\P (1))$ generate $\OO_\P (1)$.
The images ${\phi'}^{\ast} x_i=s_i',$ $i=0,\dots, N$, generate $\L'$
along $\widetilde \Sigma_0$ and they are not obliged to generate
$\L'$ along the whole of $\widetilde \Sigma.$ The images $i_\L^{\ast}
x_i=s_i$, $i=0,\dots, N$, generate $\L$ along the whole of
$T\times S$ and provide $i_\L$ to be a closed immersion.

We pass to the standard affine covering by $\P_i=\Spec A[x_0,\dots,
\hat x_i,\dots, x_N],$ $i=0, \dots, N$, and $\hat{\phantom q}$
means omitting the symbol below. Denote $(i_\L(T\times
S))_i:=i_\L(T\times S)\cap \P_i$ and $(\phi'(\widetilde
\Sigma))_i:=\phi'(\widetilde \Sigma) \cap \P_i.$ Set also
$(T\times S)_i:=i_\L ^{-1}(i_\L(T\times S))_i$ and $\widetilde
\Sigma_i:={\phi'}^{-1}(\phi'(\widetilde \Sigma))_i$. Now we have
the mappings $$ A[x_0,\dots, \hat x_i,\dots, x_N] \rightarrow
\Gamma(\widetilde \Sigma_i, \L')\colon x_j \mapsto s'_j
$$
and
$$
A[x_0,\dots, \hat x_i,\dots, x_N] \twoheadrightarrow
\Gamma((T\times S)_i, \L)\colon x_j \mapsto s_j
$$
included into the triangular diagram
\begin{equation}\label{diasec}\xymatrix{A[x_0,\dots, \hat x_i,\dots, x_N]
\ar@{->>}[rd] \ar@{->>}[r]& \Gamma((T\times S)_i, \L) \ar@{=}[d]\\
&\Gamma(\widetilde \Sigma_i, \L')}\end{equation} where the
vertical sign of equality means the bijection (\ref{corrsec})
rewritten for the covering under the scope. Commutativity of
(\ref{diasec}) implies that $\phi'$ factors through $i_\L(T\times
S)$, i.e. $\phi'(\widetilde \Sigma) = i_\L (T\times S).$

Now, identifying $i_\L(T\times S)$ with $T\times S$ by means of
the  obvious isomorphism, we arrive to the $T$-morphism
$$ \phi\colon \widetilde \Sigma \to T\times S.
$$
It coincides with $\phi_0\colon \widetilde \Sigma_0
\stackrel{\sim}{\to} \Sigma_0$ when restricted to $\widetilde
\Sigma_0.$

For $n>0$ consider an invertible $\OO_{T\times S}$-sheaf $U
\mapsto \widetilde \L^n(\phi_0^{-1}(U \cap \Sigma_0))$. It
coincides with $\L^n$ on $\Sigma_0$ and hence it coincides with $\L^n$
on the whole of $T\times S$.
%%%%%%%%%%%%%%%%%%%%%%%%%%%%%%%%%%%%%%%%%%%%%%%%%%%%%%%%%%%%%%%%%

Now there is a commutative triangle
\begin{equation*}\xymatrix{\widetilde \Sigma \ar[r]^\phi
\ar[rd]_\pi &T\times S \ar[d]^p\\&T}
\end{equation*}

Firstly note that $T$ contains at least one closed point, say
$t\in T$; let $\widetilde S_t=\pi^{-1}(t)$ be the corresponding
closed fiber and $\widetilde L_t=\widetilde \L|_{\widetilde S_t}$
and $\widetilde E_t=\widetilde \E|_{\widetilde S_t}$ the restrictions
of the sheaves onto this fiber. By the definition of admissible scheme there is
a canonical morphism $\sigma_t\colon \widetilde S_t \to S$. Then
$(\sigma_{t\ast} \widetilde L_t)^{\vee \vee}=L.$ Indeed, both sheaves $(\sigma_{t\ast} \widetilde L_t)^{\vee \vee}$ and $L$ are reflexive  on the nonsingular variety $S$. Also $\sigma_{t\ast} \widetilde L_t/\tors$ is a sheaf of ideals tensored by some invertible sheaf, hence $(\sigma_{t\ast} \widetilde L_t)^{\vee \vee}$ is invertible. Both the sheaves coincide off a closed subset of codimension $\ge 2$. Hence they coincide on the whole of $S.$

Secondly, the family $\pi\colon \widetilde \Sigma \to T$ is
birationally trivial, i.e. there exist isomorphic open subschemes
$\widetilde \Sigma_0 \subset \widetilde \Sigma$ and $\Sigma_0
\subset T\times S$  and
there is a scheme equality $\pi(\widetilde \Sigma_0)=T$. Note that the ``boundary'' $\Delta=T\times S
\setminus \Sigma_0$  has codimension $\ge 2$ and that for any
closed point $t\in T$ $\codim \Delta\cap (t\times S)\ge 2.$

Thirdly, the morphism of multiplication of sections $$
(\sigma_{t\ast} \widetilde L_t)^n \to \sigma_{t\ast} \widetilde
L_t^n
$$
induces the morphism of reflexive hulls
$$
((\sigma_{t\ast} \widetilde L_t)^n )^{\vee \vee}\to
(\sigma_{t\ast} \widetilde L_t^n)^{\vee \vee},
$$
which are normal sheaves on a nonsingular variety and coincide apart
from a subset of codimension $\ge 2$. Hence they are equal on the whole of $S$. Also the sheaf
$((\sigma_{t\ast} \widetilde L_t)^{\vee \vee})^n= L^n$ coincides
with them by the analogous reason. Then for all $n>0$
$$
((\sigma_{t\ast} \widetilde L_t)^n )^{\vee \vee} =L^n.
$$

Now take a product $\Sigma=\A^1 \times S \stackrel{p}{\longrightarrow}\A^1$,  $\A^1 =\Spec k[u]$. Let $\sigma\!\!\!\sigma\colon \widetilde \Sigma \to \Sigma$ be the sequence of blowups whose loci do not dominate $\A^1$ and $\widetilde S$ be the zero fiber $\widetilde S=(p\circ \sigma\!\!\!\sigma)^{-1}(0)$. In this case a general fiber is isomorphic to $S$, $\widetilde \Sigma \stackrel{\pi}{\longrightarrow} \A^1$ is a flat family of admissible schemes and hence $H^0(\widetilde S, \widetilde L^n)= h^0(S, L^n)$ for $n\gg 0.$

%%%%%%%%%%%%%%%%%%%%%%%%%%%%%%%%%%%%%%%%%%%%%%%%%%%%%%%%%%%%%%%%%

\begin{proposition} There are morphisms of $\OO_{T\times S}$
-sheaves $$ \sigma\!\!\!\sigma_{\ast} \widetilde \L^n \to \L^n$$ for all $n>0$.
\end{proposition}
\begin{proof}
For any open $U\in T\times S$ and any $n>0$ there is a restriction
map of sections $res\colon (\sigma\!\!\!\sigma_{\ast} \widetilde \L^n)(U)\to
(\sigma\!\!\!\sigma_{\ast} \widetilde \L^n)(U\cap \Sigma_0).$ Denoting as
usually the preimage $\sigma\!\!\!\sigma^{-1}(\Sigma_0)$ by $\widetilde
\Sigma_0$ (recall that $\sigma\!\!\!\sigma|_{\widetilde \Sigma_0}$ is an
isomorphism) one arrives to the chain of equalities:
$$
(\sigma\!\!\!\sigma_{\ast} \widetilde \L^n)(U\cap \Sigma_0)= \widetilde \L^n
(\sigma\!\!\!\sigma^{-1}(U\cap \Sigma_0))=
 \L^n (U).
$$
\end{proof}
\begin{remark}
Also the step of the descent can be taken of any length. Denoting by $\sigma\!\!\!\sigma_{[i]}$ the composite 
$\Sigma_i \stackrel{\sigma\!\!\!\sigma_i}{\longrightarrow} \dots \stackrel{\sigma\!\!\!\sigma_1}{\longrightarrow} \Sigma$ one can repeat the previous proof and come to  morphisms  of $\OO_{\Sigma}$
-sheaves $$ \sigma\!\!\!\sigma_{[i]\ast} \widetilde \L_i^n \to \L^n$$ for all $n>0$. 
\end{remark}
Applying $p_{\ast}$ yields 
\begin{corollary} For $n>0$ the morphisms $
\sigma\!\!\!\sigma_{[i]\ast} \widetilde \L_i^n \to \L^n$ induce isomorphisms of
$\OO_T$-sheaves
$$\pi_{i\ast} \widetilde \L_i^n \stackrel{\sim}{\to} p_{\ast}\L^n.
$$
\end{corollary}
\begin{proof} Both sheaves $\pi_{i\ast} \widetilde \L_i^n$ and  $p_{\ast}\L^n$ are
locally free and have equal ranks. Passing to fiberwise
consideration one gets a $k_t$-homomorphism $\pi_{i\ast} \widetilde \L_i^n \otimes k_t \to
p_{\ast}\L^n \otimes k_t$ or, equivalently, $H^0(\widetilde S_{i,t},
\widetilde L^n_{it})\to H^0(t\times S, L^n)$. This map is an
isomorphism and hence $\pi_{i\ast} \widetilde \L_i^n
\stackrel{\sim}{\to} p_{\ast}\L^n$.\end{proof}
%%%%%%%%%%%%%%%%%%%%%%%%%%%%%%%%%%%%%%%%%%%%%%%%%%%%%%%%%%%%%%%%
%%%%%%%%%%%%%%%%%%%%%%%%%%%%%%%%%%%%%%%%%%%%%%%%%%%%%%%%%%%%%%%%

We will need sheaves $$\widetilde \V_m=\pi_{\ast}(\widetilde \E
\otimes \widetilde \L^m)$$ for $m\gg 0$ such that $\widetilde
\V_m$ are locally free of rank $rp(m)$ and $\widetilde \E \otimes
\widetilde \L^m$ are fiberwise globally generated in such sense
that the canonical morphisms
$$\pi^{\ast} \widetilde \V_m \to \widetilde
\E \otimes \widetilde \L^m$$ are surjective for those $m$'s.

Let also for $m\gg 0$ $$\EE_m=\sigma\!\!\!\sigma_{\ast} (\widetilde \E \otimes
\widetilde \L^m),$$ now $$p_{\ast} \EE_m = p_{\ast}
\sigma\!\!\!\sigma_{\ast}(\widetilde \E\otimes \widetilde \L^m)=
\pi_{\ast}(\widetilde \E\otimes \widetilde \L^m)=\widetilde
\V_m.$$

We intend to confirm  that the sheaves $p_{\ast} (\EE_m \otimes \L^n)$
are locally free of rank~$rp(m+n)$ for all $m,n\gg 0$. This
implies $T$-flatness of $\EE_m$.

To proceed further we need morphisms $\widetilde \L^n \to
\sigma\!\!\!\sigma^{\ast} \L^n,$ $n>0.$

\begin{proposition} \label{injinv} For all $n>0$ there are injective morphisms
$\iota_n\colon \widetilde \L^n \to \sigma\!\!\!\sigma^{\ast}\L^n$ of invertible
$\OO_{\widetilde \Sigma}$-sheaves.
\end{proposition}
\begin{proof}
For $n>0$ and for any open $U\subset \widetilde \Sigma$ there is a
restriction map on sections
$$\widetilde \L^n(U) \stackrel{res}{\longrightarrow}
\widetilde \L^n(U\cap \widetilde \Sigma_0)= \L^n(\sigma\!\!\!\sigma_0(U\cap
\widetilde \Sigma_0))= \L^n(\sigma\!\!\!\sigma_0(U) \cap
\Sigma_0)=\L^n(\sigma\!\!\!\sigma(U)).
$$
Since $\sigma\!\!\!\sigma$ is projective and hence takes closed subsets to
closed subsets (resp., open subsets to open subsets), this implies the sheaf
morphism $\widetilde \L^n \to \sigma\!\!\!\sigma^{-1}\L^n$. Combining it with
multiplication by the unity section $1\in \OO_{\widetilde \Sigma}(U)$
leads to the morphism $\iota_n\colon \widetilde \L^n \to \sigma\!\!\!\sigma^{\ast}
\L^n$ of invertible $\OO_{\widetilde \Sigma}$-modules.
\end{proof}

%%%%%%%%%%%%%%%%%%%%%%%%%%%%%%%%%%%%%%%%%%%%%%%%%%%%%%%%%%%%%%%%
%%%%%%%%%%%%% SHEAVES \EE_m ARE T-FLAT FOR n>>0 %%%%%%%%%%%%%%%%
%%%%%%%%%%%%%%%%%%%%%%%%%%%%%%%%%%%%%%%%%%%%%%%%%%%%%%%%%%%%%%%%
\begin{proposition}\label{fltnss} $\EE_m$ are $T$-flat for $m\gg 0.$
\end{proposition}
\begin{proof} Consider the morphism of multiplication of sections
$$
p_{\ast} \sigma\!\!\!\sigma_{\ast} (\widetilde \E \otimes \widetilde \L^m)
\otimes p_{\ast} \L^n \to p_{\ast}(\sigma\!\!\!\sigma_{\ast}(\widetilde \E
\otimes \widetilde \L^m) \otimes \L^n)$$ which is surjective for
$m,n\gg 0$. By the projection formula
$$p_{\ast}(\sigma\!\!\!\sigma_{\ast}(\widetilde \E \otimes \widetilde \L^m)
\otimes \L^n)= p_{\ast}(\sigma\!\!\!\sigma_{\ast}(\widetilde \E \otimes
\widetilde \L^m\otimes \sigma\!\!\!\sigma^{\ast}\L^n).
$$
Also for the projection $\pi$ we have another morphism of
multiplication of sections
$$\pi_{\ast}(\widetilde \E \otimes \widetilde \L^m)
\otimes \pi_{\ast} \widetilde \L^n \to \pi_{\ast}(\widetilde \E
\otimes \widetilde \L^{m+n}).
$$
The injective $\OO_{\widetilde \Sigma}$-morphism $\widetilde \L^n
\hookrightarrow \sigma\!\!\!\sigma^{\ast}\L^n$ in Proposition \ref{injinv} after tensoring by $\widetilde \E
\otimes \widetilde \L^m$ and after applying  $\pi_\ast$ leads to
$$ \pi_{\ast}(\widetilde \E \otimes \widetilde \L^{m+n}) \hookrightarrow
\pi_{\ast}(\widetilde \E \otimes \widetilde \L^m\otimes
\sigma\!\!\!\sigma^{\ast}\L^n)
$$
Taking into account the isomorphism $p_{\ast}\L^n=
\pi_{\ast}\widetilde \L^n$ and Proposition \ref{injinv}, we gather
these mappings into the commutative diagram
\begin{equation} \label{diagEmn}\xymatrix{
p_{\ast} \sigma\!\!\!\sigma_{\ast}(\widetilde \E \otimes \widetilde \L^m)\otimes
p_{\ast}\L^n \ar@{=}[d] \ar@{->>}[r]&
p_{\ast}(\sigma\!\!\!\sigma_{\ast}(\widetilde \E \otimes \widetilde \L^m \otimes
\sigma\!\!\!\sigma^{\ast}\L^n))\\
\pi_{\ast}(\widetilde \E \otimes \widetilde L^m)\otimes
\pi_{\ast}\widetilde \L^n \ar@{->>}[r]& \pi_{\ast}(\widetilde
\E\otimes \widetilde \L^{m+n}) \ar@{^(->}[u]}\end{equation} By
commutativity of this diagram we conclude that
\begin{equation} \label{eqtydir}\pi_{\ast}(\widetilde \E\otimes \widetilde \L^{m+n})=
p_{\ast}(\sigma\!\!\!\sigma_{\ast}(\widetilde \E \otimes \widetilde \L^m)
\otimes \L^n))\end{equation} or, in our notation,
$p_{\ast}(\EE_m\otimes \L^n)=\widetilde \V_{m+n}=p_\ast \EE_{m+n}$ for $m,n \gg 0$.
This guarantees that the sheaves $\EE_m$ are $T$-flat for $m\gg 0.$\end{proof}

We intend to confirm  that $\EE_m\otimes \L^{-m}$ are families of
semistable sheaves on $S$ as we need. First we prove  following
\begin{proposition} $\EE_{m+n}=\EE_m \otimes \L^n$ for any $m\gg 0,$ $n>0$.
\end{proposition}
\begin{proof}
Consider a sheaf inclusion $\iota_n\colon \widetilde \L^n
\hookrightarrow \sigma\!\!\!\sigma^{\ast}\L^n$ valid for any $n>0$. Tensoring by
the locally free $\OO_{\widetilde \Sigma}$-module $\widetilde \E
\otimes \widetilde \L^m$, formation of a direct image under $\sigma\!\!\!\sigma$
and the projection formula yield  the inclusion \begin{equation}i_{m,n}\colon \EE_{m+n}
\hookrightarrow \EE_m \otimes \L^n. \label{imn}
\end{equation}

By (\ref{diagEmn}) we come to the epimorphism $p_\ast \EE_m \otimes p_\ast \L^n \twoheadrightarrow p_\ast \EE_{m+n}$. Now apply $p^\ast$:
$$p^\ast ( p_\ast \EE_m \otimes p_\ast \L^n )\twoheadrightarrow p^\ast p_\ast \EE_{m+n}
$$ 
and combine this morphism with natural morphisms
\begin{eqnarray*}
p^\ast ( p_\ast \EE_m \otimes p_\ast \L^n ) \twoheadrightarrow \EE_m \otimes  \L^n,\\
p^\ast p_\ast \EE_{m+n} \to \EE_{m+n}
\end{eqnarray*}
and with (\ref{imn}): 
\begin{equation} \label{Emn}
    \xymatrix{p^\ast (p_\ast \EE_m\otimes p_\ast \L^n) \ar@{->>}[d] \ar@{->>}[r]& p^\ast p_\ast \EE_{m+n}\ar[d]\\
    \EE_m \otimes \L^n &\ar@{_(->}[l]_{i_{m,n}} \EE_{m+n}}
\end{equation}
By commutativity of (\ref{Emn}) we conclude that $i_{m,n}$ is surjective as well as injective and hence $\EE_m \otimes \L^n=\EE_{m+n}$ whenever $m\gg 0,$ $n\gg 0$.
This proves the proposition.
\end{proof}

 We can introduce the goal sheaf of our construction
$$\E:=\EE_m \otimes \L^{-m}.$$
By the proposition proved this definition is independent of $m$ at
least in case when $m\gg 0.$ The sheaf $\E$ is $T$-flat.

\begin{proposition}\label{hpol} The sheaf $\E$ with respect to the invertible
sheaf $\L$ has fiberwise Hilbert polynomial equal to $rp(n),$ i.e.
for $n\gg 0$
$$\rank p_{\ast} (\E \otimes \L^n) =rp(n).$$
\end{proposition}
\begin{proof}
For $n\gg m\gg 0$ by (\ref{eqtydir}) we have the chain of equalities
$$p_{\ast} (\E \otimes \L^n)= p_{\ast}(\EE_m \otimes \L^{n-m})=
p_{\ast}(\sigma\!\!\!\sigma_{\ast}(\widetilde \E \otimes \widetilde \L^m)
\otimes \L^{n-m})= \pi_{\ast}(\widetilde \E \otimes \widetilde
\L^n).$$ The latter sheaf of the chain has rank equal to $rp(n).$
\end{proof}

%%%%%%%%%%%%%%%%%%%%%%%%%%%%%%%%%%%%%%%%%%%%%%%%%%%%%%%%%%%%%%%%
%%%%%%%%%%%%%%%%%%%%%%%%%%%%%%%%%%%%%%%%%%%%%%%%%%%%%%%%%%%%%%%%
\begin{proposition} For any closed point $t\in T$ the sheaf $$
E_t:=\E|_{t\times S}
$$
is torsion-free and Gieseker-semistable with respect to
$$
L_t:=\L|_{t\times S}\cong L.
$$
\end{proposition}
\begin{proof}
The isomorphism $\L|_{t\times S}\cong L$ is the subject of
Proposition \ref{polariz}. Now for $E_t$ one has
$$E_t=\E|_{t\times S}=(\EE_m \otimes \L^{-m})|_{t\times S}=
\EE_m|_{t\times S} \otimes L^{-m}=\sigma\!\!\!\sigma_{\ast} (\widetilde \E
\otimes \widetilde \L^m)|_{t\times S}\otimes L^{-m}.$$ Denoting the morphisms of
closed immersions of fibers by
$i_t\colon t\times S \hookrightarrow T\times S$ and $\widetilde i_t \colon 
\widetilde S_t \hookrightarrow \widetilde \Sigma$ 
 we come to the equality
$$\sigma\!\!\!\sigma_{\ast} (\widetilde \E
\otimes \widetilde \L^m)|_{t\times S}\otimes L^{-m}
=(i_t^{\ast}\sigma\!\!\!\sigma_{\ast} (\widetilde \E \otimes \widetilde
\L^m))\otimes L^{-m}$$ and to the base change morphism
\begin{equation} \label{bchfib}\beta_t\colon i_t^{\ast}\sigma\!\!\!\sigma_{\ast} (\widetilde \E
\otimes \widetilde \L^m) \to \sigma_{t \ast}\widetilde
i_t^{\ast}(\widetilde \E \otimes \widetilde \L^m)
\end{equation}
 in the fibered square
$$
\xymatrix{t\times S \ar@{^(->}[r]^{i_t}&T\times S\\
\widetilde S_t \ar[u]^{\sigma_t} \ar@{^(->}[r]^{\widetilde i_t}&
\widetilde \Sigma \ar[u]_{\sigma\!\!\!\sigma} }
$$
The following lemma will be proven
below.
\begin{lemma}\label{torsfree} The sheaf $\sigma_{t\ast} \widetilde E_t$ is torsion-free.
\end{lemma}
Both sheaves in (\ref{bchfib}) coincide along the open subset $(t\times S) \cap
\Sigma_0.$ Now consider the map of global sections corresponding to the morphism (\ref{bchfib}):
$$H^0(\beta_t)\colon H^0(t\times S, i_t^{\ast}\sigma\!\!\!\sigma_{\ast} (\widetilde \E
\otimes \widetilde \L^m)) \to H^0(t\times S,\sigma_{t \ast}
\widetilde i_t^{\ast}(\widetilde \E \otimes \widetilde \L^m)).$$
It is injective. The left hand side takes the view
$$H^0(t\times S, i_t^{\ast}\sigma\!\!\!\sigma_{\ast} (\widetilde \E
\otimes \widetilde \L^m))\otimes k_t= i_t^\ast p_{\ast}\sigma\!\!\!\sigma_{\ast}
(\widetilde \E \otimes \widetilde \L^m)=k_t^{\oplus rp(m)}.$$ Also
for the right hand side one has $$H^0(t\times S,\sigma_{t \ast}
\widetilde i_t^{\ast}(\widetilde \E \otimes \widetilde \L^m))
\otimes k_t= H^0(\widetilde S_t, \widetilde E_t \otimes \widetilde
L_t^m)\otimes k_t = k_t^{\oplus rp(m)}.$$ This implies that
$H^0(\beta_t)$ is bijective and there is a commutative diagram
\begin{equation}\label{beta}
\xymatrix{H^0(t\times S, \sigma_{t\ast} (\widetilde E_t\otimes
\widetilde L_t^m))\otimes \OO_S \ar@{=}[d]_{H^0(\beta_t)}
\ar@{->>}[r]&
\sigma_{t\ast}(\widetilde E_t\otimes \widetilde L_t^m)\\
i_t^{\ast} p^{\ast}\widetilde \V_m \ar@{->>}[r]& i_t^{\ast}\EE_m
\ar[u]_{\beta_t}} \end{equation} Observe that  $$\sigma_{t\ast}
(\widetilde E_t \otimes \widetilde L_t^m)=\sigma_{t\ast}
(\widetilde E_t \otimes \sigma_t^{\ast} L^m \otimes (\mathrm{Exc}_{\widetilde S_t})^m)=\sigma_{t\ast} (\widetilde E_t
\otimes (\mathrm{Exc}_{\widetilde S_t})^m)\otimes
L^m,$$ and for  $m\gg 0$ the latter sheaf is globally generated. We have used the notations (\ref{polshort}) and (\ref{polshortexc}).
This implies surjectivity of the upper horizontal arrow in
(\ref{beta}). It follows from (\ref{beta}) that $\beta_t$ is
surjective. Since $\ker H^0(\beta_t)=H^0(t\times S,\ker
\beta_t)=0,$ the morphism $\beta_t$ is isomorphic.

Now take a subsheaf $F_t \subset E_t$. Now for $m \gg 0$ there is
a commutative diagram
$$\xymatrix{H^0(t\times S, E_t\otimes L^m) \otimes \OO_S \ar@{->>}[r]& E_t \otimes L^m\\
H^0(t\times S, F_t\otimes L^m) \otimes \OO_S \ar@{^(->}[u]
\ar@{->>}[r]& F_t \otimes L^m \ar@{^(->}[u]}
$$
The isomorphism $E_t\otimes L^m= \sigma_{t\ast}(\widetilde E_t
\otimes \widetilde L_t^m)$ proven above fixes a bijection on global
sections $H^0(t\times S, E_t \otimes L^m)\simeq H^0(t\times S,
\sigma_{t\ast}(\widetilde E_t \otimes \widetilde L_t^m))=
H^0(\widetilde S_t, \widetilde E_t \otimes \widetilde L_t^m).$ Let
$\widetilde V_t \subset H^0(\widetilde S_t, \widetilde E_t \otimes
\widetilde L_t^m)$ be the subspace corresponding to $H^0(t\times
S, F_t \otimes L^m)\subset H^0(t\times S, E_t \otimes L^m)$ under
this bijection. Now one has a commutative diagram
$$
\xymatrix{ H^0(\widetilde S_t, \widetilde E_t \otimes \widetilde
L_t^m) \otimes \OO_{\widetilde S_t} \ar@{->>}[r]^>>>>>\varepsilon&
\widetilde E_t \otimes
\widetilde L_t^m\\
\widetilde V_t \otimes \OO_{\widetilde S_t} \ar@{^(->}[u]_\gamma
\ar@{->>}[r]^{\varepsilon'}& \widetilde F_t \otimes \widetilde L_t
\ar@{^(->}[u]}
$$
where $\widetilde F_t \otimes \widetilde L_t^m \subset \widetilde
E_t \otimes \widetilde L_t^m$ is defined as a subsheaf generated
by the subspace $\widetilde V_t$ by means of the morphism
$\varepsilon\circ \gamma$ and $\gamma$ is the morphism induced by the inclusion $\widetilde V_t \subset H^0(\widetilde S_t, \widetilde E_t \otimes
\widetilde L_t^m)$. The associated map of global sections
$$
H^0(\varepsilon')\colon \widetilde V_t \to H^0(\widetilde S_t,
\widetilde F_t \otimes \widetilde L_t^m)$$ is  included into the
commutative triangle
$$
\xymatrix{\widetilde V_t \ar@{_(->}[rd]
\ar[r]^<<<<<<{H^0(\varepsilon')} &H^0(\widetilde S_t,
\widetilde F_t \otimes \widetilde L_t^m) \ar@{^(->}[d]\\
& H^0(\widetilde S_t, \widetilde E_t \otimes \widetilde L_t^m) }
$$
what implies that $H^0(\varepsilon')$ is injective. Since each
section from $H^0(\widetilde S_t, \widetilde F_t \otimes
\widetilde L_t^m)$ corresponds to a section in $H^0(t\times S, F_t
\times L^m)\subset H^0(t\times S, E_t \otimes L^m)$, then
$H^0(\varepsilon')$ is surjective. Hence $h^0(t\times S, F_t
\otimes L^m)= h^0(\widetilde S_t, \widetilde F_t \otimes
\widetilde L_t^m)$ for all $m\gg 0$, and stability (resp.,
semistability) for $\widetilde E_t$ implies stability (resp.,
semistability) for $E_t.$
\end{proof}
%%%%%%%%%%%%%%%%%%%%%%%%%%%%%%%%%%%%%%%%%%%%%%%%%%%%%%%%%%%%%%%%
%%%%%%%%%%%%%%%%%%%%%%%%%%%%%%%%%%%%%%%%%%%%%%%%%%%%%%%%%%%%%%%%
\begin{proof}[Proof of Lemma \ref{torsfree}] Since the sheaves
$\widetilde E_t$ and $\sigma_{t\ast} \widetilde E_t$ coincide
along the identified open subschemes $\widetilde S_t \cap \widetilde
\Sigma_0 \simeq (t\times S) \cap \Sigma_0$, it is enough to confirm
that there is no torsion subsheaf concentrated on $t\times S \cap
(T\times S \setminus \Sigma_0)$ in $\sigma_{t\ast} \widetilde
E_t$. Assume that $T\subset \sigma_{t\ast} \widetilde E_t$ is such
a torsion subsheaf, i.e. $T \ne 0$ and for any open $U \subset
t\times S \cap \Sigma_0$ one has $T(U)=0$. Let $A =\Supp T \subset t\times
S$ and $U\subset t\times S$ be such an open subset that $T(U)\ne
0,$ i.e. $U\cap A \ne \emptyset.$ Now
$$
T(U)\subset \sigma_{t\ast} \widetilde E_t (U)= \widetilde E_t
(\sigma^{-1}U),
$$
and any nonzero section $s\in T(U)$ is supported in $U\cap A$ and
comes from the section $\widetilde s\in \widetilde
E_t(\sigma^{-1}U)$ with support in $\sigma^{-1}(U\cap A).$ This
means that $\widetilde s$ is supported on some additional
component $\widetilde S^{add}_{t,j}$ of the admissible scheme
$\widetilde S_t$. Hence $\widetilde s \in \tors\!_j$. But on the additional components
$\widetilde S^{add}_{t,j}$ of $\widetilde S_t$ necessarily  $\tors\!_j=0.$ This
implies that $T=0$.
\end{proof}

\section{The moduli scheme isomorphism}\label{s5}
In this section we investigate how the natural transformations $\underline
\kappa\colon {\mathfrak f}^{GM}\to {\mathfrak f}$ and $\underline \tau\colon 
{\mathfrak f}\to {\mathfrak f}^{GM}$ are related to each other. By this relation we get an isomorphism
of moduli schemes for these moduli functors, with no dependence on
the number and the geometry of their connected components, on reducedness
of the scheme structure and on presence of locally free sheaves
(respectively, of $S$-pairs) in each component.

The subject has two aspects.
\begin{enumerate}
\item{{\it Pointwise.} {\it a}) For any torsion-free semistable
$\OO_S$-sheaf the composite of transform\-at\-ions $$ E \mapsto
((\widetilde S, \widetilde L), \widetilde E) \mapsto E'$$
returns $E'=E.$ \\
{\it b}) Conversely, for any admissible semistable pair
$((\widetilde S, \widetilde L), \widetilde E)$ the composite of
trans\-form\-at\-ions $$ ((\widetilde S, \widetilde L), \widetilde
E)\mapsto E \mapsto ((\widetilde S', \widetilde L'), \widetilde
E')$$ returns a pair $((\widetilde S', \widetilde L'), \widetilde E')$ which is M-equivalent to 
$((\widetilde S, \widetilde L), \widetilde E).$}
\item{{\it Global.} {\it a}) For any family of semistable torsion-free
sheaves $\E$ with a base scheme $T$ the composite
$$(\E, \L) \mapsto ((\pi\colon \widetilde \Sigma \to T, \widetilde \L),
\widetilde \E) \mapsto (\E', \L')
$$ returns $(\E', \L')\sim (\E, \L)$ in the sense of the description (\ref{funcGM},~\ref{famGM})
of the functor~${\mathfrak f}^{GM}$. \\
{\it b}) Conversely, for any family $((\pi \colon \widetilde \Sigma \to
T, \widetilde \L), \widetilde \E)$ of admissible semistable pairs
with a base scheme $T$ the composite of transformations $$
((\pi\colon \widetilde \Sigma \to T, \widetilde \L), \widetilde \E)
\mapsto (\E,\L) \mapsto ((\pi'\colon \widetilde \Sigma' \to T,
\widetilde \L'), \widetilde \E')
$$ returns a family $((\pi'\colon \widetilde \Sigma' \to T, \widetilde \L'),
\widetilde \E')$ such that it is M-equivalent to the  family $((\pi\colon \widetilde \Sigma \to T, \widetilde \L), \widetilde \E)$ in the following strict sense: for $\widetilde \Sigma \stackrel{\widehat \sigma\!\!\!\sigma}{\longleftarrow} \widetilde \Sigma_\Delta \stackrel{\widehat \sigma\!\!\!\sigma'}{\longrightarrow} \widetilde \Sigma'$ the inverse images of~$\widetilde E$ and of~$\widetilde \E'$ on the scheme~$\widetilde \Sigma_\Delta$ are equivalent (generally S-equivalence is enough but there is an isomorphism) in the sense of the description  (\ref{funcmy}, \ref{class}) of the functor
${\mathfrak f}$.}
\end{enumerate}

Observe that $\underline \kappa$ operates by standard resolution and its target consists of classes of admissible pairs which are obtained by resolutions of coherent sheaves. At the same time classes in the source of $\underline \tau$ are essentially wider then (i.e. non-bijective to) ones in its target.
%%%%%%%%%%%%%%%%%%%%%%%%%%%%%%%%%%%%%%%%%%%%%%%%%%%%%%%%%%%%%%%%
%%%%%%%%%%%%%%%%%%%%%%% composite a) %%%%%%%%%%%%%%%%%%%%%%%%%%
%%%%%%%%%%%%%%%%%%%%%%%%%%%%%%%%%%%%%%%%%%%%%%%%%%%%%%%%%%%%%%%%

We begin with 2 {\it a}); it will be specialised to the pointwise
version 1 {\it a}) when $T=\Spec k.$

The families of polarizations $\L$ and $\L'$ coincide along the open
subset (locally free locus for the sheaves $\E$ and $\E'$) $\Sigma_0$
where $\codim _{T\times S} T\times S \setminus \Sigma_0\ge 2$.
Since $\L$ and $\L'$ are locally free,  this implies that $\L=\L'$.

%%%%%%%%%%%%%%%%%%%%%%%%%%%%%%%%%%%%%%%%%%%%%%%%%%%%%%%%%%%%%%%%
%%%%%%%%%%%% THREE SHEAVES OF "FIBERWISE SECTIONS" %%%%%%%%%%%%%
%%%%%%%%%%%%%%%%%%%%%%%%%%%%%%%%%%%%%%%%%%%%%%%%%%%%%%%%%%%%%%%%

Now consider three locally free $\OO_T$-sheaves of ranks equal to
$rp(m)$:
$$\V_m=p_{\ast}(\E\otimes \L^m),\quad
\widetilde \V_m=\pi_{\ast} (\widetilde \E\otimes \widetilde
\L^m),\quad \V'_m=p_{\ast}(\E'\otimes \L^m).$$
\begin{lemma} $\V_m \cong \widetilde \V_m \cong \V'_m$.
\end{lemma}
\begin{proof} Start with the epimorphism
$\sigma\!\!\!\sigma^{\ast} \E \twoheadrightarrow \widetilde \E$ where $\sigma\!\!\!\sigma\colon \widetilde \Sigma \to \Sigma$ is well defined as morphism of the resolution.
Tensoring  by $\widetilde \L^m$ and the direct image
$\sigma\!\!\!\sigma_{\ast}$ yield  the morphism of
$\OO_T$-sheaves
$$\sigma\!\!\!\sigma_{\ast}(\sigma\!\!\!\sigma^{\ast} \E \otimes
\widetilde \L^m) \rightarrow \sigma\!\!\!\sigma_{\ast}(\widetilde
\E \otimes \widetilde \L^m).
$$
Turn our attention to the following result.
\begin{lemma}[\cite{Tim9}] \label{profor} Let $f\colon (X,\OO_X) \to (Y, \OO_Y)$ be a morphism of
locally ringed spaces such that $f_{\ast}\OO_X=\OO_Y$, $\EE$ an
$\OO_Y$-module of finite presentation, $\FF$ an $\OO_X$-module. Then
there is a monomorphism $\EE \otimes f_{\ast}\FF \hookrightarrow
f_{\ast}[f^{\ast}\EE \otimes \FF].$
\end{lemma}
\begin{remark} In a general case we have only the morphism 
$\OO_Y \to f_{\ast} \OO_X$ in our disposal. Applying the inverse
image $f^{\ast}$, tensoring by $\otimes \FF$ and the direct image
$f_{\ast}$ to a finite presentation $E_1 \to E_0 \to \EE \to 0$ we
come to the commutative diagram
$$\xymatrix{f_{\ast} [f^{\ast}E_1 \otimes \FF]\ar[r]&
f_{\ast} [f^{\ast}E_0 \otimes \FF]\ar[r]&
f_{\ast} [f^{\ast}\EE \otimes \FF] \\
E_1 \otimes f_{\ast} \FF \ar[u] \ar[r]&E_0 \otimes f_{\ast} \FF
\ar[u] \ar[r]&\EE \otimes f_{\ast} \FF \ar[u] \ar[r]&0}
$$
where the right hand side is the morphism of interest
$\EE \otimes f_{\ast} \FF \to f_{\ast} [f^{\ast}\EE \otimes
\FF].$
\end{remark}

Setting $\EE=\E,$ $\FF=\widetilde \L^m$, and $f=\sigma\!\!\!\sigma$,
we get
$$
\xymatrix{ \E \otimes \L^m \otimes \sigma\!\!\!\sigma_{\ast}
(\sigma\!\!\!\sigma^{-1} \I\cdot \OO_{\widetilde \Sigma})^m \ar[d] \\
\sigma\!\!\!\sigma_{\ast}(\sigma\!\!\!\sigma^{\ast}\E \otimes
\widetilde \L^m)\ar[r] & \sigma\!\!\!\sigma_{\ast}(\widetilde \E
\otimes \widetilde \L^m)}
$$
After taking the direct image $p_{\ast}$ we come to the commutative diagram 
$$
\xymatrix{ p_{\ast}(\E \otimes \L^m \otimes
\sigma\!\!\!\sigma_{\ast}
(\sigma\!\!\!\sigma^{-1} \I\cdot \OO_{\widetilde \Sigma})^m )\ar[d]\ar[dr]^{\eta}\ar@{^(->}[r]&p_{\ast}(\E \otimes \L^m) \\
p_{\ast}\sigma\!\!\!\sigma_{\ast}(\sigma\!\!\!\sigma^{\ast}\E
\otimes \widetilde \L^m)\ar[r] & \pi_{\ast}(\widetilde \E \otimes
\widetilde \L^m)}
$$
where the upper horizontal arrow is a natural immersion into a reflexive
hull. Since the target sheaf of the composite map $\eta$ is reflexive,
$\eta$ factors through $p_{\ast}(\E\otimes \L^m)$ as a reflexive
hull of the source. This yields  existence of the morphism of
locally free sheaves of equal ranks
$$
\widetilde \eta\colon p_{\ast}(\E\otimes \L^m) \to \pi_{\ast}(\widetilde
\E \otimes \widetilde \L^m).
$$
The morphism of sheaves is an isomorphism iff it is stalkwise
isomorphic. Fix an arbitrary closed point $t\in T$; till the end
of this proof we omit the subscript $t$ in the notations corresponding to
the sheaves corresponding to $t$: $\E_t=:E,$ $\widetilde
E_t=:\widetilde E,$ $\widetilde S_t=:\widetilde S,$
$\sigma_t=:\sigma$, $\sigma \colon \widetilde S \to S.$ There is an epimorphism of
$\OO_{\widetilde S}$-modules
$
\sigma^{\ast} E \otimes \widetilde L^m \twoheadrightarrow
\widetilde E \otimes \widetilde L^m.
$
The analog of the projection formula as in the global case leads to
$$
\xymatrix{ E \otimes L^m \otimes \sigma_{\ast}
(\sigma^{-1} I\cdot \OO_{\widetilde S})^m \ar@{^(->}[d] \\
\sigma_{\ast}(\sigma^{\ast}E \otimes \widetilde L^m)\ar[r] &
\sigma_{\ast}(\widetilde E \otimes \widetilde L^m)}
$$
After taking global sections we come to the commutative diagram 
$$
\xymatrix{H^0(S,E \otimes L^m \otimes \sigma_{\ast} (\sigma^{-1}
I\cdot \OO_{\widetilde S})^m )\ar@{^(->}[d]
\ar[dr]^{\eta_t}\ar@{^(->}[r]&H^0(S,E \otimes L^m)\ar[d]_{\widetilde \eta_t} \\
H^0(S,\sigma_{\ast}(\sigma^{\ast}E \otimes \widetilde L^m))\ar[r]
& H^0(\widetilde S,\widetilde E \otimes \widetilde L^m)}
$$
Here $\widetilde \eta_t$ is also included into the commutative diagram
$$
\xymatrix{H^0(S,E \otimes L^m)\ar[d]_{\widetilde \eta_t}\ar[r] &
H^0(t\times S\cap \Sigma_0),E \otimes L^m)\ar@{=}[d]\\
H^0(\widetilde S,\widetilde E \otimes \widetilde L^m) \ar[r]&
H^0(\widetilde S\cap \widetilde \Sigma_0, \widetilde E \otimes
\widetilde L^m)}
$$
where both horizontal arrows are restriction maps and the upper
restriction map is injective. Hence $\widetilde \eta_t$ is also
injective. Since it is a monomorphism of vector spaces of equal
dimensions is is an isomorphism. Then $\eta\colon \V_m \to \widetilde
\V_m $ is also an isomorphism. The isomorphism $\widetilde \V_m
\cong \V'_m$ has been proven in the previous section (proof of
Proposition~\ref{hpol}).
\end{proof}

Identifying the locally free sheaves $\V_m=\V'_m$ to each other  we consider a relative
Grothendieck scheme $\Quot^{rp(n)}_T (p^{\ast} \V_m \otimes
\L^{-m})$ and two $T$-morphisms of closed immersions
$$
T\times S \stackrel{\widetilde{ev}}{\hookrightarrow}
\Quot^{rp(n)}_T (p^{\ast} \V_m \otimes \L^{-m}) \times S
\stackrel{\widetilde{ev'}}{\hookleftarrow} T\times S.
$$
The morphism $\widetilde{ev}$ is induced by the morphism $ev\colon 
p^{\ast}\V_m \otimes \L^{-m}\twoheadrightarrow \E$ and
$\widetilde{ev'}$ by the morphism $ev'\colon p^{\ast}\V_m \otimes
\L^{-m}\twoheadrightarrow \E'$.

Since both morphisms $\widetilde{ev}$ and $\widetilde{ev'}$ are
proper and they coincide along the open subset~$\Sigma_0$ such that
$\codim_{T\times S} (T\times S) \setminus \Sigma_0\ge 2$, then
$\widetilde{ev}=\widetilde{ev'}$ and $\widetilde{ev}(T\times
S)=\widetilde{ev'}(T\times S)$ in the scheme sense. Hence by
universality of $\Quot$-scheme $\E=\E'$ as inverse images of
the universal quotient sheaf over the scheme $\Quot^{rp(n)}_T (p^{\ast} \V_m
\otimes \L^{-m})$ under the morphisms
$\widetilde{ev}=\widetilde{ev'}.$

For the global version of 2,{\it b}) consider two families $(\pi\colon \widetilde \Sigma \to
T, \widetilde \L)$ and $(\pi'\colon \widetilde \Sigma' \to T,
\widetilde \L')$ and two epimorphisms $\pi^{\ast}\pi_{\ast}\widetilde
\L \twoheadrightarrow \widetilde \L$ and
${\pi'}^{\ast}\pi'_{\ast}\widetilde \L' \twoheadrightarrow
\widetilde \L'$. In general, the families $\widetilde \Sigma$ and $\widetilde \Sigma'$ are not obliged to be isomorphic. We can assume that $\pi_{\ast} \widetilde \L =
\pi'_{\ast} \widetilde \L'$ and then identify the corresponding projective bundles
$\P(\pi_{\ast} \widetilde \L)^{\vee} =\P(\pi'_{\ast} \widetilde
\L')^{\vee}$. We introduce shorthand notations $\P:=\P(\pi_{\ast} \widetilde \L)^{\vee} $ and $\P:=\P(\pi'_{\ast} \widetilde \L')^{\vee}$ and consider a fibered product together with a diagonal embedding 
$$
\xymatrix{&\ar[ld]\P\times _T \P' \ar[dr]\\
\P \ar[dr]&\P_\Delta \ar@{^(->}[u]&\ar[ld] \P'\\
&T}
$$
Closed immersions of $T$-schemes
$i\colon \widetilde \Sigma \hookrightarrow \P$ and $i'\colon \widetilde \Sigma' \hookrightarrow \P'$
complete it to the commutative diagram
\begin{equation}\label{house}
\xymatrix{&&\ar[lldd]\widetilde \Sigma \times_T \widetilde \Sigma' \ar@{^(->}[d] \ar[ddrr]\\
&&\ar[ld]\P\times _T \P' \ar[dr]\\
\widetilde \Sigma \ar[drr]_\pi \ar@{^(->}[r]^i&\P \ar[dr]&\P_\Delta \ar@{^(->}[u]&\ar[ld] \P'&\ar@{_(->}[l]_{i'} \ar[lld]^{\pi'} \widetilde \Sigma'\\
&&T}
\end{equation}
Now we set $\widetilde \Sigma_\Delta:=(\widetilde \Sigma \times_T \widetilde \Sigma') \cap \P_\Delta$ and 
$i_\Delta\colon \widetilde \Sigma_\Delta \hookrightarrow \P_\Delta$. Using the standard notations for ample invertible sheaves on projective spaces we write $\widetilde \L=i^\ast \OO_{\P}(1),$ $\widetilde \L'=i{'^\ast} \OO_{\P'}(1)$, and $\widetilde \L_\Delta =i_\Delta^\ast \OO_{\P_\Delta}(1)$ for the polarization on $\widetilde \Sigma_\Delta.$

Regretting to the one point case (or reducing to a single fiber) we come to the following diagram for the principal components of the fibers:
$$\xymatrix{&\ar[ld]\widetilde S_{\Delta 0} \ar@{^(->}[d] \ar[rd]\\
\widetilde S_0& \ar[l] \widetilde S_0 \times \widetilde S'_0 \ar[r]& \widetilde S'_0}
$$
Here $\widetilde S_{\Delta 0}$ is a closure of the image of the diagonal embedding $U \hookrightarrow U\times U' \subset \widetilde S_0 \times \widetilde S'_0$ of an open subset $U \subset \widetilde S_0$ where $U$ is the maximal open subset where $\sigma\colon \widetilde S \to S$ is isomorphic, $U'$ the maximal open subset where $\sigma'\colon \widetilde S' \to S$ is isomorphic, and $U \cong U'.$

Now switch to immersions into relative Grassmannians. Let $\G:=\Grass(\widetilde \V_m,r)$ and $\G':=\Grass (\widetilde \V'_m, r)$ be the relative Grassmannians of $r$-dimensional quotient spaces in the vector bundles $\widetilde \V_m \cong \widetilde \V'_m$. Let $\QQ$ be a universal quotient sheaf for $\G$ and $\QQ'$  the same for $\G'$, then $\widetilde \E \otimes \widetilde \L^m=j^\ast \QQ$, $\widetilde \E' \otimes \widetilde \L{'^m}=j{'^\ast} \QQ'$ for the closed immersions $j\colon \widetilde \Sigma \hookrightarrow \G$ and $j'\colon \widetilde \Sigma' \hookrightarrow \G'$ respectively.

Considering fibered products of Grassmannians analogous to the ones of projective spaces we come to analogous commutative diagram with a diagonal $\G_\Delta \hookrightarrow \G \times_T \G'$
$$
\xymatrix{&&\ar[lldd]\widetilde \Sigma \times_T \widetilde \Sigma' \ar@{^(->}[d] \ar[ddrr]\\
&&\ar[ld]\G\times _T \G' \ar[dr]\\
\widetilde \Sigma \ar[drr]_\pi \ar@{^(->}[r]^j&\G \ar[dr]&\G_\Delta \ar@{^(->}[u]&\ar[ld] \G'&\ar@{_(->}[l]_{j'} \ar[lld]^{\pi'} \widetilde \Sigma'\\
&&T}
$$
Also $\widetilde \Sigma_\Delta=(\widetilde \Sigma \times _T \widetilde \Sigma') \cap \G_\Delta$ (with the same $\widetilde \Sigma_\Delta$), $j_\Delta\colon \widetilde \Sigma_\Delta\hookrightarrow \G_\Delta$ is the corresponding closed immersion, and the equality $\widetilde \E_\Delta \otimes \widetilde \L^m_\Delta=j_\Delta^\ast \QQ_\Delta$ defines a locally free sheaf $\widetilde \E_\Delta$ on $\widetilde \Sigma_\Delta.$

Now we came to the family $\pi_\Delta\colon \widetilde \Sigma_\Delta \to T$ included into the commutative diagram 
$$\xymatrix{&\ar[ld]_{\widehat \sigma\!\!\!\sigma}\widetilde \Sigma_\Delta \ar[rd]^{\widehat \sigma\!\!\!\sigma'}\\
\widetilde \Sigma \ar[dr]_{\sigma\!\!\!\sigma}&&\ar[ld]^{\sigma\!\!\!\sigma'} \widetilde \Sigma'\\
&\Sigma}
$$
Start with $\widetilde \E$ and prove that $\widehat \sigma\!\!\!\sigma^\ast \widetilde \E\cong \widehat \sigma\!\!\!\sigma{'^\ast} \sigma\!\!\!\sigma{'^\ast}(\EE_m \otimes \L^{-m})/\tors$ for $m\gg 0.$ We use an obvious notation $\widetilde \E'=\sigma\!\!\!\sigma{'^\ast}(\EE_m \otimes \L^{-m})/\tors$ and a symbol $U_\Delta$ for an open subset of $\widetilde \Sigma_\Delta.$ If restricted to an irreducible component, $\widehat \sigma\!\!\!\sigma$ and $\widehat \sigma\!\!\!\sigma'$ take open subsets to open subsets. If any of the schemes $\widetilde \Sigma,$ $\widetilde \Sigma'$ or $\widetilde \Sigma_\Delta$ has more then one component, we replace it by its principal component (a principal component of, say, the scheme $\widetilde \Sigma$ is the one containing the union of principal components of fibers of the morphism $\pi\colon \widetilde \Sigma \to T$). Now we can write down the sheaves of interest explicitly.
\begin{eqnarray*}\widehat \sigma\!\!\!\sigma^\ast \widetilde \E&=&(U_\Delta \mapsto (\widehat \sigma\!\!\!\sigma^{-1} \widetilde \E)(U_\Delta) \otimes_{\widehat \sigma\!\!\!\sigma^{-1}\OO_{\widetilde \Sigma}(U_\Delta)}\OO_{\widetilde \Sigma_\Delta}(U_\Delta))^+\\
&=&(U_\Delta \mapsto  \lim_{\widetilde V \supset \widehat \sigma\!\!\!\sigma (U_\Delta)}\widetilde \E( \widetilde V) \otimes_{\lim_{\widetilde W \supset \widehat \sigma\!\!\!\sigma (U_\Delta)}\OO_{\widetilde \Sigma}( \widetilde W)}\OO_{\widetilde \Sigma_\Delta}(U_\Delta))^+.\end{eqnarray*}
For $\widehat \sigma\!\!\!\sigma{'^\ast} \widetilde \E'$ we use the fact that $\sigma\!\!\!\sigma \circ \widehat \sigma\!\!\!\sigma =\sigma\!\!\!\sigma ' \circ \widehat \sigma\!\!\!\sigma '$, the symbol $\widetilde V'$ for an open subset in $\widetilde \Sigma'$, $\widetilde V$ for an open subset in $\widetilde \Sigma$, $V$ for an open subset in $\Sigma$. 
\begin{eqnarray*}
&&\!\!\widehat \sigma\!\!\!\sigma{'^\ast} \widetilde \E'=(U_\Delta \!\mapsto  \! \lim_{\widetilde V' \supset \widehat \sigma\!\!\!\sigma' (U_\Delta)}\widetilde \E'( \widetilde V') \otimes_{\lim_{\widetilde W' \supset \widehat \sigma\!\!\!\sigma' (U_\Delta)}\OO_{\widetilde \Sigma'}( \widetilde W')}\OO_{\widetilde \Sigma_\Delta}(U_\Delta))^+\\
&&\!\!=(U_\Delta \!\mapsto  \! \lim_{\widetilde V' \supset \widehat \sigma\!\!\!\sigma' (U_\Delta)} (\sigma\!\!\!\sigma{'^\ast}\E/\tors) ( \widetilde V') \otimes_{\lim_{\widetilde W' \supset \widehat \sigma\!\!\!\sigma' (U_\Delta)}\OO_{\widetilde \Sigma'}( \widetilde W')}\OO_{\widetilde \Sigma_\Delta}(U_\Delta))^+\\
&&\!\!=(U_\Delta \!\mapsto  \! \lim_{V \supset \sigma\!\!\!\sigma \widehat \sigma\!\!\!\sigma (U_\Delta)} \E ( V) \otimes_{(\widehat \sigma\!\!\!\sigma^{-1}\sigma\!\!\!\sigma^{-1} \OO_\Sigma) (U_\Delta)}\OO_{\widetilde \Sigma_\Delta}(U_\Delta)/\tors)^+\\
&&\!\!=(U_\Delta \!\mapsto   \!\lim_{V \supset \sigma\!\!\!\sigma \widehat \sigma\!\!\!\sigma (U_\Delta)} \!\sigma\!\!\!\sigma_\ast(\widetilde \E \otimes \widetilde \L^m)( V)\otimes _{\OO_{\Sigma(V)}}\L^{-m}(V) \otimes_{(\widehat \sigma\!\!\!\sigma^{-1}\sigma\!\!\!\sigma^{-1} \OO_\Sigma) (U_\Delta)}\OO_{\widetilde \Sigma_\Delta}(U_\Delta)/\tors)^+\\
&&\!\!=(U_\Delta \!\mapsto \! \lim_{V \supset \sigma\!\!\!\sigma \widehat \sigma\!\!\!\sigma (U_\Delta)} (\widetilde \E \otimes \widetilde \L^m)(\sigma\!\!\!\sigma^{-1} V)\otimes _{\OO_{\Sigma(V)}}\L^{-m}(V) \otimes_{(\widehat \sigma\!\!\!\sigma^{-1}\sigma\!\!\!\sigma^{-1} \OO_\Sigma) (U_\Delta)}\OO_{\widetilde \Sigma_\Delta}(U_\Delta)/\tors)^+\\
&&\!\!=(U_\Delta \!\!\mapsto  \!\lim_{V \supset \sigma\!\!\!\sigma \widehat \sigma\!\!\!\sigma (U_\Delta)}\!\!\widetilde \E (\sigma\!\!\!\sigma^{-1} \!V) \otimes_{\OO_{\Sigma}(\sigma\!\!\!\sigma^{-1} \!V)} \widetilde \L^m(\sigma\!\!\!\sigma^{-1}\! V)\!\otimes _{\OO_{\Sigma(V)}}\!\L^{-m}(V) \!\\
&&\;\;\;\;\;\;\;\;\;\;\;\;\;\;\;\;\;\;\;\;\;\;\;\;\;\;\;\;\;\;\;\;\;\;\;\;\;\;\;\;\;\;\;\;\;\;\;\;\;\;\;\;\;\;\;\;\;\;\;\;\;\;\;\;\;\;\;\;\;\;\;\;\;\;\;\;\;\;\otimes_{(\widehat \sigma\!\!\!\sigma^{-1}\sigma\!\!\!\sigma^{-1} \OO_\Sigma) (\!U_\Delta\!)}\OO_{\widetilde \Sigma_\Delta}\!\!(\!U_\Delta\!)\!/\!\tors)^+.
\end{eqnarray*}
Shrinking  $U_\Delta$ one comes to $\widetilde V \subset \sigma\!\!\!\sigma^{-1} V$ and to the restriction map $\widetilde \E(\sigma\!\!\!\sigma^{-1} V)\stackrel{res}{\longrightarrow} \widetilde \E (\widetilde V)$. This leads to the morphism of inductive limits 
\begin{eqnarray*}
&&\lim_{\widetilde V \supset \widehat \sigma\!\!\!\sigma (U_\Delta)}\widetilde \E( \widetilde V) \otimes_{(\widehat \sigma\!\!\!\sigma^{-1} \OO_{\widetilde \Sigma}) (U_\Delta)} \OO_{\widetilde \Sigma_\Delta}(U_\Delta)\longrightarrow\\
&&\lim_{V \supset \sigma\!\!\!\sigma \widehat \sigma\!\!\!\sigma (U_\Delta)} (\widetilde \E (\sigma\!\!\!\sigma^{-1} V) \!\otimes_{\OO_{\Sigma}(\sigma\!\!\!\sigma^{-1} V)} \!\widetilde \L^m(\sigma\!\!\!\sigma^{-1} V)\!\otimes _{\OO_{\Sigma(V)}}\!\L^{-m}(V) \! \otimes_{(\widehat \sigma\!\!\!\sigma^{-1}\sigma\!\!\!\sigma^{-1}\! \OO_\Sigma)\! (U_\Delta)}\OO_{\widetilde \Sigma_\Delta}(U_\Delta)
\end{eqnarray*}
which yields  the morphism of presheaves 
\begin{eqnarray*}
(\U_\Delta &\mapsto & \lim_{\widetilde V \supset \widehat \sigma\!\!\!\sigma (U_\Delta)}\widetilde \E( \widetilde V) \otimes_{(\widehat \sigma\!\!\!\sigma^{-1} \OO_{\widetilde \Sigma}) (U_\Delta)} \OO_{\widetilde \Sigma_\Delta}(U_\Delta))\longrightarrow\\
(\U_\Delta &\mapsto & \lim_{V \!\supset \sigma\!\!\!\sigma \widehat \sigma\!\!\!\sigma \!(U_\Delta)}  \frac{\widetilde \E (\sigma\!\!\!\sigma^{-1} \!V) \otimes_{\OO_{\Sigma}(\sigma\!\!\!\sigma^{-1} V)} \widetilde \L^m(\sigma\!\!\!\sigma^{-1} \!V)\otimes _{\OO_{\Sigma(V)}}\!\L^{-m}(V) \otimes_{(\widehat \sigma\!\!\!\sigma^{-1}\sigma\!\!\!\sigma^{-1} \!\OO_\Sigma) (U_\Delta\!)}\!\OO_{\widetilde \Sigma_\Delta}\!(\!U_\Delta\!)}{\tors})
\end{eqnarray*}
and, consequently, the morphism of the corresponding associated locally free $\OO_{\widetilde \Sigma_\Delta}$-sheaves $\widehat \sigma\!\!\!\sigma^\ast \widetilde \E \to \widehat \sigma\!\!\!\sigma{'^\ast}\widetilde \E'$. Let $U_0 \subset \widetilde \Sigma_\Delta$ be a maximal open subscheme where both the morphisms $\sigma\!\!\!\sigma'$ and $\widehat \sigma\!\!\!\sigma'$ become isomorphisms.  The restriction of the morphism $\widehat \sigma\!\!\!\sigma^\ast \widetilde \E \to \widehat \sigma\!\!\!\sigma{'^\ast}\widetilde \E'$ to $U_0$ is an  identity isomorphism. Hence, when the sheaf morphism restricted to the maximal reduced subscheme $\widetilde \Sigma_{\Delta red}$, it remains the morphism of locally free sheaves but has a torsion subsheaf as a kernel. It is impossible. We have proven that both  sheaves $\widehat \sigma\!\!\!\sigma^\ast \widetilde \E$ and $\widehat \sigma\!\!\!\sigma{'^\ast}\widetilde \E'$ have equal Hilbert polynomials when restricted to fibers over the base $T$.

From this we conclude the $\OO_{\widetilde \Sigma_\Delta}$-isomorphism $\widehat \sigma\!\!\!\sigma^\ast \widetilde \E \cong \widehat \sigma\!\!\!\sigma{'^\ast}\widetilde \E'$.

\end{document}